\documentclass[reqno]{amsart}

\usepackage{amssymb,amscd}

\usepackage{enumerate}

\usepackage{multirow}

\usepackage[utf8]{inputenc}
\usepackage[table]{xcolor}

\usepackage{graphicx}
\usepackage{subfigure}

\usepackage{hyperref}
\hypersetup{colorlinks=false}

\theoremstyle{plain}

\newtheorem{thm}{Theorem}[section]
\newtheorem{lem}[thm]{Lemma}
\newtheorem{cor}[thm]{Corollary}
\newtheorem{prop}[thm]{Proposition}

\theoremstyle{definition}

\theoremstyle{remark}

\newtheorem{rem}{Remark}

\newcommand{\Z}{\mathbb{Z}}
\newcommand{\R}{\mathbb{R}}

\newcommand{\N}{\mathbb{N}}

\renewcommand{\S}{\mathbb{S}}

\renewcommand{\AA}{\mathcal{A}}
\newcommand{\BB}{\mathcal{B}}
\newcommand{\CC}{\mathcal{C}}

\newcommand{\EE}{\mathcal{E}}
\newcommand{\FF}{\mathcal{F}}

\newcommand{\HH}{\mathcal{H}}

\newcommand{\PP}{\mathcal{P}}
\newcommand{\QQ}{\mathcal{Q}}

\renewcommand{\SS}{\mathcal{S}}

\newcommand{\WW}{\mathcal{W}}

\newcommand{\supp}{\operatorname{supp}}
\newcommand{\codim}{\operatorname{codim}}

\newcommand{\id}{\operatorname{id}}

\newcommand{\sign}{\operatorname{sign}}

\newcommand{\Aut}{\operatorname{Aut}}

\newcommand{\depth}{\operatorname{depth}}

\newcommand{\diam}{\operatorname{diam}}
\newcommand{\vol}{\operatorname{vol}}

\newcommand{\Hess}{\operatorname{Hess}}

\newcommand{\Crit}{\operatorname{Crit}}

\newcommand{\Tr}{\operatorname{Tr}}

\newcommand{\Cinf}{C^\infty}
\newcommand{\D}{\Delta}
\newcommand{\sm}{\smallsetminus}
\newcommand{\ol}{\overline}

\newcommand{\bDelta}{\boldsymbol{\D}}

\newcommand{\bd}{\mathbf{d}}

\newcommand{\bD}{\mathbf{D}}

\newcommand{\fG}{\mathfrak{G}}
\newcommand{\fH}{\mathfrak{H}}

\newcommand{\fN}{\mathfrak{N}}

\newcommand{\sD}{\mathsf{D}}
\newcommand{\sR}{\mathsf{R}}

\newdimen\theight
\def\TeXref#1{%
             \leavevmode\vadjust{\setbox0=\hbox{{\tt
                     \quad\quad  {\small \textrm #1}}}%
             \theight=\ht0
             \advance\theight by \lineskip
             \kern -\theight \vbox to
             \theight{\rightline{\rlap{\box0}}%
             \vss}%
             }}%

\definecolor{darkgreen}{cmyk}{1,0,1,.2}
\definecolor{m}{rgb}{1,0.1,1}

\title[Witten's perturbation on strata]{Witten's perturbation on strata with general adapted metrics}

\author[J.A. \'Alvarez L\'opez]{Jes\'us A. \'Alvarez L\'opez}
\address{Departamento de Matem\'aticas\\
         Facultade de Matem\'aticas\\
         Universidade de Santiago de Compostela\\
         15782 Santiago de Compostela\\ Spain}
\email{jesus.alvarez@usc.es}
\thanks{The first author is partially supported by MICINN, Grant MTM2011-25656, and by MEC, Grant MTM2014-56950-P}

\author[M. Calaza]{Manuel Calaza}
\address{Laboratorio de Investigacion 2 and Rheumatology Unit, 
Hospital Clinico Universitario de Santiago, Santiago de Compostela, Spain}
\email{manuel.calaza@usc.es}

\author[C. Franco]{Carlos Franco}
\address{Departamento de Matem\'aticas\\
         Facultade de Matem\'aticas\\
         Universidade de Santiago de Compostela\\
         15782 Santiago de Compostela\\ Spain}
\email{carlosluis.franco@usc.es}
\thanks{The third author has received financial support from the Xunta de Galicia and the European Union (European Social Fund - ESF)}

\date{}

\subjclass{58A14, 32S60}

\keywords{Morse inequalities, ideal boundary condition, stratification, general adapted metric, Witten's perturbation}

\begin{document}

\maketitle

\begin{abstract}
  Let $M$ be a stratum of a compact stratified space $A$. It is equipped with a general adapted metric $g$, which is slightly more general than the adapted metrics of Nagase and Brasselet-Hector-Saralegi. In particular, $g$ has a general type, which is an extension of the type of an adapted metric. A restriction on this general type is assumed, and then $g$ is called good. We consider the maximum/minimum ideal boundary condition, $d_{\text{\rm max/min}}$, of the compactly supported de~Rham complex on $M$, in the sense of Br\"uning-Lesch. Let $H^*_{\text{\rm max/min}}(M)$ and $\Delta_{\text{\rm max/min}}$ denote the cohomology and Laplacian of $d_{\text{\rm max/min}}$. The first main theorem states that $\Delta_{\text{\rm max/min}}$ has a discrete spectrum satisfying a weak form of the Weyl's asymptotic formula. The second main theorem is a version of Morse inequalities using $H_{\text{\rm max/min}}^*(M)$ and what we call rel-Morse functions. An ingredient of the proofs of both theorems is a version for $d_{\text{\rm max/min}}$ of the Witten's perturbation of the de~Rham complex. Another ingredient is certain perturbation of the Dunkl harmonic oscillator previously studied by the authors using classical perturbation theory.  
  
  The condition on $g$ to be good is general enough in the following sense. Assume that $A$ is a stratified pseudomanifold, and consider its intersection homology $I^{\bar p}H_*(A)$ with perversity $\bar p$; in particular, the lower and upper middle perversities are denoted by $\bar m$ and $\bar n$, respectively. Then, for any perversity $\bar p\le\bar m$, there is an associated good adapted metric on $M$ satisfying the Nagase isomorphism $H^r_{\text{\rm max}}(M)\cong I^{\bar p}H_r(A)^*$ ($r\in\N$). If $M$ is oriented and $\bar p\ge\bar n$, we also get $H^r_{\text{\rm min}}(M)\cong I^{\bar p}H_r(A)$. Thus our version of the Morse inequalities can be described in terms of $I^{\bar p}H_*(A)$.
\end{abstract}

\tableofcontents

\section{Introduction} \label{s: intro}

\subsection{Ideal boundary conditions of the de~Rham complex}\label{ss: ibc}

The following usual notation is used for a densely defined linear operator $T$ in a Hilbert space. Its domain and range are denoted by $\sD(T)$ and $\sR(T)$. If $T$ is essentially self-adjoint, its closure is denoted by $\ol T$. If $T$ is self-adjoint, its \emph{smooth core} is $\sD^\infty(T):=\bigcap_{m=1}^\infty\sD(T^m)$, and its spectrum is denoted by $\sigma(T)$. 

A {\em Hilbert complex\/} $(\sD,\bd)$ is a differential complex of finite length given by a densely defined closed operator $\bd$ in a graded separable Hilbert space $\fH$ \cite{BruningLesch1992}. Then the operator $\bD=\bd+\bd^*$, with $\sD(\bD)=\sD(\bd)\cap\sD(\bd^*)$, is self-adjoint in $\fH$, and therefore the {\em Laplacian\/} $\bDelta=\bD^2=\bd\bd^*+\bd^*\bd$ is also self-adjoint. Moreover $\sD^\infty(\bDelta)$ is a subcomplex of $(\sD,\bd)$ with the same homology \cite[Theorem~2.12]{BruningLesch1992}; it may be also said that $\sD^\infty(\bDelta)$ is the {\em smooth core\/} of $\bd$.

The above notion is applied here in the following case. For a Riemannian manifold $M$, let $\Omega_0(M)$ be the space of compactly supported differential forms, and $L^2\Omega(M)$ the graded Hilbert space of square integrable differential forms. Let $d$ and $\delta$ be the de~Rham derivative and coderivative acting on $\Omega_0(M)$, and let $D=d+\delta$ and $\D=D^2=d\delta+\delta d$ (the Laplacian). Every Hilbert complex extension $\bd$ of $d$ in $L^2\Omega(M)$ is called an {\em ideal boundary condition\/} ({\em i.b.c.\/}) \cite{BruningLesch1992}, giving rise to self-adjoint extensions $\bD$ and $\bDelta$ of $D$ and $\D$ in $L^2\Omega(M)$. There exists a minimum/maximum i.b.c., $d_{\text{\rm min}}=\ol d$ and $d_{\text{\rm max}}=\delta^*$, inducing self-adjoint extensions $D_{\text{\rm max/min}}$ and $\D_{\text{\rm max/min}}$ of $D$ and $\D$. If $M$ is oriented, then $\D_{\text{\rm max}}$ corresponds to $\D_{\text{\rm min}}$ by the Hodge star operator. The corresponding cohomologies, $H_{\text{\rm max/min}}(M)$, are quasi-isometric invariants of $M$; for instance, $H_{\text{\rm max}}(M)$ is the usual $L^2$ cohomology $H_{(2)}(M)$ \cite{Cheeger1980}. They give rise to versions of Betti numbers and Euler characteristic, $\beta^r_{\text{\rm max/min}}=\beta^r_{\text{\rm max/min}}(M)$ and $\chi_{\text{\rm max/min}}=\chi_{\text{\rm max/min}}(M)$. These concepts can indeed be defined for arbitrary elliptic complexes  \cite{BruningLesch1992}. It is well known that $d_{\text{\rm min}}=d_{\text{\rm max}}$ if $M$ is complete. Thus considering an i.b.c.\ becomes interesting when $M$ is not complete. For example, if $M$ is the interior of a compact Riemannian manifold $N$ with with $\partial N\ne\emptyset$, then $d_{\text{\rm max/min}}$ is defined by taking absolute/relative boundary conditions. With more generality, we will assume that $M$ is a stratum of a compact stratified space $A$ \cite{Thom1969,Mather1970,Mather1973,Verona1984}, equipped with a generalization of the adapted metrics considered in \cite{Nagase1983,Nagase1986,BrasseletHectorSaralegi1992}. As we will see, we can assume $\ol M=A$ if desired (it can be said that $M$ is the {\em regular strum\/} in this case).

\subsection{Stratified spaces}\label{stratified spaces}

Roughly speaking, a ({\em Thom-Mather\/}) {\em stratified space\/} (or {\em stratification\/}) is a Hausdorff, locally compact and second countable space $A$ equipped with a partition into $C^\infty$ manifolds (the {\em strata\/}), satisfying certain conditions \cite{Thom1969,Mather1970}. In particular, an order relation on the family of strata is defined by declaring $X\le Y$ when $X\subset\ol Y$. With respect to this ordering, the maximum length of chains of strata less or equal than a stratum $X$ is called the {\em depth\/} of $X$. The supremum of the strata depth is called the {\em depth\/} of $A$, denoted $\depth A$. The precise definition and needed preliminaries were collected in \cite[Section~3]{AlvCalaza2017}, where we have mainly followed \cite{Verona1984}. Instead of recalling it, let us describe how the strata of $A$ fit together, describing also {\em morphisms/isomorphisms\/} of stratifications, and, in particular, the group of {\em automorphisms\/}, $\Aut(A)$. We proceed by induction on its depth. If $\depth A=0$, then $A$ is just a $C^\infty$ manifold, and $\Aut(A)$ consists of its diffeomorphisms. 

Now, given any $k\in\Z_+$, assume that any stratified space $L$ is described if $\depth L<k$, as well as $\Aut(L)$. If $L$ is compact, the {\em cone\/} with {\em link\/} $L$ is $c(L)=(L\times[0,\infty))/(L\times\{0\})$, whose {\em vertex\/} is the point $*=L\times\{0\}\in c(L)$. Let $L'$ be another compact stratification of depth $<k$, and $\phi :L\to L'$ a morphism. Then let $c(\phi):c(L)\to c(L')$ be the map induced by $\phi\times\id:L\times[0,\infty)\to L'\times[0,\infty)$; in particular, we get the group $c(\Aut(L))=\{\,c(\phi)\mid\phi\in\Aut(L)\,\}$. It is also declared that $c(\emptyset)=\{*\}$, for the empty stratification, and $c(\emptyset)=\id$, for the empty map. The cone $c(L)$ is used as a model stratified space of depth $k$ if $L$ is of depth $k-1$, whose strata are $\{*\}$ and the manifolds $Y\times\R_+$ for strata $Y$ of $L$. The second factor projection $L\times[0,\infty)\to [0,\infty)$ defines a $c(\Aut(L))$-invariant function $\rho:c(L)\to[0,\infty)$, called the {\em radial function\/}. The restrictions of $\rho$ to the strata are $C^\infty$. A {\em conic bundle\/} is a fiber bundle $T$ over a manifold $X$ with typical fiber $c(L)$ and structural group $c(\Aut(L))$. Then $\rho$ induces a {\em radial function\/} on $T$, also denoted by $\rho$, and the vertex of $c(L)$ defines the {\em vertex section\/} of $T$, whose image is identified with $X$. Moreover the stratified structure defined on $c(L)$ can be used to define a stratified structure on $T$, where $X$ becomes the {\em vertex stratum\/}. 

For any stratification $A$ of depth $k$, every stratum $X$ has an open neighborhood (a {\em tube\/} representative) that is isomorphic to an open neighborhood of $X$ in some conic bundle $T_X$ over $X$ (with the obvious restrictions of stratified structures to open subsets). The typical fiber of $T_X$ is of the form $c(L_X)$ for some compact stratification $L_X$ (the {\em link\/} of $X$) with $\depth L_X<\depth A$. The vertex and radial function of $c(L_X)$ are denoted by $*_X$ and $\rho_X$. Two such neighborhoods of $X$ represent the same {\em tube\/} if their structure is equal on some smaller neighborhood of $X$. Note that $X$ is open in $A$ if and only if $L_X=\emptyset$. 

Finally, a {\em morphism\/} between two stratifications is a continuous map sending every stratum to another stratum, whose restrictions to the strata are $C^\infty$, and whose restrictions to small enough tube representatives are restrictions of conic bundle morphisms. Then {\em isomorphisms\/} and {\em automorphisms\/} of stratifications have the obvious meaning. This completes the description because the depth is locally finite by the local compactness.

The (topological) dimension of a stratification $A$ equals the supremum of the dimensions of its strata. It may be infinite, but it is locally finite. The {\em codimension\/} of every stratum $X$ is $\dim A-\dim X$. Our main results will assume that the stratification is compact, but non-compact stratifications will be also used in the proofs. In any case, we will only consider stratifications of finite dimension. If the above description of $A$ is modified by requiring that, at every inductive step, only stratifications with no strata of codimension $1$ are used, then $A$ is called a {\em stratified pseudomanifold\/}.

A locally closed subset $B\subset A$ is called a {\em substratification\/} of $A$ if the restrictions of the strata and tubes of $A$ to $B$ define a stratified structure on $B$. For instance, $A$ can be restricted to any open subset, to any locally closed union of strata, and to the closure of any stratum. If moreover there are tube representatives of $A$ whose restrictions to $B$ have the same fibers over points of $B$, then $B$ is called {\em saturated\/}.

Let $x$ be a point of a stratum $X$ of dimension $m_X$ in a stratification $A$. A local trivialization of $T_X$ on some open neighborhood $U$ of $x$ defines a {\em chart\/} $O\equiv O'$ of $A$ for some open $O'\subset \R^{m_X}\times c(L_X)$. We can assume $O'=U'\times c_\epsilon(L_X)$, where $U'$ is some open neighborhood of $0$ in $\R^{m_X}$ and $c_\epsilon(L_X)$ is the subset of $c(L_X)$ defined by the condition $\rho_X<\epsilon$, for some $\epsilon>0$. This chart is said to be {\em centered\/} at $x$ if $x\equiv(0,*_X)\in O'$. The corresponding concept of {\em atlas\/} has the obvious meaning. These concepts can be generalized as follows. Any finite product of stratifications has a non-canonical stratified structure \cite[Section~3.1.2]{AlvCalaza2017}; in particular, any finite product of cones is isomorphic to a cone \cite[Lemma~3.8]{AlvCalaza2017}. Moreover $\Aut(P)\times\Aut(Q)$ is canonically injected in $\Aut(P\times Q)$ for stratifications $P$ and $Q$. Thus it makes sense to consider a decomposition $c(L_X)\cong\prod_{i=1}^{a_X}c(L_{X,i})$ ($a_X\in\N$), for compact stratifications $L_{X,i}$. The vertex and radial function of every $c(L_{X,i})$ are denoted by $*_{X,i}$ and $\rho_{X,i}$. Then we can also consider {\em general tube\/} representatives given by bundles $T_X$ with typical fibers $\prod_{i=1}^{a_X}c(L_{X,i})$ and structural groups $\prod_{i=1}^{a_X}c(\Aut(L_{X,i}))$. This gives rise to a {\em general chart\/} $O\equiv O'$ around $x$ for some open $O'\subset \R^{m_X}\times\prod_{i=1}^{a_X}c(L_{X,i})$, which is {\em centered\/} at $x$ if $x\equiv(0,*_{X,1},\dots,*_{X,a_X})\in O'$. As above, we can assume $O'=U'\times\prod_{i=1}^{a_X}c_\epsilon(L_{X,i})$ for some $\epsilon>0$. Let $\rho_{X,0}$ denote the norm function on $\R^{m_X}$. The function $\rho=(\rho_{X,0}^2+\dots+\rho_{X,a_X}^2)^{1/2}$ is called the {\em radial\/} function of $\R^{m_X}\times\prod_{i=1}^{a_X}c(L_{X,i})$, even though, when $m_X=0$, $\rho$ is not the radial function of any cone structure on $\prod_{i=1}^{a_X}c(L_{X,i})$ \cite[Example~3.6 and Proof of Lemma~3.8]{AlvCalaza2017}. A collection of general charts covering $A$ is called a {\em general atlas\/}.

We can suppose that the strata of $A$ are connected \cite[Remark~1~(v)]{AlvCalaza2017}. Fix a stratum $M$ of dimension $n$ in $A$. Since the stratified structure of $A$ can be restricted to $\ol M$ \cite[Section~3.1.1]{AlvCalaza2017}, we can also assume without loss of generality that $\ol M=A$ (any other stratum is $<M$); in particular, $\depth A=\depth M$ and $\dim A=n$. With the above notation, for a chart $O\equiv O'$ centered at $x$, we get $M\cap O\equiv M'\cap O'$, where $M'=\R^{m_X}\times N\times\R_+$ for some dense stratum $N$ on $L_X$. In the case of a general chart $O\equiv O'$ centered at $x$, we have $M\cap O\equiv M'\cap O'$ for $M'=\R^{m_X}\times\prod_{i=1}^{a_X}(N_i\times\R_+)$, where every $N_i$ is some dense stratum of $L_{X,i}$. We will use the notation $k_{X,i}=\dim N_i+1$.

\subsection{General adapted metrics}\label{ss: general adapted metrics}

A {\em general adapted metric\/} $g$ on $M$ is defined by induction on the depth of $M$. It is any (Riemannian) metric if $\depth M=0$. Now, assume that $\depth M>0$ and general adapted metrics are defined for lower depth. Given any general chart $O\equiv O'$ as above, take any general adapted metric $\tilde g_i$ on every $N_i$ ($\depth N_i<\depth M$), and let $g_i=\rho_{X,i}^{2u_{X,i}}\tilde g_i+(d\rho_{X,i})^2$ on $N_i\times\R_+$ for some $u_{X,i}>0$. Let also $g_0$ be the Euclidean metric on $\R^{m_X}$. Then $g$ is a {\em general adapted metric\/} if, via any such general chart, $g|_O$ is quasi-isometric to $(\sum_{i=0}^{a_X}g_i)|_{O'}$. In this case, the mapping $X\mapsto u_X:=(u_{X,1},\dots,u_{X,a_X})\in\R_+^{a_X}$ ($X<M$) is called the {\em general type\/} of $g$. Such a general chart is called {\em compatible\/} with $g$, or with its general type.

Let us point out that a general metric does not completely determine its general type. For instance, suppose $u_{X,i}=u_{X,j}=1$ for indices $i\ne j$. Write $c(L_{X,i})\times c(L_{X,j})\equiv c(L)$, with radial function $\rho$, for some stratification $L$. Then $N_i\times\R_+\times N_j\times\R_+\equiv N\times\R_+$ for some dense stratum $N$ of $L$. Moreover there is a general adapted metric $\tilde g$ on $N$ such that $g_i+g_j$ is quasi-isometric to $\rho^2\tilde g+(d\rho)^2$ via the above identity. Therefore we can omit $u_{X,i}$ or $u_{X,j}$ in $u_X$, obtaining a different type of $g$. This cannot be done if $u_{X,i}=u_{X,j}\ne1$ (Proposition~\ref{p: phi is a quasi-isometry if and only if u=1}).

If the above definition of general adapted metric is modified by requiring that, at every inductive step, the general type satisfies $u_{X,i}\le1$ for all $X<M$ and $i=1,\dots,a_X$, then the general adapted metric is called {\em good\/} for the scope of this paper. On the other hand, if the definition is modified by requiring at every inductive step that $a_X=1$ and $u_X$ depends only on $k:=k_{X,1}=\codim X$ for all $X<M$, then we get the {\em adapted metrics\/} considered in \cite{Nagase1983,Nagase1986,BrasseletHectorSaralegi1992}. In this case, the general charts compatible with the general type are indeed charts. Writing $u_k=u_X\equiv u_{X,1}\in\R_+$, the condition on an adapted metric $g$ to be good becomes $u_k\le1$ for all $k$, at every inductive step of its definition. In \cite{Nagase1983,Nagase1986,BrasseletHectorSaralegi1992}, it is assumed that $A$ is a stratified pseudomanifold, and then $\hat u=(u_2,\dots,u_n)$ stands for the {\em type\/} of $g$. This $\hat u$ is determined by $g$. In particular, if the definition is modified by taking $u_k=1$ for all $k$ at every inductive step, we get the adapted metrics of {\em conic type\/} considered in \cite{Cheeger1980,Cheeger1983,CheegerGoreskyMacPherson1982}. Be alerted about the three slightly different terms used for the scope of this paper: adapted metrics of conic type, adapted metrics and general adapted metrics. The class of (good) general adapted metrics is preserved by products, as well as the class of adapted metrics of conic type, but the class of adapted metrics does not have this property. The existence of general adapted metrics with any possible general type can be shown like in the case of adapted metrics \cite[Lemma~4.3]{Nagase1983}, \cite[Appendix]{BrasseletHectorSaralegi1992}.

Like in \cite{AlvCalaza2017}, the term ``relative(ly)'' (or simply ``rel-'') usually means that some condition is required in the intersection of $M$ with small neighborhoods of the points in $\ol M$, or that some concept can be described using those intersections. 

Let $M$ be equipped with a general adapted metric $g$, with a general type $X\mapsto u_X$ as above. The {\em rel-local metric completion\/} $\widehat M$ of $M$ consists of the points in the metric completion represented by Cauchy sequences that converge in $\ol M$ ($\widehat M$ is the metric completion of $M$ if $\ol M$ is compact). Figure~\ref{fig: widehat M} illustrates this concept. The limits of Cauchy sequences define a continuous map $\lim:\widehat M\to\ol M$. The following properties can be proved like in the case of conic metrics \cite[Proposition~3.20~(i),(ii)]{AlvCalaza2017}. $\widehat M$ has a unique stratified structure with connected strata so that $\lim:\widehat M\to\ol M$ is a morphism whose restrictions to the strata are local diffeomorphisms. Moreover $g$ is also a general adapted metric with respect to $\widehat M$.

\begin{figure}[h]
\centering
\subfigure[$\ol M$]{
\includegraphics[width=3.5cm]{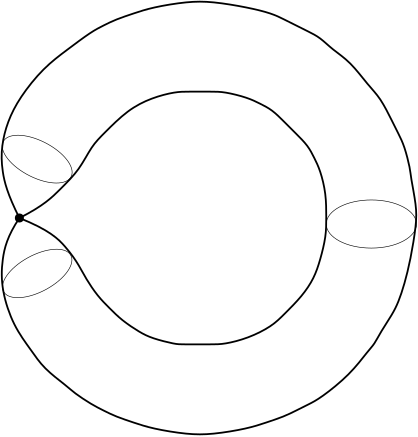}}
\qquad
\subfigure[$\widehat M$]{
\includegraphics[width=3.5cm]{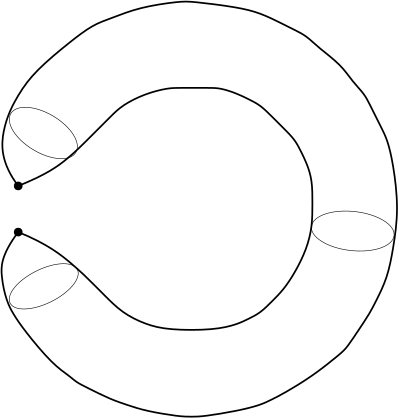}}
\caption{The stratified space $\widehat M$.}
\label{fig: widehat M}
\end{figure}

\subsection{Relatively Morse functions}\label{ss: rel-Morse}

A smooth function $f$ on $M$ is called {\em rel-admissible\/} when the functions $f$, $|df|$ and $|\Hess f|$ are rel-bounded. In this case, $f$ may not have any continuous extension to $\ol{M}$, but it has a continuous extension to $\widehat{M}$. So it makes sense to say that $x\in\widehat{M}$ is a {\em rel-critical point\/} of $f$ when $\liminf|df(y)|=0$ as $y\to x$ in $\widehat{M}$ with $y\in M$. The set of rel-critical points of $f$ is denoted by $\Crit_{\text{\rm rel}}(f)$. It is said that $f$ is a {\em rel-Morse function\/} if it is rel-admissible and has the following description around every $x\in\Crit_{\text{\rm rel}}(f)$:
	\begin{itemize}
	
		\item there is a general chart $O\equiv O'$ of $\widehat{M}$, centered at $x$ and compatible with $g$, such that $M\cap O\equiv M'\cap O'$ for $M'=\R^{m_X}\times\prod_{i=1}^{a_X}(N_i\times\R_+)$, where $X$ is the stratum of $\widehat M$ containing $x$; and 
		
		\item $f|_{M\cap O}\equiv f(x)+\frac{1}{2}(\rho_+^2-\rho_-^2)|_{M'\cap O'}$, where $\rho_\pm$ is the radial function of $\R^{m_\pm}\times \prod_{i\in I_\pm}c(L_{X,i})$ for some expression $m_X=m_++m_-$ ($m_\pm\in\N$) and some partition of $\{1,\dots,a_X\}$ into sets $I_\pm$. 
	\end{itemize}
This local condition is used instead of requiring that $\Hess f$ is ``rel-non-degenerate'' at the rel-critical points because a ``rel-Morse lemma'' is missing. Moreover, for every $r\in\{0,\dots,n\}$, let
	\begin{equation}\label{nu_x,max/min^r}
        		\nu_{x,\text{\rm max/min}}^r=\sum_{(r_1,\dots,r_{a_X})}
		\prod_{i=1}^{a_X}\beta_{\text{\rm max/min}}^{r_i}(N_i)\;,
      	\end{equation}
    where $(r_1,\dots,r_{a_X})$ runs in the subset of $\N^{a_X}$ determined by
	\begin{equation}\label{r = m_- + sum_i=1^a_X r_i + |I_-|}
		\left.
			\begin{array}{c}
				r=m_-+\sum_{i=1}^{a_X}r_i+|I_-|\;,\\[4pt]
				\begin{array}{ll}
					\left.
						\begin{array}{ll}
							r_i<\frac{k_{X,i}-1}{2}+\frac{1}{2u_{X,i}} & \text{if $i\in I_+$}\\[4pt]
							r_i\ge\frac{k_{X,i}-1}{2}+\frac{1}{2u_{X,i}} & \text{if $i\in I_-$}
						\end{array}
					\right\}\quad&\text{for $\nu_{x,\text{\rm max}}^r$}\;,\\[4pt]
					\left.
						\begin{array}{ll}
							r_i\le\frac{k_{X,i}-1}{2}-\frac{1}{2u_{X,i}} & \text{if $i\in I_+$}\\[4pt]
							r_i>\frac{k_{X,i}-1}{2}-\frac{1}{2u_{X,i}} & \text{if $i\in I_-$}
						\end{array}
					\right\}\quad&\text{for $\nu_{x,\text{\rm min}}^r$}\;.
				\end{array}
			\end{array}
		\right\}
	\end{equation}
When $a_X=0$ in~\eqref{nu_x,max/min^r}, the singleton $\N^0$ consists of the empty sequence, obtaining\footnote{Kronecker's delta symbol is used.} $\nu_{x,\text{\rm max/min}}^r=\delta_{r,m_-}$ with the convention that the value of empty products is $1$. Finally, let $\nu_{\text{\rm max/min}}^r=\sum_x\nu_{x,\text{\rm max/min}}^r$ with $x$ running in $\Crit_{\text{\rm rel}}(f)$. The notation $\nu_{x,\text{\rm max/min}}^r(f)$ and $\nu_{\text{\rm max/min}}^r(f)$ may be used if necessary. The existence of rel-Morse functions for general adapted metrics holds like in the case of adapted metrics \cite[Proposition~4.9]{AlvCalaza2017}.

\subsection{Main theorems}\label{ss: main thms}

The following is our first main theorem, where property~\eqref{i: liminf_k lambda_max/min,k k^-theta > 0 for some theta>0} is a weak version of the Weyl's asymptotic formula.

\begin{thm}\label{t: spectrum of Delta_max/min}
  The following properties hold on any stratum of a compact stratification with a good general adapted metric:
    \begin{enumerate}[{\rm(}i\/{\rm)}]
  
      \item\label{i: Delta_max/min has a discrete spectrum} $\D_{\text{\rm max/min}}$ has a discrete spectrum, $0\le\lambda_{\text{\rm max/min},0}\le\lambda_{\text{\rm max/min},1}\le\cdots$, where every eigenvalue is repeated according to its multiplicity.
    
      \item\label{i: liminf_k lambda_max/min,k k^-theta > 0 for some theta>0} $\liminf_k\lambda_{\text{\rm max/min},k}\,k^{-\theta}>0$ for some $\theta>0$.
    
    \end{enumerate}
\end{thm}

Our second main result is the following version of Morse inequalities for rel-Morse functions.

\begin{thm}\label{t: Morse inequalities}
  	For any rel-Morse function on a stratum of dimension $n$ of a compact stratification, equipped with a good general adapted metric, we have
		\begin{align*}
			\sum_{r=0}^k(-1)^{k-r}\,\beta_{\text{\rm max/min}}^r
			&\le\sum_{r=0}^k(-1)^{k-r}\,\nu_{\text{\rm max/min}}^r\quad(0\le k<n)\;,\\
      			\chi_{\text{\rm max/min}}&=\sum_{r=0}^n(-1)^r\,\nu_{\text{\rm max/min}}^r\;.
    		\end{align*}
\end{thm}

In the case of adapted metrics of conic type, Theorem~\ref{t: spectrum of Delta_max/min}~\eqref{i: Delta_max/min has a discrete spectrum} is essentially due to Cheeger \cite{Cheeger1980,Cheeger1983} (see also \cite{AlbinLeichtnamMazzeoPiazza:Hodge,AlbinLeichtnamMazzeoPiazza2012,AlvCalaza2017}), Theorem~\ref{t: spectrum of Delta_max/min}--\eqref{i: liminf_k lambda_max/min,k k^-theta > 0 for some theta>0} was proved by the authors \cite{AlvCalaza2017}, and Theorem~\ref{t: Morse inequalities} was proved by the authors \cite{AlvCalaza2017} and Ludwig \cite{Ludwig2017} (with more restrictive conditions but stronger consequences). Other developments of elliptic theory on strata were made in \cite{BruningLesch1993,Lesch1997,HunsickerMazzeo2005,Schulze2009,DebordLescureNistor2009,AlbinLeichtnamMazzeoPiazza2012,AlbinLeichtnamMazzeoPiazza:Hodge}, all of them using adapted metrics of conic type. The main novelty of our paper is the extension of the elliptic theory on strata to the wider class of good general adapted metrics, including good adapted metrics.

\subsection{Applications to intersection homology}\label{ss: applications}

Consider now the case where $A$ is a stratified pseudomanifold, and therefore $M$ is its regular stratum. Let $I^{\bar p}H_*(A)$ denote its intersection homology with perversity $\bar p$ \cite{GoreskyMacPherson1980:Intersection,GoreskyMacPherson1983:IntersectionII}, taking real coefficients. Let $\beta_r^{\bar p}=\beta_r^{\bar p}(A)$ and $\chi^{\bar p}=\chi^{\bar p}(A)$ denote the versions of Betti numbers and Euler characteristic for $I^{\bar p}H_*(A)$. Every perversity can be considered as a sequence $\bar p=(p_2,p_3,\dots)$ in $\N$ satisfying $p_2=0$ and $p_k\le p_{k+1}\le p_k+1$. For example, the zero perversity is $\bar0=(0,0,\dots)$, the top perversity is $\bar t=(0,1,2,\dots)$ ($t_k=k-2$), the lower middle perversity is $\bar m=(0,0,1,1,2,2,3,\dots)$ ($m_k=\lfloor\frac{k}{2}\rfloor-1$), and the upper middle perversity is $\bar n=(0,1,1,2,2,3,3,\dots)$ ($n_k=\lceil\frac{k}{2}\rceil-1$). Recall also that two perversities $\bar p$ and $\bar q$ are called complementary if $\bar p+\bar q=\bar t$. Write $\bar p\le\bar q$ if $p_k\le q_k$ for all $k$. Let $g$ be an adapted metric on $M$ of type $\hat u=(u_2,\dots,u_n)$. If $\hat u$ is {\em associated\/} with a perversity $\bar p\le\bar m$ in the sense
	\begin{equation}\label{bar p associated with hat u}
		\left.
			\begin{array}{cc}
				\textstyle{\frac{1}{k-1-2p_k}}\le u_k<\textstyle{\frac{1}{k-3-2p_k}}
				&\quad\text{if}\quad2p_k\le k-3\;,\\[4pt]
				\ \,1\le u_k<\infty&\quad\text{if}\quad2p_k=k-2\;,
			\end{array}
		\right\}
	\end{equation}
then $H_{(2)}^r(M)\cong I^{\bar p}H_r(A)^*$ \cite{Nagase1983,Nagase1986,BrasseletHectorSaralegi1992}, and therefore $\beta^{\bar p}_r=\beta^r_{\text{\rm max}}$. In particular, $H_{(2)}^r(M)\cong I^{\bar m}H_r(A)^*$ if $g$ is an adapted metric of conic type \cite{CheegerGoreskyMacPherson1982}. Thus the incompatibility of adapted metrics with products is related to the subtleties of the versions of the K\"unneth theorem for intersection homology \cite{CohenGoreskyJi1992,Friedman2009}. For instance, the isomorphism $I^{\bar p}H_*(P\times Q)\cong I^{\bar p}H_*(P)\otimes I^{\bar p}H_*(Q)$, for arbitrary pseudomanifolds $P$ and $Q$, only holds with some special perversities $\bar p$, including $\bar p=\bar m$. According to~\eqref{bar p associated with hat u}, there exist good adapted metrics on $M$ whose type is associated with any given perversity $\le\bar m$. 

In~\eqref{bar p associated with hat u}, only the choices $2p_k=k-2,k-4,\dots$ are possible if $k$ is even, and only the choices $2p_k=k-3,k-5,\dots$ are possible if $k$ is odd. Thus, for every $k$,~\eqref{bar p associated with hat u} establishes a bijection between the possibilities for $p_k$ and a partition of $[\frac{1}{k-1},\infty)$ into semi-open intervals, where $u_k$ is taken.

Let $f$ be a rel-Morse function on $M$, let $x\in\Crit_{\text{\rm rel}}(f)$, let $X$ be the stratum of $\widehat M$ containing $x$, and let $k=\codim X$. With the above notation for a chart $O\equiv O'$ of $\widehat M$ centered at $x$, there is an adapted metric $\tilde g$ on $N$ so that, via the chart, $g|_O$ is quasi-isometric to the restriction of $g_0+\rho_X^{2u_k}\tilde g+(d\rho_X)^2$ to $M'\cap O'$. Then the type of $\tilde g$ is also associated with $\bar p$. Moreover there is some expression, $m_X=m_++m_-$ ($m_\pm\in\N$), and some decomposition, $c(L_X)\equiv c(L_+)\times c(L_-)$, so that $M'\equiv\R^{m_+}\times N_+\times\R_+\times\R^{m_-}\times N_-\times\R_+$ for dense strata $N_\pm$ of $L_\pm$, and $f|_O\equiv f(x)+\frac{1}{2}(\rho_+^2-\rho_-^2)|_{O'}$, where $\rho_\pm$ is the radial function of $\R^{m_\pm}\times c(L_\pm)$. Let $k_\pm=\dim N_\pm+1$; thus $k=k_++k_-$. Here, some of the stratifications $L_\pm$ may be empty; in fact, $L_+\ne\emptyset\ne L_-$ only can happen if $u_k=1$ (Section~\ref{ss: general adapted metrics}). From~\eqref{nu_x,max/min^r} and~\eqref{r = m_- + sum_i=1^a_X r_i + |I_-|}, it follows that the numbers $\nu_{x,\text{\rm max}}^r$ are independent of the choice of $\hat u$ associated with $\bar p$, and therefore the notation $\nu^{\bar p}_{x,r}=\nu^{\bar p}_{x,r}(f)$ will be used. Precisely, they have the following expressions:
	\begin{itemize}
	
		\item If $L_+\ne\emptyset\ne L_-$ (only if $u_k=1$), then
			\[
				\nu^{\bar p}_{x,r}=\sum_{(r_+,r_-)}\beta^{\bar p}_{r_+}(L_+)\,\beta^{\bar p}_{r_-}(L_-)\;,
			\]
		where $(r_+,r_-)$ runs in the subset of $\N^2$ determined by the conditions
			\[
				r=m_-+r_++r_-+1\;,\quad r_+<\textstyle{\frac{k_+}{2}}\;,\quad r_-\ge\textstyle{\frac{k_-}{2}}\;.
			\]
			
		\item If $L_X=L_+\ne\emptyset$ ($L_-=\emptyset$), then
			\[
				\nu^{\bar p}_{x,r}=\sum_{r_+}\beta^{\bar p}_{r_+}(L_X)\;,
			\]
		where $r_+$ runs in the subset of $\N$ determined by the conditions
			\[
				r=m_-+r_+\;,\quad 
				r_+< 
					\begin{cases}
						k-1-p_k & \text{if $u_k<1$}\\
						\frac{k}{2} & \text{if $u_k=1$}\;.
					\end{cases}
			\]
			
		\item If $L_X=L_-\ne\emptyset$ ($L_+=\emptyset$), then
			\[
				\nu^{\bar p}_{x,r}=\sum_{r_-}\beta^{\bar p}_{r_-}(L_X)\;,
			\]
		where $r_-$ runs in the subset of $\N$ determined by the conditions
			\[
				r=m_-+r_-+1\;,\quad 
				r_-\ge 
					\begin{cases}
						k-1-p_k& \text{if $u_k<1$}\\
						\frac{k}{2} & \text{if $u_k=1$}\;.
					\end{cases}
			\]
			
		\item If $L_X=\emptyset$, then $\nu^{\bar p}_{x,r}=\delta_{r,m_-}$.
	
	\end{itemize}
Finally, let $\nu^{\bar p}_r=\nu^{\bar p}_r(f)=\sum_x\nu^{\bar p}_{x,r}$ ($x\in\Crit_{\text{\rm rel}}(f)$), which equals $\nu_{\text{\rm max}}^r$.

Suppose now that $A$ is oriented ($M$ is oriented) and compact. We have $\beta_{\text{\rm min}}^r=\beta_{\text{\rm max}}^{n-r}$ for all $r$ because $\D_{\text{\rm min}}$ corresponds to $\D_{\text{\rm max}}$ by the Hodge star operator. On the other hand, for any perversity $\bar q\ge\bar n$, if $\bar p\le\bar m$ is complementary of $\bar q$, then $I^{\bar q}H_r(A)\cong I^{\bar p}H_{n-r}(A)^*$  \cite{GoreskyMacPherson1980:Intersection,GoreskyMacPherson1983:IntersectionII}, and therefore $\beta^{\bar q}_r=\beta^{\bar p}_{n-r}$, obtaining $\beta^{\bar q}_r=\beta_{\text{\rm min}}^r$. As before, it follows from~\eqref{nu_x,max/min^r} and~\eqref{r = m_- + sum_i=1^a_X r_i + |I_-|} that the numbers $\nu_{x,\text{\rm min}}^r$ are independent of the choice of $\hat u$ associated with $\bar p$. Precisely, with the notation $\nu^{\bar q}_{x,r}=\nu^{\bar q}_{x,r}(f)=\nu_{x,\text{\rm min}}^r$, they have the following expressions:
	\begin{itemize}
	
		\item If $L_+\ne\emptyset\ne L_-$ (only if $u_k=1$), then
			\[
				\nu^{\bar q}_{x,r}=\sum_{(r_+,r_-)}\beta^{\bar q}_{r_+}(L_+)\,\beta^{\bar q}_{r_-}(L_-)\;,
			\]
		where $(r_+,r_-)$ runs in the subset of $\N^2$ determined by the conditions
			\[
				r=m_-+r_++r_-+1\;,\quad r_+\le\textstyle{\frac{k_+}{2}}-1\;,\quad r_->\textstyle{\frac{k_-}{2}}-1\;.
			\]
			
		\item If $L_X=L_+\ne\emptyset$ ($L_-=\emptyset$), then
			\[
				\nu^{\bar q}_{x,r}=\sum_{r_+}\beta^{\bar q}_{r_+}(L_X)\;,
			\]
		where $r_+$ runs in the subset of $\N$ determined by the conditions
			\[
				r=m_-+r_+\;,\quad 
				r_+\le
					\begin{cases}
						k-2-q_k & \text{if $u_k<1$}\\
						\frac{k}{2}-1 & \text{if $u_k=1$}\;.
					\end{cases}
			\]
			
		\item If $L_X=L_-\ne\emptyset$ ($L_+=\emptyset$), then
			\[
				\nu^{\bar q}_{x,r}=\sum_{r_-}\beta^{\bar q}_{r_-}(L_X)\;,
			\]
		where $r_-$ runs in the subset of $\N$ determined by the conditions
			\[
				r=m_-+r_-+1\;,\quad 
				r_-> 
					\begin{cases}
						k-2-q_k& \text{if $u_k<1$}\\
						\frac{k}{2}-1 & \text{if $u_k=1$}\;.
					\end{cases}
			\]
			
		\item If $L_X=\emptyset$, then $\nu^{\bar q}_{x,r}=\delta_{r,m_-}$.
	
	\end{itemize}
Like $\nu^{\bar p}_r$, we also define $\nu^{\bar q}_r=\nu^{\bar q}_r(f)=\sum_x\nu^{\bar q}_{x,r}$ ($x\in\Crit_{\text{\rm rel}}(f)$), which equals $\nu_{\text{\rm min}}^r$.

Theorem~\ref{t: Morse inequalities} has the following direct consequence.

\begin{cor}\label{c: Morse inequalities}
	Let $A$ be a compact pseudomanifold of dimension $n$, let $M$ be its regular stratum, and let $\bar p$ be a perversity. If $\bar p\le\bar m$, or if $A$ is oriented and $\bar p\ge\bar n$, then, for any rel-Morse function on $M$ {\rm(}with respect to any good adapted metric\/{\rm)}, we have
		\begin{align*}
			\sum_{r=0}^k(-1)^{k-r}\,\beta^{\bar p}_r
			&\le\sum_{r=0}^k(-1)^{k-r}\,\nu^{\bar p}_r\quad(0\le k<n)\;,\\
      			\chi^{\bar p}&=\sum_{r=0}^n(-1)^r\,\nu^{\bar p}_r\;.
    		\end{align*}
\end{cor}

Stratified Morse theory was introduced by Goresky and MacPherson \cite{GoreskyMacPherson1988:stratifiedMorse}, and has a great wealth of applications. In particular, Goresky and MacPherson have proved Morse inequalities on complex analytic varieties with Whitney stratifications, involving the intersection homology with perversity $\bar m$ \cite[Chapter~6, Section~6.12]{GoreskyMacPherson1988:stratifiedMorse}. Ludwig also gave an analytic interpretation of Morse theory in the spirit of Goresky and MacPherson for conformally conic manifolds \cite{Ludwig2010,Ludwig2011b,Ludwig2011a,Ludwig2013}.  Our version of Morse functions, critical points and associated numbers is different from those used in \cite{GoreskyMacPherson1988:stratifiedMorse}, even in the case of perversity $\bar m$. To the authors' knowledge, Corollary~\ref{c: Morse inequalities} is the first version of Morse inequalities for intersection homology with perversity $\ne\bar m$.

\subsection{Ideas of the proofs}\label{ss: ideas}

In the proofs of Theorems~\ref{t: spectrum of Delta_max/min} and~\ref{t: Morse inequalities}, several steps are like in the case of adapted metrics of conic type \cite{AlvCalaza2017}. Only brief indications of those steps are given in this paper, whereas the parts with new ideas are explained with detail. We adapt the well-known analytic method of Witten \cite{Witten1982}; specially, as described in \cite[Chapters~9 and~14]{Roe1998}. Thus, given a rel-Morse function $f$ on $M$, we consider the Witten's perturbation $d_s=e^{-sf}de^{sf}=d+s\,df\wedge$ on $\Omega_0(M)$ ($s>0$). Let $d_{s,\text{\rm max/min}}$ denote its maximum/minimum i.b.c., with corresponding Laplacian $\D_{s,\text{\rm max/min}}$. Since $\D_{s,\text{\rm max/min}}-\D_{\text{\rm max/min}}$ is bounded, it is enough to prove the properties of Theorem~\ref{t: spectrum of Delta_max/min} for $\D_{s,\text{\rm max/min}}$. Moreover, using a globalization procedure \cite[Propositions~14.2 and~14.3]{AlvCalaza2017} and a version of the K\"unneth theorem \cite[Corollary~2.15]{BruningLesch1992}, \cite[Lemma~5.1]{AlvCalaza2017}, it is enough to consider the case of a stratum $M=N\times\R_+$ of a cone $c(L)$ (a non-compact stratification), with a good general adapted metric of the form $g=\rho^{2u}\tilde g+d\rho^2$, and the rel-Morse function $\pm\frac{1}{2}\rho^2$, where $\rho$ is the radial function and $L$ a compact stratification of smaller depth. A tilde is added to the notation of concepts considered for $N$. By induction on the depth, it is assumed that $\widetilde\D_{\text{\rm max/min}}$ satisfies the properties of Theorem~\ref{t: spectrum of Delta_max/min}. Then its eigenforms are used like in \cite{AlvCalaza2017} to split $d_{s,\text{\rm max/min}}$ into a direct sum of Hilbert complexes of length one and two, which can be described as the maximum/minimum i.b.c.\ of certain elliptic complexes on $\R_+$. The elliptic complexes of length one are of the same kind as in \cite{AlvCalaza2017}, so that the Laplacian of their maximum/minimum i.b.c.\ is induced by the Dunkl harmonic oscillator on $\R$ \cite{AlvCalaza2014}, whose spectrum is well known. However, the Laplacian of the elliptic complexes of length two is a perturbation of the Dunkl harmonic oscillator containing new terms of the form $\rho^{-2u}$ and $\rho^{-2u-1}$. A different analytic tool is used here, which was developed by the authors \cite{AlvCalazaFranco2015}. Precisely, classical perturbation methods were used in \cite{AlvCalazaFranco2015} to determine self-adjoint operators with discrete spectra defined by this perturbation of the Dunkl harmonic oscillator, giving also upper and lower estimates of its eigenvalues. The application of this analytic tool is what requires $g$ to be good. The information obtained for this perturbation is weaker than for the Dunkl harmonic oscillator. For instance, such self-adjoint operators are only known to exist in some cases, and only a core of their square root is known. Thus more work is needed here than in \cite{AlvCalaza2017} to describe the Laplacians of the maximum/minimum i.b.c.\ of the simple elliptic complexes of length two, using those self-adjoint operators. The proof of Theorem~\ref{t: spectrum of Delta_max/min} can be completed with such information like in \cite{AlvCalaza2017}. On the other hand, only eigenvalue estimates of those self-adjoint operators are known, which makes it more difficult to determine the ``cohomological contribution'' of the rel-critical points. This is the key idea to complete the proof of Theorem~\ref{t: Morse inequalities} like in \cite{AlvCalaza2017}.

\subsection{Some open problems}\label{ss: open problems}

We do not know whether the condition on $g$ to be good could be deleted. It depends on whether the result used from \cite{AlvCalazaFranco2015} holds with weaker hypothesis.

The applications would increase by extending our version of Morse inequalities to ``rel-Morse-Bott functions.'' Their rel-critical point set would be a finite union of substratifications. 

There should be an extension of the isomorphism $H_{(2)}^r(M)\cong I^{\bar p}H_r(A)^*$ to the case of general adapted metrics and general perversities \cite{Friedman2011}. In that direction, an extension of the de~Rham theorem with general perversities was proved in \cite{Saralegi1994,Saralegi2005}. The case with classical perversities was previously considered in \cite{Brylinski1992,BrasseletHectorSaralegi1991}.

It is also natural to continue with the following program, already achieved on closed manifolds. First, it should be shown that there is a spectral gap of the form $\sigma(\D_{s,\text{\rm max/min}})\cap(C_1e^{-C_2s},C_3s)=\emptyset$, for some $C_1,C_2,C_3>0$. This would define a finite-dimensional complex $(\SS_{s,\text{\rm max/min}},d_s)$ generated by the eigenforms corresponding to eigenvalues in $[0,C_1e^{-C_2s}]$ (``small eigenvalues''). Second, it should be proved that $(\SS_{s,\text{\rm max/min}},d_s)$ ``converges'' to the ``rel-Morse-Thom-Smale complex,'' assuming that the function satisfies the ``rel-Morse-Smale transversality condition.'' It seems that the existence of the above spectral gap would follow easily by adapting the arguments of \cite[Propositions~14.2 and~14.3]{AlvCalaza2017}. The comparison of $(\SS_{s,\text{\rm max/min}},d_s)$ with the ``rel-Morse-Thom-Smale complex'' would require additional techniques, according to the case of closed manifolds \cite{HelfferSjostrand1985}, \cite[Section~6]{BismutZhang1994}. This program was developed by Ludwig in a special case \cite{Ludwig2017}.


\section{Preliminaries}\label{s: prelim}

\subsection{Products of cones}\label{ss: prelim, products}

Let $L$ and $L'$ be compact stratifications, and let $*$ and $\rho$, and $*'$ and $\rho'$ be the vertices and radial functions of $c(L)$ and $c(L')$. Any morphism $\psi:c(L)\to c(L')$ is of the form $c(\phi)$ around $*$ for some morphism $\phi:L\to L'$. In particular, $\psi(*)=*'$, and $\psi^*\rho'=\rho$ around $*$.

The product of two stratifications, $A\times A'$, has a stratification structure whose strata are the products of strata of $A$ and $A'$. However the tubes in $A\times A'$ depend on the choice of a function $h:[0,\infty)^2\to[0,\infty)$ that is continuous, homogeneous of degree one, smooth on $\R_+^2$, with $h^{-1}(0)=\{(0,0)\}$, and such that, for some $C>1$, we have $h(r,r')=\max\{r,r'\}$ if $C\min\{r,r'\}<\max\{r,r'\}$ \cite[Section~3.1.2]{AlvCalaza2017}. Thus the stratification structure of $A\times A'$ is not unique.

In the case of two cones, $c(L)\times c(L')$ can be described as another cone in the following way \cite[Lemma~3.8]{AlvCalaza2017}. The function $h(\rho\times\rho'):c(L)\times c(L')\to[0,\infty)$ satisfies that $L''=(h(\rho\times\rho'))^{-1}(1)$ is a compact saturated substratification of $c(L)\times c(L')$. Then the map
	\[
		\phi:c(L'')\to c(L)\times c(L')\;,\quad[([x,r],[x',r']),s]\mapsto([x,rs],[x',r's])\;,
	\]
is an isomorphism of stratifications. The vertex of $c(L'')$ is $*''=\phi^{-1}(*,*')$, and its radial function is $\rho''=\phi^*(h(\rho\times\rho'))$. Thus the radial function of $c(L)\times c(L')$, $(\rho^2+{\rho'}^2)^{1/2}$, does not correspond to $\rho''$ via $\phi$ if $L\ne\emptyset\ne L'$.

Assume that $L\ne\emptyset\ne L'$. Let $N$ and $N'$ be strata of $L$ and $L'$, and let $M=N\times\R_+$ and $M'=N'\times\R_+$ be the corresponding strata of $c(L)$ and $c(L')$. Take general adapted metrics $\tilde g$ and $\tilde g'$ on $N$ and $N'$, and fix any $u>0$. We get general adapted metrics $g=\rho^{2u}\tilde g+(d\rho)^2$ and $g'={\rho'}^{2u}\tilde g'+(d\rho')^2$ on $M$ and $M'$. On the other hand, with the above notation,  we have $\phi^{-1}(M\times M')=N''\times\R_+=:M''$, where $N''=(M\times M')\cap L''$ (a stratum of $L''$). Let $\tilde g''$ be any general adapted metric on $N''$ so that $N''\hookrightarrow M\times M'$ is quasi-isometric; for instance, we may take $\tilde g''=(g+g')|_{N''}$. We get the general adapted metric $g''={\rho''}^{2u}\tilde g''+(d\rho'')^2$ on $M''$. Equip $M\times M'$ with $g+g'$ and $M''$ with $g''$.

\begin{prop}\label{p: phi is a quasi-isometry if and only if u=1}
	\begin{enumerate}[{\rm(}i\/{\rm)}]
		
		\item\label{i: phi is a quasi-isometry} If $u=1$, then $\phi:M''\to M\times M'$ is a quasi-isometry. 
		
		\item\label{i: phi is not quasi-isometric} If $u<1$, then $\phi:M''\cap O\to (M\times M')\cap\phi(O)$ is not quasi-isometric for any neighborhood $O$ of $*''$ in $c(L'')$.
		
	\end{enumerate}
\end{prop}

\begin{proof}
	Without lost of generality, we can assume $\tilde g''=(g+g')|_{N''}$. We have
		\[
			M''=N''\times\R_+\subset M\times M'\times\R_+=N\times\R_+\times N'\times\R_+\times\R_+\;.
		\]
	According to this expression, an arbitrary point $p\in M''$ can be written as $p=(x,r,x',r',r'')\equiv(\bar p,r'')$, obtaining 
		\[
			\phi(p)=(x,rr'',x',r'r'')\in M\times M'=N\times\R_+\times N'\times\R_+\;.
		\]
	Thus we can canonically consider
		\begin{align*}
			T_{\bar p}N''&\subset T_xN\oplus\R\oplus T_{x'}N'\oplus\R\;,\\
			T_pM''&\subset T_xN\oplus\R\oplus T_{x'}N'\oplus\R\oplus\R\;,\\
			T_{\phi(p)}(M\times M')&=T_xN\oplus\R\oplus T_{x'}N'\oplus\R\;.
		\end{align*}
	We easily get
		\begin{align*}
			\phi_*(\partial_{\rho''}(p))&=(0,r\partial_\rho(rr''),0,r'\partial_{\rho'}(r'r''))\;,\\
			\phi_*(X,0)&=(Y,cr''\partial_\rho(rr''),Y',c'r''\partial_{\rho'}(r'r''))\;,
		\end{align*}
	for $X=(Y,c\partial_\rho(r),Y',c'\partial_{\rho'}(r'))\in T_{\bar p}N''$. Hence
		\begin{align}
			\|\partial_{\rho''}(p)\|_{g''}^2&=1\;,\label{|partial_rho''(p)|_g''^2 = 1}\\
			\|\phi_*(\partial_{\rho''}(p))\|_{g+g'}^2&=r^2+{r'}^2\;,
			\label{|phi_*(partial_rho''(p))|_g+g'^2 = r^2+r'^2}\\
			\|(X,0)\|_{g''}^2&={r''}^{2u}\,\|X\|_{\tilde g+\tilde g'}^2\notag\\
			&={r''}^{2u}\left(\|Y\|_{\tilde g}^2+c^2+\|Y'\|_{\tilde g'}^2+{c'}^2\right)\;,
			\label{|(X,0)|_g''^2}\\
			\|\phi_*(X,0)\|_{g+g'}^2&=\left({r''}^{2u}\,\|Y\|_{\tilde g}^2+c^2{r''}^2
			+{r''}^{2u}\,\|Y'\|_{\tilde g'}^2+{c'}^2{r''}^2\right)\notag\\
			&={r''}^{2u}\left(\|Y\|_{\tilde g}^2+c^2{r''}^{2(1-u)}
			+\|Y'\|_{\tilde g'}^2+{c'}^2{r''}^{2(1-u)}\right)\;,
			\label{|phi_*(X,0)|_g''^2}
		\end{align}
	where every metric is added as subindex of the corresponding norm. 
	
	Observe that $C_0:=\min_{N''}(\rho^2+{\rho'}^2)>0$ and $C_1:=\max_{N''}(\rho^2+{\rho'}^2)<\infty$ by the properties of $h$. So, by~\eqref{|partial_rho''(p)|_g''^2 = 1} and~\eqref{|phi_*(partial_rho''(p))|_g+g'^2 = r^2+r'^2},
		\[
			C_0\,\|\partial_{\rho''}(p)\|_{g''}^2\le\|\phi_*(\partial_{\rho''}(p))\|_{g+g'}^2
			\le C_1\,\|\partial_{\rho''}(p)\|_{g''}^2\;.
		\]
	Moreover, if $u=1$, then $\|\phi_*(X,0)\|_{g+g'}^2=\|(X,0)\|_{g''}^2$ by~\eqref{|(X,0)|_g''^2} and~\eqref{|phi_*(X,0)|_g''^2}, obtaining~\eqref{i: phi is a quasi-isometry}.
	
	Now, suppose that $u<1$. With the above notation, by the conditions satisfied by $h$, we can take  $\bar p=(x,r,x',1)\in N''$ and $X=(0,\partial_\rho(r),0,0)\in T_{\bar p}N''$ for all $r$ small enough. By~\eqref{|(X,0)|_g''^2} and~\eqref{|phi_*(X,0)|_g''^2}, it follows that
		\[
			\frac{\|\phi_*(X,0)\|_{g+g'}^2}{\|(X,0)\|_{g''}^2}
			={r''}^{2(1-u)}\to0
		\]
	as $r''\to0$, giving~\eqref{i: phi is not quasi-isometric}.	
\end{proof}

Similar observations apply to the product of any finite number of cones.

\subsection{General adapted metrics}\label{ss: prelim, general adapted metrics}

Consider the notation of Section~\ref{ss: general adapted metrics}. 

\begin{rem}\label{r: c(S^m-1) approx R^m}
	For every $m\in\Z_+$, there is a canonical homeomorphism $c(\S^{m-1})\approx\R^m$, $[x,\rho]\mapsto\rho x$, so that the radial function $\rho$ corresponds to the norm on $\R^m$ \cite[Example~3.7]{AlvCalaza2017}. This is not an isomorphism of stratifications: $c(\S^{m-1})$ has two strata and $\R^m$ only one; the stratum $\S^{m-1}\times\R_+$ of  $c(\S^{m-1})$ corresponds to $\R^m\sm\{0\}$. If $\tilde g$ denotes the standard metric on $\S^{m-1}$, then $\rho^2\tilde g+(d\rho)^2$ on $\S^{m-1}\times\R_+$ corresponds to the Euclidean metric on $\R^m\sm\{0\}$. Thus, with the notation of Section~\ref{s: intro}, the factors $\R^{m_X}$ or $\R^{m_\pm}$ could be also described as cones, or as strata of cones after removing one point.
\end{rem}

\begin{rem}\label{r: adapted metric 2}
  By taking charts and using induction on the depth, we get the following (cf. \cite[Remark~7]{AlvCalaza2017}):
    \begin{enumerate}[{\rm(}i\/{\rm)}]
          
      \item\label{i: rel-locally quasi-isometric} If two general adapted metrics on $M$ have the same type with respect to the same general tubes, then they are rel-locally quasi-isometric. In particular, they are quasi-isometric if $\ol{M}$ is compact.
      
      \item\label{i: vol(M cap O_m) to 0} Any point in $\ol{M}$ has a countable base $\{\,O_m\mid m\in\N\,\}$ of open neighborhoods such that, with respect to any general adapted metric, $\vol(M\cap O_m)\to0$ and $\max\{\,\diam P\mid P\in\pi_0(M\cap O_m)\,\}\to0$ as $m\to\infty$. Thus, if $\ol{M}$ is compact, then $\vol M<\infty$ and $\diam P<\infty$ for all $P\in\pi_0(M)$.
                
    \end{enumerate}
\end{rem}

\begin{rem}\label{r: sum_a lambda_a g_a}
The argument of \cite[Appendix]{BrasseletHectorSaralegi1992} also shows the following. Let $\{O_a\}$ be a locally finite open covering of $\ol{M}$, let $\{\lambda_a\}$ be a smooth partition of unity of $M$ subordinated to the open covering $\{M\cap O_a\}$, and let $g_a$ be a general adapted metric on every $M\cap O_a$. Suppose that the metrics $g_a$ have the same general type with respect to restrictions to the sets $O_a$ of the same general tubes. Then the metric $\sum_a\lambda_ag_a$ is general adapted on $M$ and has the same general type with respect to those general tubes.
\end{rem}

When $M$ is not connected, $\widehat M$ is defined as the disjoint union of the rel-local completion of the connected components of $M$ (Section~\ref{ss: general adapted metrics}), using \cite[Remark~1~(v)]{AlvCalaza2017}.

\begin{rem}\label{r: widehat M}
    \begin{enumerate}[{\rm(}i\/{\rm)}]
    
      \item\label{i: widehat M is independent of the adapted metric} By Remark~\ref{r: adapted metric 2}~\eqref{i: rel-locally quasi-isometric}, $\widehat{M}$ is independent of the choice of the general adapted metric of a given general type. In fact, by Remark~\ref{r: adapted metric 2}~\eqref{i: vol(M cap O_m) to 0} and \cite[Example~3.19]{AlvCalaza2017}, $\widehat{M}$ is also independent of the general type.
      
      \item\label{i: widehat M cap O} For any open $O\subset A$, we have $\widehat{M\cap O}\equiv\lim^{-1}(\ol M\cap O)\subset\widehat{M}$.
      
    \end{enumerate}
\end{rem}

\begin{rem}\label{r: rel-local completion of the strata of a cone}
  	The following is a direct consequence of Remark~\ref{r: widehat M}~\eqref{i: widehat M is independent of the adapted metric} and \cite[Remark~9~(i),(ii) and Proposition~3.20~(iii)]{AlvCalaza2017}:
	\begin{enumerate}[{\rm(}i\/{\rm)}]
    
      		\item\label{i: lim: widehat M to ol M is surjective} $\lim:\widehat{M}\to\ol{M}$ is surjective with finite fibers.
      
      		\item\label{i: M is rel-locally connected w.r.t widehat M} $M$ is rel-locally connected with respect to $\widehat{M}$.
		
		\item\label{i: hat phi} Let $M'$ be a connected stratum of another stratification $A'$ equipped with a general adapted metric, and let $\phi:A\to A'$ be a morphism with $\phi(M)\subset M'$. Then the restriction $\phi:M\to M'$ extends to a morphism $\hat\phi:\widehat{M}\to\widehat{M'}$. Moreover $\hat\phi$ is an isomorphism if $\phi$ is an isomorphism. 
    
    	\end{enumerate}
\end{rem}

\subsection{Relatively Morse functions}\label{ss: prelim, rel-Morse}

Consider the notation of Section~\ref{ss: rel-Morse}. Besides the observations given in that section, the following holds like in the case of adapted metrics of conic type \cite[Section~4]{AlvCalaza2017}.

\begin{rem}\label{r: rel-admissible, rel-critical, rel-non-degenerate}
    \begin{enumerate}[{\rm(}i\/{\rm)}]
    
      \item\label{i: rel-local boundedness of |df| is invariant} The rel-local boundedness of $|df|$ is invariant by rel-local quasi-isometries, and therefore it depends only on the general type of $g$. Similarly, the definition of rel-critical point depends only on the general type of $g$. But the rel-local boundedness of $|\Hess f|$ depends on the choice of $g$. However it follows from~\eqref{i: lambda_a} and~\eqref{i: rel-admissible w.r.t. g = sum_a lambda_ag_a} below that the existence of $g$ so that $f$ is rel-admissible with respect to $g$ is a rel-local property.
      
      \item\label{i: depth M=0} If $\depth M=0$, then any smooth function is admissible, and its rel-critical points are its critical points.
      
      \item\label{i: ex of rel-admissible function} With the notation of Section~\ref{ss: prelim, products}, let $h\in\Cinf(\R_+)$ with $h'\in\Cinf_0(\R_+)$. Then the function $h(\rho)$ is rel-admissible on the stratum $M$ of $c(L)$ with respect to any general adapted metric.
      
      \item\label{i: lambda_a} Let $\{\,O_a\mid a\in\AA\,\}$ be a locally finite covering of $\ol{M}$ by open subsets of $A$. Then there is a $C^\infty$ partition of unity $\{\lambda_a\}$ on $M$ subordinated to $\{M\cap O_a\}$ such that $|d\lambda_a|$ is rel-locally bounded for all general adapted metrics on $M$ of any fixed general type.
      
      \item\label{i: rel-admissible w.r.t. g = sum_a lambda_ag_a} Suppose that $\{\lambda_a\}$ and  $\{g_a\}$ satisfy the conditions of Remark~\ref{r: sum_a lambda_a g_a} and~\eqref{i: lambda_a}. Let $f\in\Cinf(M)$ such that every $f|_{M\cap O_a}$ is rel-admissible with respect to $g_a$. Then $f$ is rel-admissible with respect to the general adapted metric $g=\sum_a\lambda_ag_a$ on $M$.
      
      \item\label{i: existence of rel-Morse functions} Let $\FF\subset C^\infty(M)$ denote the subset of functions with continuous extensions to $\ol{M}$ that restrict to rel-Morse functions with respect to all general adapted metrics of all possible general types on all strata $\le M$. Then $\FF$ is dense in $C^\infty(M)$ with the weak $C^\infty$ topology.
          
    \end{enumerate}
\end{rem}

\subsection{Hilbert and elliptic complexes}\label{ss: prelim Hilbert}

Consider the notation of Section~\ref{ss: ibc}.

\subsubsection{Hilbert complexes with a discrete positive spectrum}\label{sss: discrete Hilbert complexes}

Let $(\sD,\bd)$ be a Hilbert complex in a graded separable Hilbert space $\fH$, defining self-adjoint operators $\bD$ and $\bDelta$ according to Section~\ref{ss: ibc}. The direct sum of homogeneous subspaces of even/odd degree are denoted with the subindex ``ev/odd''. The same subindex is used to denote the restriction of homogeneous operators to such subspaces. 

\begin{lem}\label{l: discrete spectrum}
	The positive spectrum of $\bDelta_{\text{\rm ev}}$ is discrete\footnote{Recall that a complex number is in the discrete spectrum of a normal operator in a Hilbert space when it is an eigenvalue of finite multiplicity.} and bounded away from zero if and only if the positive spectrum of $\bDelta_{\text{\rm odd}}$ is discrete and bounded away from zero. In this case, both operators have the same positive eigenvalues, with the same multiplicity.
\end{lem}

\begin{proof}
	For instance, suppose that the positive spectrum of $\bDelta_{\text{\rm ev}}$ is discrete and bounded away from zero. It follows from the spectral theorem that
		\[
			\sD^\infty(\bDelta_{\text{\rm ev/odd}})=\ker\bDelta_{\text{\rm ev/odd}}
			\oplus\bDelta(\sD^\infty(\bDelta_{\text{\rm ev/odd}}))\;,
		\]
	and
		\[
			\bD_{\text{\rm ev}}:\bDelta(\sD^\infty(\bDelta_{\text{\rm ev}}))
			\to\bDelta(\sD^\infty(\bDelta_{\text{\rm odd}}))
		\]
	is a linear isomorphism satisfying $\bD_{\text{\rm ev}}\bDelta_{\text{\rm ev}}=\bDelta_{\text{\rm odd}}\bD_{\text{\rm ev}}$.
\end{proof}

\subsubsection{Elliptic complexes with a term that is a direct sum}\label{sss: elliptic oplus}

Let $E=\bigoplus_rE_r$ be a graded Riemannian or Hermitian vector bundle over a Riemannian manifold $M$. The space of its smooth sections is denoted by $\Cinf(E)$, its subspace of compactly supported smooth sections is denoted by $\Cinf_0(E)$, and the Hilbert space of square integrable sections of $E$ is denoted by $L^2(E)$. All of these are graded spaces. Consider differential operators of the same order, $d_r:\Cinf(E_r)\to\Cinf(E_{r+1})$, such that $(\Cinf(E),d=\bigoplus_rd_r)$ is an elliptic\footnote{Recall that ellipticity means that the sequence of principal symbols of the operators $d_r$ is exact over every nonzero cotangent vector.} complex. The simpler notation $(E,d)$ (or even $d$) will be preferred. Elliptic complexes with nonzero terms of negative degrees or homogeneous differential operators of degree $-1$ may be also considered without any essential change. For instance, we have the formal adjoint elliptic complex $(E,\delta)$.

Suppose that there is an orthogonal decomposition $E_{r+1}=E_{r+1,1}\oplus E_{r+1,2}$ for some degree $r+1$. Thus
  \begin{align*}
    C^\infty(E_{r+1})&\equiv C^\infty(E_{r+1,1})\oplus C^\infty(E_{r+1,2})\;,\\ 
    C^\infty_0(E_{r+1})&\equiv C^\infty_0(E_{r+1,1})\oplus C^\infty_0(E_{r+1,2})\;,\\
    L^2(E_{r+1})&\equiv L^2(E_{r+1,1})\oplus L^2(E_{r+1,2})\;,
  \end{align*}
and we can write
  \begin{alignat*}{2}
    d_r&=
      \begin{pmatrix}
        d_{r,1}\\
        d_{r,2}
      \end{pmatrix}\;,&\quad
    \delta_r&=
      \begin{pmatrix}
        \delta_{r,1} & \delta_{r,2}
      \end{pmatrix}\;,\\
      d_{r+1}&=
      \begin{pmatrix}
        d_{r+1,1} & d_{r+1,2}
      \end{pmatrix}\;,&\quad
    \delta_{r+1}&=
      \begin{pmatrix}
        \delta_{r+1,1} \\ 
        \delta_{r+1,2}
      \end{pmatrix}\;.
  \end{alignat*}
The operators $d_{r,i}$ and $\delta_{r,i}$ can be also considered as elliptic complexes of length one, and therefore they have a maximum/minimum i.b.c., $d_{r,i,\text{\rm max/min}}$ and $\delta_{r,i,\text{\rm max/min}}$. 
  
\begin{lem}[{\cite[Lemma~8.2]{AlvCalaza2017}}]\label{l: oplus 1}
  We have:
    \[
      \sD(d_{\text{\rm max},r})=\sD(d_{r,1,\text{\rm max}})\cap\sD(d_{r,2,\text{\rm max}})\;,\quad
      d_{\text{\rm max},r}=
        \begin{pmatrix}
          d_{r,1,\text{\rm max}}|_{\sD(d_{\text{\rm max},r})}\\
          d_{r,2,\text{\rm max}}|_{\sD(d_{\text{\rm max},r})}
        \end{pmatrix}\;.
    \]
\end{lem}

\begin{lem}\label{l: oplus 2}
	We have:
		\begin{equation}
			\sD(d_{r+1,1,\text{\rm max/min}})\oplus\sD(d_{r+1,2,\text{\rm max/min}})
			\subset\sD(d_{\text{\rm max/min},r+1})\;.
		\end{equation}
\end{lem}

\begin{proof}
	Take any $\left(\begin{smallmatrix}u\\v\end{smallmatrix}\right)\in\sD(d_{r+1,1,\text{\rm min}})\oplus\sD(d_{r+1,2,\text{\rm min}})$, and let $u'=d_{r+1,1,\text{\rm min}}u$ and $v'=d_{r+1,2,\text{\rm min}}v$. This means that there are sequences, $u_i$ in $C^\infty_0(E_{r+1,1})$ and $v_i$ in $C^\infty_0(E_{r+1,2})$, such that $u_i\to u$ in $L^2(E_{r+1,1})$, $v_i\to v$ in $L^2(E_{r+1,2})$, $d_{r+1,1}u_i\to u'$ and $d_{r+1,2}v_i\to v'$ in $L^2(E_{r+2})$. So $\left(\begin{smallmatrix}u_i\\v_i\end{smallmatrix}\right)\in C^\infty_0(E_{r+1,1})\oplus C^\infty_0(E_{r+1,2})\equiv C^\infty_0(E_{r+1})$, $\left(\begin{smallmatrix}u_i\\v_i\end{smallmatrix}\right)\to\left(\begin{smallmatrix}u\\v\end{smallmatrix}\right)$ in $L^2(E_{r+1})$ and $d_{r+1}\left(\begin{smallmatrix}u_i\\v_i\end{smallmatrix}\right)\to u'+v'$ in $L^2(E_{r+2})$, obtaining $\left(\begin{smallmatrix}u\\v\end{smallmatrix}\right)\in\sD(d_{\text{\rm min},r+1})$.
	
	Now, take any $\left(\begin{smallmatrix}u\\v\end{smallmatrix}\right)\in\sD(d_{r+1,1,\text{\rm max}})\oplus\sD(d_{r+1,2,\text{\rm max}})$, and let $u'=d_{r+1,1,\text{\rm max}}u$ and $v'=d_{r+1,2,\text{\rm max}}v$. This means that $\langle u,\delta_{r+1,1}w\rangle=\langle u',w\rangle$ and $\langle v,\delta_{r+1,2}w\rangle=\langle v',w\rangle$ for all $w\in C^\infty_0(E_{r+2})$. Thus $\langle\left(\begin{smallmatrix}u\\v\end{smallmatrix}\right),\delta_{r+1}w\rangle=\langle u'+v',w\rangle$ for all $w\in C^\infty_0(E_{r+2})$, obtaining that $\left(\begin{smallmatrix}u\\v\end{smallmatrix}\right)\in\sD(d_{\text{\rm max},r+1})$.
\end{proof}

\section{A perturbation of the Dunkl harmonic oscillator}\label{s: P, Q, W}

This section is devoted to recall the study of self-adjoint operators on $\R_+$ induced by the Dunkl harmonic oscillator on $\R$ \cite{AlvCalaza2014}, and also by certain perturbation of the Dunkl harmonic oscillator on $\R$ \cite{AlvCalazaFranco2015}. This is the main analytic tool of the paper.

Let $\SS=\SS(\R)$ be the real-/complex-valued Schwartz space on $\R$, with its Fr\'echet topology. It decomposes as direct sum of subspaces of even and odd functions, $\SS=\SS_{\text{\rm ev}}\oplus\SS_{\text{\rm odd}}$. For $\sigma>-\frac{1}{2}$, the sequence of generalized Hermite polynomials, $p_k=p_{s,\sigma,k}(x)$, consists of the orthogonal polynomials associated with the measure $e^{-sx^2}|x|^{2\sigma}\,dx$ on $\R$ \cite[p.~380, Problem~25]{Szego1975}. It is assumed that every $p_k$ is normalized and has positive leading coefficient. They give rise to the general Hermite functions $\phi_k=\phi_{s,\sigma,k}(x)=p_ke^{-sx^2/2}\in\SS$. If $k$ is odd, then $p_{s,\tau,k}$ and $\phi_{s,\tau,k}$ also make sense for $\tau>-\frac{3}{2}$.

Now, let $\rho$ denote the canonical coordinate of $\R_+$. Consider the spaces of real-/complex-valued functions, $C^\infty=C^\infty(\R)$, $C^\infty_+=C^\infty(\R_+)$ and $C^\infty_{+,0}=C^\infty_0(\R_+)$, where the subindex $0$ is used for compactly supported functions or sections. For every $a\in\R$, the operator of multiplication by the function $\rho^a$ on $C^\infty_+$ will be also denoted by $\rho^a$. We have
  \begin{equation}\label{[d/d rho,rho^a right]}
    [\textstyle{\frac{d}{d\rho}},\rho^a]=a\rho^{a-1}\;,\quad
    [\textstyle{\frac{d^2}{d\rho^2}},\rho^a]=2a\rho^{a-1}\,\textstyle{\frac{d}{d\rho}}+a(a-1)\rho^{a-2}\;.
  \end{equation}
For every $\phi\in C^\infty$, let $\phi_+=\phi|_{\R_+}$, and let $\SS_{\text{\rm ev/odd},+}=\{\,\phi_+\mid\phi\in\SS_{\text{\rm ev/odd}}\,\}$. For $c,d>-\frac{1}{2}$, let $L^2_{c,+}=L^2(\R_+,\rho^{2c}\,d\rho)$ and $L^2_{c,d,+}=L^2_{c,+}\oplus L^2_{d,+}$, whose scalar products are denoted by $\langle\ ,\ \rangle_c$ and $\langle\ ,\ \rangle_{c,d}$, and the corresponding norms by $\|\ \|_c$ and $\|\ \|_{c,d}$, respectively. The simpler notation $L^2_+$, $\langle\ ,\ \rangle$ and $\|\ \|$ is used when $c=0$. Recall that the harmonic oscillator on $C^\infty_+$ is the operator $H=-\frac{d^2}{d\rho^2}+s^2\rho^2$ ($s>0$). For $c_1,c_2,d_1,d_2\in\R$, let	
	\begin{equation}\label{P_0, Q_0}
		P_0=H-2c_1\rho^{-1}\,\textstyle{\frac{d}{d\rho}}+c_2\rho^{-2}\;,\quad
		Q_0=H-2d_1\textstyle{\frac{d}{d\rho}}\,\rho^{-1}+d_2\rho^{-2}\;.
	\end{equation} 

\begin{prop}[{\cite[Theorem~1.4]{AlvCalaza2014}}]\label{p: PP_0}
  	If $a\in\R$ satisfies
    		\begin{gather}
      			a^2+(2c_1-1)a-c_2=0\;,\label{a}\\
      			\sigma:=a+c_1>-\textstyle{\frac{1}{2}}\;,\label{sigma > -1/2}
    		\end{gather}
  	then the following holds:
    		\begin{enumerate}[{\rm(}i\/{\rm)}]
  
      			\item\label{i: P_0 is essentially self-adjoint} $P_0$, with $\sD(P_0)=\rho^a\SS_{\text{\rm ev},+}$, is essentially self-adjoint in $L^2_{c_1,+}$.
    
      			\item\label{i: the spectrum of PP_0} The spectrum of $\PP_0:=\ol{P_0}$ consists of the eigenvalues
				\begin{equation}\label{lambda_k = ..., case of PP_0}
					\lambda_k=(2k+1+2\sigma)s\;,
				\end{equation} 
			for $k\in2\N$, with multiplicity one and corresponding normalized eigenfunctions $\chi_k=\chi_{s,\sigma,a,k}:=\sqrt{2}\,\rho^a\phi_{s,\sigma,k,+}$.
    
      			\item\label{i: sD^infty(PP_0)} $\sD^\infty(\PP_0)=\rho^a\SS_{\text{\rm ev},+}$.
    
    		\end{enumerate}
\end{prop}

\begin{prop}[{See \cite[Section~5]{AlvCalaza2014}}]\label{p: QQ_0}
  	If $b\in\R$ satisfies
  		\begin{gather}
			b^2+(2d_1+1)b-d_2=0\;,\label{b}\\
			\tau:=b+d_1>-\textstyle{\frac{3}{2}}\;,\label{tau > -3/2}
		\end{gather}
  then the following holds:
    \begin{enumerate}[{\rm(}i\/{\rm)}]
  
      \item\label{i: Q_0 is essentially self-adjoint} $Q_0$, with $\sD(Q_0)=\rho^b\SS_{\text{\rm odd},+}$, is essentially self-adjoint in $L^2_{d_1,+}$.
    
      \item\label{i: the spectrum of ol Q_0} The spectrum of $\QQ_0:=\ol{Q_0}$ consists of the eigenvalues given by the expression~\eqref{lambda_k = ..., case of PP_0}, for $k\in2\N+1$ and using $\tau$ instead of $\sigma$, with multiplicity one and corresponding normalized eigenfunctions $\chi_k=\chi_{s,\tau,b,k}:=\sqrt{2}\,\rho^b\phi_{s,\tau,k,+}$.
    
      \item\label{i: sD^infty(QQ_0)} $\sD^\infty(\QQ_0)=\rho^b\SS_{\text{\rm odd},+}$.
    
    \end{enumerate}
\end{prop}	

\begin{prop}[{\cite[Corollary~8.1]{AlvCalazaFranco2015}}]\label{p: PP}
	Let $\xi>0$ and
		\begin{equation}\label{0 < u < 1}
			0<u<1\;.
		\end{equation}
	If $a\in\R$ satisfies~\eqref{a} and
		\begin{equation}\label{sigma > u - 1/2}
			\sigma:=a+c_1>u-\textstyle{\frac{1}{2}}\;,
		\end{equation}
	then there is a positive self-adjoint operator $\PP$ in $L^2_{c_1,+}$ satisfying the following:
    		\begin{enumerate}[{\rm(}i\/{\rm)}]
  
      			\item\label{i: langle PP^1/2 phi,PP^1/2 psi rangle_c_1} $\rho^a\SS_{\text{\rm ev},+}$ is a core of ${\PP}^{1/2}$ and, for all $\phi,\psi\in\rho^a\SS_{\text{\rm ev},+}$,
				\begin{equation}\label{langle PP^1/2 phi,PP^1/2 psi rangle_c_1}
					\langle{\PP}^{1/2}\phi,{\PP}^{1/2}\psi\rangle_{c_1}=\langle P_0\phi,\psi\rangle_{c_1}
					+\xi\langle\rho^{-u}\phi,\rho^{-u}\psi\rangle_{c_1}\;.
				\end{equation}
    
      			\item\label{i: ... le lambda_k ge ..., case of PP} $\PP$ has a discrete spectrum. Let $\lambda_0\le\lambda_2\le\cdots$ be its eigenvalues, repeated according to their multiplicity. There is some $D=D(\sigma,u)>0$ and, for any $\epsilon>0$, there is some $C=C(\epsilon,\sigma,u)>0$ so that, for all $k\in2\N$,
				\begin{align}
					\lambda_k&\ge(2k+1+2\sigma)s+\xi Ds^u(k+1)^{-u}\;,
					\label{lambda_k ge ..., case of PP}\\
					\lambda_k&\le(2k+1+2\sigma)(s+\xi\epsilon s^u)+\xi Cs^u\;.
					\label{lambda_k le ..., case of PP}
				\end{align}
    
    		\end{enumerate}
\end{prop}

\begin{prop}[{\cite[Corollary~8.2]{AlvCalazaFranco2015}}]\label{p: QQ}
	For $\xi$ and $u$ like in Proposition~\ref{p: PP}, if $b\in\R$ satisfies~\eqref{b} and
		\begin{equation}\label{tau > u - 3/2}
			\tau:=b+d_1>u-\textstyle{\frac{3}{2}}\;,
		\end{equation}
	then there is a positive self-adjoint operator $\QQ$ in $L^2_{d_1,+}$ satisfying the following:
    		\begin{enumerate}[{\rm(}i\/{\rm)}]
  
      			\item\label{i: langle QQ^1/2 phi,QQ^1/2 psi rangle_d_1} $\rho^b\SS_{\text{\rm odd},+}$ is a core of ${\QQ}^{1/2}$ and, for all $\phi,\psi\in\rho^b\SS_{\text{\rm odd},+}$,
				\begin{equation}\label{langle QQ^1/2 phi,QQ^1/2 psi rangle_d_1}
					\langle{\QQ}^{1/2}\phi,{\QQ}^{1/2}\psi\rangle_{d_1}=\langle Q_0\phi,\psi\rangle_{d_1}
					+\xi\langle\rho^{-u}\phi,\rho^{-u}\psi\rangle_{d_1}\;.
				\end{equation}
    
      			\item\label{i: lambda_k ge ..., case of QQ} $\QQ$ has a discrete spectrum. Let $\lambda_1\le\lambda_3\le\cdots$ be its eigenvalues, repeated according to their multiplicity.  There is some $D=D(\tau,u)>0$ and, for any $\epsilon>0$, there is some $C=C(\epsilon,\tau,u)>0$ so that~\eqref{lambda_k ge ..., case of PP} and~\eqref{lambda_k le ..., case of PP} are satisfied, for $k\in2\N+1$ and with $\tau$ instead of $\sigma$.
    
    		\end{enumerate}
\end{prop}

\begin{prop}[{\cite[Corollary~8.3]{AlvCalazaFranco2015}}]\label{p: WW}
	Consider the notation and conditions of Propositions~\ref{p: PP} and~\ref{p: QQ}. Fix also some $\eta\in\R$, and let
	\begin{equation}\label{theta > -1/2}
		\theta>-\textstyle{\frac{1}{2}}\;.
	\end{equation}
Moreover suppose that the following properties hold:
		\begin{enumerate}[{\rm(}a\/{\rm)}]
		
			\item\label{i: sigma = theta ne tau} If $\sigma=\theta\ne\tau$ and $\tau-\sigma\not\in-\N$, then
				\begin{equation}\label{WW, sigma = theta ne tau}
					\sigma-1<\tau<\sigma+1,2\sigma+\textstyle{\frac{1}{2}}\;.
				\end{equation}
			
			\item\label{i: sigma ne theta = tau} If $\sigma\ne\theta=\tau$ and $\sigma-\tau\not\in-\N$, then 
				\begin{equation}\label{WW, sigma ne theta = tau}
					-\tau,\tau-1<\sigma<3\tau+1,11\tau+2,\tau+1\;.
				\end{equation}	
			
			\item\label{i: sigma ne theta = tau+1} If $\sigma\ne\theta=\tau+1$ and $\sigma-\tau-1\not\in-\N$, then 
				\begin{equation}\label{WW, sigma ne theta = tau+1}
					\tau+1<\sigma<\tau+3,2\tau+\textstyle{\frac{7}{2}}\;.
				\end{equation}
			
			\item\label{i: sigma ne theta ne tau} If $\sigma\ne\theta\ne\tau$ and $\sigma-\theta,\tau-\theta\not\in-\N$, then 
				\begin{equation}\label{WW, sigma ne theta ne tau}
					\left.
					\begin{array}{c}
						\frac{\sigma-\tau}{2}-1,\frac{\tau-\sigma}{2},\frac{\sigma+\tau-1}{4},\frac{\sigma+3\tau-2}{14},\frac{3\sigma+\tau-4}{14},\frac{\sigma+\tau-1}{2}
						<\theta<\frac{\sigma+\tau+1}{2}\;,\\
						\tau-1<\sigma<\tau+3\;.
					\end{array}
					\right\}
				\end{equation}
				
			\end{enumerate}
	Then there is a positive self-adjoint operator $\WW$ in $L^2_{c_1,d_1,+}$ satisfying the following:
    		\begin{enumerate}[{\rm(}i\/{\rm)}]
  
      			\item\label{i: langle WW^1/2 phi,WW^1/2 psi rangle_c_1,d_1} $\rho^a\SS_{\text{\rm ev},+}\oplus\rho^b\SS_{\text{\rm odd},+}$ is a core of $\WW^{1/2}$, and, for $\phi=(\phi_1,\phi_2)$ and $\psi=(\psi_1,\psi_2)$ in $\rho^a\SS_{\text{\rm ev},+}\oplus\rho^b\SS_{\text{\rm odd},+}$,
				\begin{multline}\label{langle WW^1/2 phi, WW^1/2 psi rangle_c_1,d_1}
					\langle{\WW}^{1/2}\phi,{\WW}^{1/2}\psi\rangle_{c_1,d_1}
					=\langle(P_0\oplus Q_0)\phi,\psi\rangle_{c_1,d_1}
					+\xi \langle\rho^{-u}\phi,\rho^{-u}\psi\rangle_{c_1,d_1}\\
					\text{}+\eta\left(\langle\rho^{-a-b-1}\phi_2,\psi_1\rangle_\theta
					+\langle\phi_1,\rho^{-a-b-1}\psi_2\rangle_\theta\right)\;.
				\end{multline}
    
      			\item\label{i: ... le lambda_k le ..., case of WW} $\WW$ has a discrete spectrum. Its eigenvalues form two groups, $\lambda_0\le\lambda_2\le\cdots$ and $\lambda_1\le\lambda_3\le\cdots$, repeated according to their multiplicity, such that there is some $D=D(\sigma,\tau,u)>0$ and, for every $\epsilon>0$, there are some $C=C(\epsilon,\sigma,\tau,u)>0$ and $E=E(\epsilon,\sigma,\tau,\theta)>0$ so that, for all $k\in\N$,
				\begin{align}
					\lambda_k&\ge(2k+1+2\varsigma_k)\big(s-2|\eta|\epsilon s^{\frac{v+1}{2}}\big)
					+\xi Ds^u(k+1)^{-u}-2|\eta|Es^{\frac{v+1}{2}}\;,
					\label{lambda_k ge ..., case of WW}\\
					\lambda_k&\le(2k+1+2\varsigma_k)\big(s+\epsilon\big(\xi s^u+2|\eta|s^{\frac{v+1}{2}}\big)\big)
					+\xi Cs^u+2|\eta|Es^{\frac{v+1}{2}}\;,
					\label{lambda_k le ..., case of WW}
				\end{align}
			where $v=\sigma+\tau-2\theta$, $\varsigma_k=\sigma$ if $k$ is even, and $\varsigma_k=\tau$ if $k$ is odd.
				
			\item\label{i: tilde u} Let $\tilde u\in\R$ such that
				\begin{equation}\label{tilde u}
					\textstyle{0,v,\tau-2\theta+\frac{1}{2},\sigma-2\theta-\frac{1}{2}<\tilde u<1,v+1,\sigma+\frac{1}{2},\tau+\frac{3}{2}}\;,
				\end{equation}
			and let $\hat u=\max\{\tilde u,v+1-\tilde u\}$. There is some $D=D(\sigma,\tau,u)>0$ and, for any $\epsilon>0$, there is some $\widetilde C=\widetilde C(\epsilon,\sigma,\tau,u)>0$ so that, for all $k\in\N$,
				\begin{equation}\label{lambda_k ge ..., case of WW with tilde u}
					\lambda_k\ge(2k+1+2\varsigma_k)\big(s-|\eta|\epsilon s^{\hat u}\big)+\xi Ds^u(k+1)^{-u}-|\eta|\widetilde Cs^{\hat u}\;.
				\end{equation}
				
			\item\label{i: lambda_k ge ..., case of WW with u = (v+1)/2, xi ge |eta|} If $u=\frac{v+1}{2}$ and $\xi\ge|\eta|$, then there is some $\widetilde D=\widetilde D(\sigma,\tau,u)>0$ so that, for all $k\in\N$,
				\begin{equation}\label{lambda_k ge ..., case of VV with u = (v+1)/2, xi ge |eta|}
					\lambda_k\ge(2k+1+2\varsigma_k)s+(\xi-|\eta|)\widetilde Ds^u(k+1)^{-u}\;.
				\end{equation}
				
			\item\label{i: WW, xi', xi''} If we add the term $\xi'\langle\phi_1,\psi_1\rangle_{c_1}+\xi''\langle\phi_2,\psi_2\rangle_{d_1}$ to the right-hand side of~\eqref{langle WW^1/2 phi, WW^1/2 psi rangle_c_1,d_1}, for some $\xi',\xi''\in\R$, then the result holds as well with the additional term $\max\{\xi',\xi''\}$ in the right-hand side of~\eqref{lambda_k le ..., case of WW}, and the additional term, $\xi'$ for $k\in2\N$ and $\xi''$ for $k\in2\N+1$, in the right-hand sides of~\eqref{lambda_k ge ..., case of WW},~\eqref{lambda_k ge ..., case of WW with tilde u} and~\eqref{lambda_k ge ..., case of VV with u = (v+1)/2, xi ge |eta|}.
    
    		\end{enumerate}
\end{prop}

\begin{rem}\label{r: P_0, Q_0, P, Q, F, G}
		\begin{enumerate}[(i)]
			
			\item\label{i: langle h chi_0,chi_0 rangle_c_1 to 1} If $h$ is a bounded measurable function on $\R_+$ with $h(\rho)\to1$ as $\rho\to0$, then $\langle h\chi_0,\chi_0\rangle_{c_1}\to1$ as $s\to\infty$ \cite[Lemma~7.3]{AlvCalaza2017}.
		
			\item\label{i: existence of a} The existence of $a\in\R$ satisfying~\eqref{a} is characterized by the condition $(2c_1-1)^2+4c_2\ge0$, which holds if $c_2\ge\min\{0,2c_1\}$. If $c_2=0$, then~\eqref{a} means that $a\in\{0,1-2c_1\}$.  If $c_2=2c_1$, then~\eqref{a} means that $a\in\{1,-2c_1\}$.
		
			\item\label{i: existence of b} The existence of $b\in\R$ satisfying~\eqref{b} is characterized by the condition $(2d_1+1)^2+4d_2\ge0$, which holds if $d_2\ge\min\{0,-2d_1\}$. If $d_2=0$, then~\eqref{b} means that $b\in\{0,-1-2d_1\}$.  If $d_2=-2d_1$, then~\eqref{b} means that $b\in\{-1,-2d_1\}$.
			
			\item\label{i: P_0 and Q_0} Propositions~\ref{p: PP_0} and~\ref{p: QQ_0} are indeed equivalent, as well as Propositions~\ref{p: PP} and~\ref{p: QQ}, because, if $c_1=d_1+1$ and $c_2=d_2$, then $Q_0=\rho P_0\rho^{-1}$ by~\eqref{[d/d rho,rho^a right]}, and $\rho:L^2_{c_1,+}\to L^2_{d_1,+}$ is a unitary isomorphism.

			\item\label{i: PP = ol P} We have $\PP=\ol{P}$, $\QQ=\ol{Q}$ and $\WW=\ol{W}$, where 
				\begin{gather}
					P=P_0+\xi\rho^{-2u}\;,\quad Q=Q_0+\xi\rho^{-2u}\;,\label{P, Q}\\
					W=
						\begin{pmatrix}
							P & \eta\rho^{2(\theta-c_1)-a-b-1} \\
							\eta\rho^{2(\theta-d_1)-a-b-1} & Q
						\end{pmatrix}\;,\label{W}
				\end{gather}
			with $\sD(P)=\sD^\infty(\PP)$, $\sD(Q)=\sD^\infty(\QQ)$ and $\sD(W)=\sD^\infty(\WW)$ \cite[Remark~1.4~(i) and Section~8]{AlvCalazaFranco2015}.
	
			\item\label{i: sD(PP^1/2) = sD(PP_0^1/2)} We have
				\begin{gather*}
					\sD(\PP^{1/2})=\sD(\PP_0^{1/2})\;,\quad\sD(\QQ^{1/2})=\sD(\QQ_0^{1/2})\;,\quad
					\sD(\WW^{1/2})=\sD((\PP_0\oplus\QQ_0)^{1/2})\;.
				\end{gather*}
			Thus the expressions~\eqref{langle PP^1/2 phi,PP^1/2 psi rangle_c_1},~\eqref{langle QQ^1/2 phi,QQ^1/2 psi rangle_d_1} and~\eqref{langle WW^1/2 phi, WW^1/2 psi rangle_c_1,d_1} can be extended to $\phi$ and $\psi$ in $\sD(\PP^{1/2})$, $\sD(\QQ^{1/2})$ and $\sD(\WW^{1/2})$, respectively, using
				\[
					\langle\PP_0^{1/2}\phi,\PP_0^{1/2}\psi\rangle_{c_1}\;,\quad
					\langle\QQ_0^{1/2}\phi,\QQ_0^{1/2}\psi\rangle_{d_1}\;,\quad
					\langle(\PP_0\oplus\QQ_0)^{1/2}\phi,
					(\PP_0\oplus\QQ_0)^{1/2}\psi\rangle_{c_1,d_1}
				\]
			instead of
				\[
					\langle P_0\phi,\psi\rangle_{c_1}\;,\quad
					\langle Q_0\phi,\psi\rangle_{d_1}\;,\quad
					\langle(P_0\oplus Q_0)\phi,\psi\rangle_{c_1,d_1}\;,
				\]
			 respectively \cite[Remark~3.21 and Section~8]{AlvCalazaFranco2015}.
			 
			 \item\label{i: remark about tilde u} In Proposition~\ref{p: WW}~\eqref{i: tilde u}, the condition~\eqref{tilde u} means that~\eqref{0 < u < 1},~\eqref{sigma > u - 1/2} and~\eqref{tau > u - 3/2} also hold with $\tilde u$ and $v+1-\tilde u$ instead of $u$. There exists $\tilde u$ satisfying~\eqref{tilde u} just when
				\begin{equation}\label{exists tilde u}
					\textstyle{0,v,\tau-2\theta+\frac{1}{2},\sigma-2\theta-\frac{1}{2}<1,v+1,\sigma+\frac{1}{2},\tau+\frac{3}{2}}\;.
				\end{equation}
			This property is satisfied in the cases~\eqref{i: sigma ne theta = tau} and~\eqref{i: sigma ne theta ne tau} by~\eqref{0 < u < 1},~\eqref{sigma > u - 1/2},~\eqref{tau > u - 3/2},~\eqref{theta > -1/2},~\eqref{WW, sigma ne theta = tau} and~\eqref{WW, sigma ne theta ne tau}; in particular, we can take $\tilde u=\frac{v+1}{2}$. By~\eqref{0 < u < 1},~\eqref{sigma > u - 1/2},~\eqref{tau > u - 3/2},~\eqref{theta > -1/2} and~\eqref{WW, sigma = theta ne tau} (respectively,~\eqref{WW, sigma ne theta = tau+1}), in the case~\eqref{i: sigma = theta ne tau} (respectively, in the case~\eqref{i: sigma ne theta = tau+1}), we have~\eqref{exists tilde u} if and only if $\tau<3\sigma$ (respectively, $\sigma<3\tau+4$).
	
	\end{enumerate}
\end{rem}

Consider the conditions and notation of Proposition~\ref{p: PP}, and the notation of Proposition~\ref{p: PP_0}. Take a complete orthonormal system $\{\,\hat\chi_k=\hat\chi_{\PP,k}\mid k\in2\N\,\}$ of $L^2_{c_1,+}$ so that every $\hat\chi_k$ is a $\lambda_k$-eigenfunction of $\PP$. Let $\hat\chi_k'=\hat\chi_{\PP,k}'$ and $\hat\chi_k''=\hat\chi_{\PP,k}''$ denote the orthogonal projections of every $\hat\chi_k$ to the subspaces spanned by $\chi_k$ and $\{\,\chi_i\mid k>i\in2\N\,\}$, respectively; in particular, $\hat\chi_0''=0$. Let also $\hat\chi_k'''=\hat\chi_{\PP,k}'''=\hat\chi_k-\hat\chi_k'-\hat\chi_k''$.

\begin{lem}\label{l: |hat chi_PP,k'|_c_1 to 1}
	$\|\hat\chi_{\PP,k}'\|_{c_1}\to1$ as $s\to\infty$ for every $k\in2\N$.
\end{lem}

\begin{proof}
	We proceed by induction on $k$. For $k=0$, take some $\epsilon>0$ and $C>0$ satisfying~\eqref{lambda_k le ..., case of PP}. By Propositions~\ref{p: PP_0}~\eqref{i: the spectrum of PP_0} and~\ref{p: PP}~\eqref{i: ... le lambda_k ge ..., case of PP}, and Remark~\ref{r: P_0, Q_0, P, Q, F, G}~\eqref{i: sD(PP^1/2) = sD(PP_0^1/2)},
		\begin{multline*}
			(1+2\sigma)(s+\xi\epsilon s^u)+\xi Cs^u\ge\lambda_0
			=\langle\PP^{1/2}\hat\chi_0,\PP^{1/2}\hat\chi_0\rangle_{c_1}
			>\langle\PP_0^{1/2}\hat\chi_0,\PP_0^{1/2}\hat\chi_0\rangle_{c_1}\\
			=\langle\PP_0^{1/2}\hat\chi_0',\PP_0^{1/2}\hat\chi_0'\rangle_{c_1}
			+\langle\PP_0^{1/2}\hat\chi_0''',\PP_0^{1/2}\hat\chi_0'''\rangle_{c_1}\\
			\ge(1+2\sigma)s\,\|\hat\chi_0'\|_{c_1}^2+(5+2\sigma)s\,\|\hat\chi_0'''\|_{c_1}^2
			=(1+2\sigma)s+4s\,\|\hat\chi_0'''\|_{c_1}^2\;,
		\end{multline*}
	giving
		\[
			\|\hat\chi_0'''\|_{c_1}^2<\frac{((1+2\sigma)\epsilon+C)\xi}{4s^{1-u}}\to0
		\]
	as $s\to\infty$, and therefore $\|\hat\chi_0'\|_{c_1}^2\to1$.
	
	Now, take any even integer $k>0$ and suppose that the result holds for all even indices $<k$. This yields $\|\hat\chi_k''\|_{c_1}\to0$ as $s\to\infty$. Thus, given any $\delta>0$, we have $\|\hat\chi_k''\|_{c_1}^2<\delta/k$ for $s$ large enough. Take some $\epsilon>0$ and $C>0$ satisfying~\eqref{lambda_k le ..., case of PP}.  By Propositions~\ref{p: PP_0}~\eqref{i: the spectrum of PP_0} and~\ref{p: PP}~\eqref{i: ... le lambda_k ge ..., case of PP}, and Remark~\ref{r: P_0, Q_0, P, Q, F, G}~\eqref{i: sD(PP^1/2) = sD(PP_0^1/2)},
		\begin{multline*}
			(2k+1+2\sigma)(s+\xi\epsilon s^u)+\xi Cs^u\ge\lambda_k
			=\langle\PP^{1/2}\hat\chi_k,\PP^{1/2}\hat\chi_k\rangle_{c_1}
			>\langle\PP_0^{1/2}\hat\chi_k,\PP_0^{1/2}\hat\chi_k\rangle_{c_1}\\
			\begin{aligned}
				&=\langle\PP_0^{1/2}\hat\chi_k',\PP_0^{1/2}\hat\chi_k'\rangle_{c_1}
				+\langle\PP_0^{1/2}\hat\chi_k'',\PP_0^{1/2}\hat\chi_k''\rangle_{c_1}
				+\langle\PP_0^{1/2}\hat\chi_k''',\PP_0^{1/2}\hat\chi_k'''\rangle_{c_1}\\
				&\ge(2k+1+2\sigma)s\,\|\hat\chi_k'\|_{c_1}^2
				+(1+2\sigma)s\,\|\hat\chi_k''\|_{c_1}^2
				+(2k+5+2\sigma)s\,\|\hat\chi_k'''\|_{c_1}^2\\
				&=(1+2\sigma)s+2ks(\|\hat\chi_k'\|_{c_1}^2+\|\hat\chi_k'''\|_{c_1}^2)
				+4s\,\|\hat\chi_k'''\|_{c_1}^2\\
				&>(1+2\sigma)s+2ks(1-\delta/k)+4s\,\|\hat\chi_k'''\|_{c_1}^2\;,
			\end{aligned}
		\end{multline*}
	giving
		\[
			\|\hat\chi_k'''\|_{c_1}^2<\frac{((2k+1+2\sigma)\epsilon+C)\xi}{4s^{1-u}}+\frac{\delta}{2}<\delta
		\]
	for $s$ large enough. Thus $\|\hat\chi_k'''\|_{c_1}^2\to0$ as $s\to\infty$, and the result follows.
\end{proof}

\begin{cor}\label{c: h hat chi_PP,0}
  If $h$ is a bounded measurable function on $\R_+$ such that $h(\rho)\to1$ as $\rho\to0$, then $\langle h\hat\chi_{\PP,0},\hat\chi_{\PP,0}\rangle_{c_1}\to1$ as $s\to\infty$.
\end{cor}

\begin{proof}
  	This follows from Lemma~\ref{l: |hat chi_PP,k'|_c_1 to 1} and Remark~\ref{r: P_0, Q_0, P, Q, F, G}~\eqref{i: langle h chi_0,chi_0 rangle_c_1 to 1}.
\end{proof}

Similar results hold for $\QQ$ and $\WW$, but they are omitted because they are not used.

\section{Two simple types of elliptic complexes}\label{s: 2 simple types of elliptic complexes}

Here, we study two simple elliptic complexes on $\R_+$, which will show up in a direct sum splitting of the rel-local model of Witten's perturbation (Section~\ref{s: splitting}).

\subsection{An elliptic complex of length one}\label{ss: complex 1}

Consider the standard metric on $\R_+$. Let $E$ be the graded Riemannian/Hermitian vector bundle over $\R_+$ whose nonzero terms are $E_0$ and $E_1$, which are real/complex trivial line bundles equipped with the standard Riemannian/Hemitian metrics. Thus
  \[
    C^\infty(E_0)\equiv C^\infty_+\equiv C^\infty(E_1)\;,\quad L^2(E_0)\equiv L_+^2\equiv L^2(E_1)\;,
  \]
where real-/complex-valued functions are considered in $C^\infty_+$ and $L_+^2$. For any fixed $s>0$ and $\kappa\in\R$, let 
\begin{center}
  \begin{picture}(129,26) 
    \put(0,10){$C^\infty(E_0)$}
    \put(94,10){$C^\infty(E_1)$}
    \put(62,19){\Small$d$}
    \put(62,0){\Small$\delta$}
    \put(41,14){\vector(1,0){47}}
    \put(88,11){\vector(-1,0){47}}
  \end{picture}
\end{center}
be the differential operators defined by 
  \[
    d=\textstyle{\frac{d}{d\rho}}-\kappa\rho^{-1}\pm s\rho\;,\quad
    \delta=-\textstyle{\frac{d}{d\rho}}-\kappa\rho^{-1}\pm s\rho\;.
  \]
It is easy to check that $(E,d)$ is an elliptic complex, and that\footnote{The superindex $\dag$ is used to denote the formal adjoint.} $\delta=d^\dag$.

\subsubsection{Self-adjoint operators defined by the Laplacian}\label{sss: complex 1, Laplacian}

By~\eqref{[d/d rho,rho^a right]}, the homogeneous components of $\D$ (or $\D^\pm$) are:
  \begin{gather}
    \D_0=H+\kappa(\kappa-1)\rho^{-2}\mp s(1+2\kappa)\;,\label{complex 1, D_0}\\
    \D_1=H+\kappa(\kappa+1)\rho^{-2}\pm s(1-2\kappa)\;,\label{complex 1, D_1}
  \end{gather}
where $H$ is the harmonic oscillator on $C^\infty_+$ defined with the constant $s$. Then $\D_0$ and $\D_1$ are like $P_0$ and $Q_0$ in~\eqref{P_0, Q_0}, with $c_1=0=d_1$, plus a constant. Then, by Propositions~\ref{p: PP_0} and~\ref{p: QQ_0}, $\D_0$ and $\D_1$ define the self-adjoint operators $\AA_i$ and $\BB_i$ in $L^2_+$ indicated in Table~\ref{table: 1, AA_i, BB_i}, where the conditions come from~\eqref{sigma > -1/2} and~\eqref{tau > -3/2}. The notation $\AA_i^\pm$ and $\BB_i^\pm$ may be used as well to specify that these operators are defined by $\D^\pm_0$ and $\D^\pm_1$. In these cases, we have $c_1=d_1=0$, and therefore $\sigma=a$ and $\tau=b$, which are given by~\eqref{a} and~\eqref{b}.

\begin{table}[h]
\renewcommand{\arraystretch}{1.3}
\begin{tabular}{cc|c|c|l|}
\cline{3-5}
&& $\sigma$ & $\tau$ & Condition \\
\hline
\multicolumn{1}{|c|}{\multirow{2}{*}{$\D_0$}} & \multicolumn{1}{c|}{$\AA_1$} & $\kappa$ && $\kappa>-\frac{1}{2}$  \\
\cline{2-5}
\multicolumn{1}{|c|}{} & \multicolumn{1}{c|}{$\AA_2$} & $1-\kappa$ && $\kappa<\frac{3}{2}$ \\
\hline
\multicolumn{1}{|c|}{\multirow{2}{*}{$\D_1$}} & \multicolumn{1}{c|}{$\BB_1$} && $\kappa$ & $\kappa>-\frac{3}{2}$ \\
\cline{2-5}
\multicolumn{1}{|c|}{} & \multicolumn{1}{c|}{$\BB_2$} && $-1-\kappa$ & $\kappa<\frac{1}{2}$ \\
\hline
\end{tabular}
\vspace{1mm}
\caption{Self-adjoint operators defined by $\D_0$ and $\D_1$}
\label{table: 1, AA_i, BB_i}
\end{table}

There are the following overlaps in Table~\ref{table: 1, AA_i, BB_i}:
	\begin{itemize}
	
		\item Both $\AA_1$ and $\AA_2$ are defined if $-\frac{1}{2}<\kappa<\frac{3}{2}$, and they are equal just when $\kappa=\frac{1}{2}$.
		
		\item Both $\BB_1$ and $\BB_2$ are defined if $-\frac{3}{2}<\kappa<\frac{1}{2}$, and they are equal just when $\kappa=-\frac{1}{2}$.
		
	\end{itemize}
The cores of $\AA_i$ and $\BB_i$, given by Propositions~\ref{p: PP_0} and~\ref{p: QQ_0}, will be denoted by $\EE_i^0$ and $\EE_i^1$, respectively. Note that the graded subspace $\EE_i=\EE_i^0\oplus\EE_i^1$ of $C^\infty(E)\cap L^2(E)$, whenever defined, is preserved by $D=d+\delta$. Propositions~\ref{p: PP_0} and~\ref{p: QQ_0} also describe the spectra of $\AA_i$ and $\BB_i$:
  	\begin{itemize}
    
    		\item The spectrum of $\AA_1$ consists of the eigenvalues
  			\begin{equation}\label{eigenvalues, 1, AA_1}
    				(2k+(1\mp1)(1+2\kappa))s\quad(k\in2\N)
  			\end{equation}
		of multiplicity one. 
    
    		\item The spectrum of $\AA_2$ consists of the eigenvalues
  			\begin{equation}\label{eigenvalues, 1, AA_2}
    				(2k+4-(1\pm1)(1+2\kappa))s\quad(k\in2\N)
  			\end{equation}
		of multiplicity one.
		
		\item The spectrum of $\BB_1$ consists of the eigenvalues
  			\begin{equation}\label{eigenvalues, 1, BB_1}
    				(2k+2+(1\mp1)(-1+2\kappa))s\quad(k\in2\N+1)
  			\end{equation}
		of multiplicity one.
    
    		\item The spectrum of $\BB_2$ consists of the eigenvalues
  			\begin{equation}\label{eigenvalues, 1, BB_2}
    				(2k-2-(1\pm1)(-1+2\kappa))s\quad(k\in2\N+1)
  			\end{equation}
of multiplicity one.
    
  	\end{itemize}
These eigenvalues have normalized eigenfunctions $\chi_k$, defined for the corresponding values of $a=\sigma$ and $b=\tau$. For $\AA^+_1$,~\eqref{eigenvalues, 1, AA_1} becomes $2ks$. For $\AA^-_1$,~\eqref{eigenvalues, 1, AA_1} is $2(k+1+2\kappa)s$. For $\AA^+_2$,~\eqref{eigenvalues, 1, AA_2} becomes $2(k+1-2\kappa)s$. For $\AA^-_2$,~\eqref{eigenvalues, 1, AA_2} is $2(k+2)s$. For $\BB^+_1$,~\eqref{eigenvalues, 1, BB_1} is $2(k+1)s$. For $\BB^-_1$,~\eqref{eigenvalues, 1, BB_1} becomes $2(k+2\kappa)s$. For $\BB^+_2$,~\eqref{eigenvalues, 1, BB_2} is $2(k-2\kappa)s$. For $\BB^-_2$,~\eqref{eigenvalues, 1, BB_2} becomes $2(k-1)s$. Using this, we get the information about the sign of the eigenvalues of $\AA_i$ and $\BB_i$ given in Table~\ref{table: 1, sign eigenvalues AA_i and BB_i}. In the tables, grey color is used for cases that will be disregarded later (for instance, if there may exist some negative eigenvalue), and a question mark is used for unknown information.

\begin{table}[h]
\renewcommand{\arraystretch}{1.3}
\begin{tabular}{cl|l|ccl|l|}
\cline{3-3}\cline{7-7}
&& Sign of eigenvalues & \quad &&& Sign of eigenvalues \\
\cline{1-3}\cline{5-7}
\multicolumn{2}{|c|}{\multirow{2}{*}{$\AA^+_1$}} & $0$\quad if $k=0$ & \quad & 
\multicolumn{2}{|c|}{$\BB^+_1$} & $+$\quad $\forall k\in2\N+1$\\
\cline{5-7}
\multicolumn{2}{|c|}{} & $+$\quad if $k\ge2$ even & \quad & 
\multicolumn{1}{|c}{\multirow{6}{*}{$\BB^-_1$}} & \multicolumn{1}{|l|}{$\kappa>-\frac{1}{2}$} & $+$\quad $\forall k\in2\N+1$ \\
\cline{1-3}\cline{6-7}
\multicolumn{2}{|c|}{$\AA^-_1$} & $+$\quad $\forall k\in2\N$ & \quad & 
\multicolumn{1}{|c}{} & \multicolumn{1}{|l|}{\multirow{2}{*}{$\kappa=-\frac{1}{2}$}} & $0$\quad if $k=1$ \\
\cline{1-3}
\multicolumn{1}{|c}{\multirow{6}{*}{$\AA^+_2$}} & \multicolumn{1}{|l|}{\multirow{3}{*}{\color{lightgray} $\kappa>\frac{1}{2}$}} &  \color{lightgray} $-$\quad if $k<2\kappa-1$ & \quad & 
\multicolumn{1}{|c}{} & \multicolumn{1}{|l|}{} & $+$\quad if $k\ge3$ odd \\
\cline{6-7}
\multicolumn{1}{|c}{} & \multicolumn{1}{|l|}{} & $0$\quad if $k=2\kappa-1$ & \quad & 
\multicolumn{1}{|c}{} & \multicolumn{1}{|l|}{\multirow{3}{*}{\color{lightgray} $\kappa<-\frac{1}{2}$}} &\color{lightgray}$-$\quad if $k<-2\kappa$ \\
\multicolumn{1}{|c}{} & \multicolumn{1}{|l|}{} & $+$\quad if $k>2\kappa-1$ & \quad & 
\multicolumn{1}{|c}{} & \multicolumn{1}{|l|}{} & $0$\quad if $k=-2\kappa$ \\
\cline{2-3}
\multicolumn{1}{|c}{} & \multicolumn{1}{|l|}{\multirow{2}{*}{$\kappa=\frac{1}{2}$}} & $0$\quad if $k=0$ & \quad & 
\multicolumn{1}{|c}{} & \multicolumn{1}{|l|}{} & $+$\quad if $k>-2\kappa$ \\
\cline{5-7}
\multicolumn{1}{|c}{} & \multicolumn{1}{|l|}{} & $+$\quad if $k\ge2$ even & \quad & 
\multicolumn{2}{|c|}{$\BB^+_2$} & $+$\quad $\forall k\in2\N+1$ \\
\cline{2-3}\cline{5-7}
\multicolumn{1}{|c}{} & \multicolumn{1}{|l|}{$\kappa<\frac{1}{2}$} & $+$\quad $\forall k\in2\N$ & \quad & 
\multicolumn{2}{|c|}{\multirow{2}{*}{$\BB^-_2$}} & $0$\quad if $k=1$ \\
\cline{1-3}
\multicolumn{2}{|c|}{$\AA^-_2$} & $+$\quad $\forall k\in2\N$ & \quad & 
\multicolumn{2}{|c|}{} & $+$\quad if $k\ge3$ odd \\
\cline{1-3}\cline{5-7}
\end{tabular}
\vspace{1mm}
\caption{Sign of the eigenvalues of $\AA_i$ and $\BB_i$}
\label{table: 1, sign eigenvalues AA_i and BB_i}
\end{table}

\subsubsection{Laplacians of the maximum/minimum i.b.c.}\label{sss: complex 1, max/min i.b.c.}



\begin{prop}[{\cite[Proposition~8.4]{AlvCalaza2017}}]\label{p: 1}
    Table~\ref{table: Delta_max/min, 1} describes $\D_{\text{\rm max/min}}$.
\end{prop}

\begin{table}[h]
\renewcommand{\arraystretch}{1.3}
\begin{tabular}{c|c|c|c|c|}
\cline{2-5}
& $\D_{\text{\rm max},0}$ & $\D_{\text{\rm min},0}$ & $\D_{\text{\rm max},1}$ & $\D_{\text{\rm min},1}$ \\
\hline
\multicolumn{1}{|c|}{$\kappa\ge\frac{1}{2}$} & \multicolumn{2}{c|}{$\AA_1$} & \multicolumn{2}{c|}{$\BB_1$} \\
\hline
\multicolumn{1}{|c|}{$|\kappa|<\frac{1}{2}$} & $\AA_1$ & $\AA_2$ & $\BB_1$ & $\BB_2$ \\
\hline
\multicolumn{1}{|c|}{$\kappa\le-\frac{1}{2}$} & \multicolumn{2}{c|}{$\AA_2$} & \multicolumn{2}{c|}{$\BB_2$} \\
\hline
\end{tabular}
\vspace{1mm}
\caption{Description of $\D_{\text{\rm max/min}}$}
\label{table: Delta_max/min, 1}
\end{table}

\begin{rem}\label{r: 1}
	\begin{enumerate}[(i)]
		
		\item\label{i: xi_n} In \cite{AlvCalaza2017}, the proof of Proposition~\ref{p: 1} uses the following property \cite[Lemma~8.5]{AlvCalaza2017}. Suppose that either $\theta>\frac{1}{2}$, or $\theta=\frac{1}{2}=\kappa$ \upn{(}respectively, $\theta=\frac{1}{2}=-\kappa$\upn{)}. Then, for every $\xi\in\rho^\theta\SS_{\text{\rm ev},+}$, considered as subspace of $C^\infty(E_0)$ (respectively, $C^\infty(E_1)$), there is a sequence $(\xi_n)$ in $C^\infty_0(E_0)$ (respectively, $C^\infty_0(E_1)$), independent of $\kappa$, such that $\lim_n\xi_n=\xi$ in $L^2(E_0)$ (respectively, $L^2(E_1)$) and $\lim_nd\xi_n=d\xi$ in $L^2(E_1)$ (respectively, $\lim_n\delta\xi_n=\delta\xi$ in $L^2(E_0)$). In particular, $\rho^\theta\SS_{\text{\rm ev},+}$ is contained in $\sD(d_{\text{\rm min}})$ (respectively, $\sD(\delta_{\text{\rm min}})$). Moreover, according to the proof of \cite[Lemma~8.5]{AlvCalaza2017}, given $0<a<b$,  we can take $\xi_n=\alpha_n\xi$ for some $\alpha_n\in C^\infty_+$ satisfying $\chi_{[\frac{b}{n},na]}\le\alpha_n\le\chi_{[\frac{a}{n},nb]}$, where $\chi_S$ denotes the characteristic function of every subset $S\subset\R_+$.
		
		\item\label{i: EE_i} $\EE_i^0$ (respectively, $\EE_i^1$) is also a core of $d_{\text{\rm max/min}}$ (respectively, $\delta_{\text{\rm min/max}}$) when $\D_{\text{\rm max/min},0}=\AA_i$ (respectively, $\D_{\text{\rm max/min},1}=\BB_i$).
		
	\end{enumerate}
\end{rem}

\subsection{An elliptic complex of length two}\label{ss: complex 2}

Consider again the standard metric on $\R_+$. Let $F$ be the graded Riemannian/Hermitian vector bundle over $\R_+$ whose nonzero terms are $F_0$, $F_1$ and $F_2$, which are trivial real/complex vector bundles of ranks $1$, $2$ and $1$, respectively, equipped with the standard Riemannian/Hermitian metrics. Thus
  \begin{gather*}
    C^\infty(F_0)\equiv C^\infty_+\equiv C^\infty(F_2)\;,\quad C^\infty(F_1)\equiv C^\infty_+\oplus C^\infty_+\;,\\
    L^2(F_0)\equiv L^2_+\equiv L^2(F_2)\;,\quad L^2(F_1)\equiv L^2_+\oplus L^2_+\;,
  \end{gather*}
where real-/complex-valued functions are considered in $C^\infty_+$ and $L^2_+$. Fix $s,\mu>0$, $0<u<1$ and $\kappa\in\R$. Let
\begin{center}
  \begin{picture}(286,46) 
    \put(0,21){$C^\infty(F_0)$}
    \put(124,21){$C^\infty(F_1)$}
    \put(248,21){$C^\infty(F_2)$}
    \put(56,35){\Small$d_0\equiv
      \begin{pmatrix}
        d_{0,1}\\
        d_{0,2}
      \end{pmatrix}$}
    \put(48,8){\Small$\delta_0\equiv
      \begin{pmatrix}
        \delta_{0,1} &
        \delta_{0,2}
      \end{pmatrix}$}
    \put(171,35){\Small$d_1\equiv
      \begin{pmatrix}
        d_{1,1} &
        d_{1,2}
      \end{pmatrix}$}
    \put(183,8){\Small$\delta_1\equiv
      \begin{pmatrix}
        \delta_{1,1}\\
        \delta_{1,2}
      \end{pmatrix}$}
    \put(41,25){\vector(1,0){77}}
    \put(118,22){\vector(-1,0){77}}
    \put(166,25){\vector(1,0){77}}
    \put(243,22){\vector(-1,0){77}}
  \end{picture}
\end{center}
be the differential operators defined by
  \begin{alignat*}{2}
    d_{0,1}&=\mu\rho^{-u}\;,&\quad 
    d_{0,2}&=\textstyle{\frac{d}{d\rho}}-(\kappa+u)\rho^{-1}\pm s\rho\;,\\
    d_{1,1}&=\textstyle{\frac{d}{d\rho}}-\kappa\rho^{-1}\pm s\rho\;,&\quad 
    d_{1,2}&=-\mu\rho^{-u}\;,\\
    \delta_{0,1}&=\mu\rho^{-u}\;,&\quad
    \delta_{0,2}&=-\textstyle{\frac{d}{d\rho}}-(\kappa+u)\rho^{-1}\pm s\rho\;,\\
    \delta_{1,1}&=-\textstyle{\frac{d}{d\rho}}-\kappa\rho^{-1}\pm s\rho\;,&\quad 
    \delta_{1,2}&=-\mu\rho^{-u}\;.
  \end{alignat*}
Observe that $\delta_0=d_0^\dag$ and $\delta_1=d_1^\dag$. We may also use the more explicit notation $d_r^\pm$, $\delta_r^\pm$, $d_{r,i}^\pm$ and $\delta_{r,i}^\pm$. A direct computation shows that $d_0$ and $d_1$ define an elliptic complex $(F,d)$ of length two. Note that, by~\eqref{[d/d rho,rho^a right]},
	\begin{equation}\label{d_1,1 = rho^-u d_0,2 rho^u}
		d_{1,1}=\rho^{-u}\,d_{0,2}\,\rho^u\;,\quad\delta_{0,2}=\rho^{-u}\,\delta_{1,1}\,\rho^u\;.
	\end{equation}

\subsubsection{Self-adjoint operators defined by the Laplacian}\label{sss: complex 2, Laplacian}

By~\eqref{[d/d rho,rho^a right]}, the homogeneous components of the corresponding Laplacian $\D$ (or $\D^\pm$)  are given by
  	\begin{align*}
    		\D_0&=H+(\kappa+u)(\kappa+u-1)\rho^{-2}+\mu^2\rho^{-2u}\mp s(1+2(\kappa+u))\;,\\
    		\D_2&=H+\kappa(\kappa+1)\rho^{-2}+\mu^2\rho^{-2u}\pm s(1-2\kappa)\;,\\
    		\D_1&=
      			\begin{pmatrix}
        				\D_{1,1} & -2\mu u\rho^{-u-1}\\
        				-2\mu u\rho^{-u-1} & \D_{1,2}
      			\end{pmatrix}\;,\\
      		\D_{1,1}&=H+\kappa(\kappa-1)\rho^{-2}+\mu^2\rho^{-2u}\mp s(1+2\kappa)\;,\\
    		\D_{1,2}&=H+(\kappa+u)(\kappa+u+1)\rho^{-2}+\mu^2\rho^{-2u}\pm s(1-2(\kappa+u))\;.
  	\end{align*}
(We may also use~\eqref{complex 1, D_0} and~\eqref{complex 1, D_1} to compute easily some parts of the above components of $\D$.) The operators $\D_0$, $\D_2$, $\D_{1,1}$  and $\D_{1,2}$ are like $P$ and $Q$ in~\eqref{P, Q}, with $c_1=0=d_1$, plus a constant term.  Write $\D_1=U\mp sV$, where
	\begin{equation}\label{V}
		V=
			\begin{pmatrix}
          			1+2\kappa & 0\\
          			0 & -1+2(\kappa+u)
        			\end{pmatrix}\;.
	\end{equation}
Then, by Propositions~\ref{p: PP},~\ref{p: QQ} and~\ref{p: WW}, and Remark~\ref{r: P_0, Q_0, P, Q, F, G}~\eqref{i: PP = ol P}, $\D_0$, $\D_2$ and $\D_1$ define the self-adjoint operators $\PP_i$ and $\QQ_j$ in $L^2_+$, and $\WW_{i,j}$ in $L^2_+\oplus L^2_+$, indicated in Table~\ref{table: 2, PP_i, QQ_j, WW_i,j}, where the conditions come from~\eqref{sigma > u - 1/2},~\eqref{tau > u - 3/2},~\eqref{theta > -1/2},~\eqref{WW, sigma = theta ne tau},~\eqref{WW, sigma ne theta = tau},~\eqref{WW, sigma ne theta = tau+1} and~\eqref{WW, sigma ne theta ne tau}. The notation $\PP_i^\pm$, $\QQ_j^\pm$ and $\WW_{i,j}^\pm$ may be used as well to specify that these operators are defined by $\D^\pm_0$,  $\D^\pm_2$ and $\D^\pm_1$. Note that $v=u$ for all $\WW_{i,j}$.  The cores of $\PP_i^{1/2}$, $\QQ_j^{1/2}$ and $\WW_{i,j}^{1/2}$, given by Propositions~\ref{p: PP},~\ref{p: QQ} and~\ref{p: WW}, will be denoted by $\FF_i^0$, $\FF_j^2$ and $\FF_{i,j}^1=\FF_i^{1,1}\oplus\FF_j^{1,2}$, respectively.
	
\begin{rem}\label{r: FF_i^0 oplus FF_i,j^1 oplus FF_j^2}
	In contrast to $\EE_i$ in Section~\ref{sss: complex 1, Laplacian}, note that the graded subspace $\FF_i^0\oplus\FF_{i,j}^1\oplus\FF_j^2$ of $C^\infty(F)\cap L^2(F)$, whenever defined, is not preserved by $D=d+\delta$. For instance, it is preserved by $d$ but not by $\delta$ when $i=j=1$, and it is preserved by $\delta$ but not by $d$ when $i=j=2$. \
\end{rem}

\begin{table}[h]
\renewcommand{\arraystretch}{1.3}
\begin{tabular}{cc|c|c|c|l|}
\cline{3-6}
&& $\sigma$ & $\tau$ & $\theta$ & Condition \\
\hline
\multicolumn{1}{|c|}{\multirow{2}{*}{$\D_0$}} & \multicolumn{1}{c|}{$\PP_1$} & $\kappa+u$ &&& $\kappa>-\frac{1}{2}$ \\
\cline{2-6}
\multicolumn{1}{|c|}{} & \multicolumn{1}{c|}{$\PP_2$} & $1-\kappa-u$ &&& $\kappa<\frac{3}{2}-2u$ \\
\hline
\multicolumn{1}{|c|}{\multirow{2}{*}{$\D_2$}} & \multicolumn{1}{c|}{$\QQ_1$} && $\kappa$ && $\kappa>u-\frac{3}{2}$ \\
\cline{2-6}
\multicolumn{1}{|c|}{} & \multicolumn{1}{c|}{$\QQ_2$} && $-1-\kappa$ && $\kappa<\frac{1}{2}-u$ \\
\hline
\multicolumn{1}{|c|}{\multirow{5}{*}{$\D_1$}} & \multicolumn{1}{c|}{$\WW_{1,1}$} & $\kappa$ & $\kappa+u$ & $\kappa$ & $\kappa>u-\frac{1}{2}$ \\
\cline{2-6}
\multicolumn{1}{|c|}{} & \multicolumn{1}{c|}{$\WW_{2,2}$} & $1-\kappa$ & $-1-\kappa-u$ & $-\kappa-u$ & $\kappa<\frac{1}{2}-2u$ \\
\cline{2-6}
\multicolumn{1}{|c|}{} & \multicolumn{1}{c|}{\color{lightgray}$\not\exists\ \WW_{1,2}$} & \color{lightgray}$\kappa$ & \color{lightgray}$-1-\kappa-u$ & \color{lightgray}$-\frac{1}{2}-u$ & \color{lightgray}Impossible \\
\cline{2-6}
\multicolumn{1}{|c|}{} & \multicolumn{1}{c|}{$\WW_{2,1}$} & $1-\kappa$ & $\kappa+u$ & $\frac{1}{2}$ & $-1-\frac{u}{2}<\kappa<1-\frac{u}{2}$ \\
\hline
\end{tabular}
\vspace{1mm}
\caption{Self-adjoint operators defined by $\D_0$, $\D_2$ and $\D_1$}
\label{table: 2, PP_i, QQ_j, WW_i,j}
\end{table}

Let us explain the contents of Table~\ref{table: 2, PP_i, QQ_j, WW_i,j}. Since $c_1=d_1=0$, we have $\sigma=a$ and $\tau=b$, which are given by~\eqref{a} and~\eqref{b}. Moreover $\sigma$, $\tau$ and $u$ determine $\theta$ in Table~\ref{table: 2, PP_i, QQ_j, WW_i,j} so that $U$ is of the form~\eqref{W} because $2\theta-\sigma-\tau=-u$. Let us check the conditions written in this table, which are given by the hypothesis of Propositions~\ref{p: PP}--\ref{p: WW}. For $\PP_i$ and $\QQ_j$, only~\eqref{sigma > u - 1/2} and~\eqref{tau > u - 3/2} are required. For $\WW_{i,j}$, we also require~\eqref{theta > -1/2}, and the hypothesis~\eqref{i: sigma = theta ne tau}--\eqref{i: sigma ne theta ne tau} of Proposition~\ref{p: WW}, obtaining the following:
	\begin{itemize}
	
		\item For $\WW_{1,1}$, we have $\sigma=\theta\ne\tau$ and $\tau-\sigma=u\not\in-\N$. Thus~\eqref{i: sigma = theta ne tau} applies in this case. Note that~\eqref{sigma > u - 1/2},~\eqref{tau > u - 3/2} and~\eqref{theta > -1/2} mean $\kappa>u-\frac{1}{2}$. Then~\eqref{WW, sigma = theta ne tau} holds because $0<u<1$ and $\kappa>u-\frac{1}{2}$. So~\eqref{i: sigma = theta ne tau} is satisfied.

		\item For $\WW_{2,2}$, we have $\sigma\ne\theta=\tau+1$ and $\sigma-\tau-1=1+u\not\in-\N$. Thus~\eqref{i: sigma ne theta = tau+1} applies in this case. Now,~\eqref{sigma > u - 1/2},~\eqref{tau > u - 3/2} and~\eqref{theta > -1/2} mean $\kappa<\frac{1}{2}-2u$. Then~\eqref{WW, sigma ne theta = tau+1} holds because $0<u<1$ and $\kappa<\frac{1}{2}-2u$. So~\eqref{i: sigma ne theta = tau+1} is satisfied.

		\item There is no $\WW_{1,2}$ because $\theta<-\frac{1}{2}$ in that case.

		\item For $\WW_{2,1}$,~\eqref{sigma > u - 1/2},~\eqref{tau > u - 3/2} and~\eqref{theta > -1/2} mean $-\frac{3}{2}<\kappa<\frac{3}{2}-u$, and we have the following possibilities: 
			\begin{itemize}
	
				\item The case $\sigma=\theta=\tau$ is not possible because $u\ne0$. 
		
				\item The case $\sigma=\theta\ne\tau$ happens when $\kappa=\frac{1}{2}$. Then $\sigma=\frac{1}{2}$ and $\tau=\frac{1}{2}+u$, obtaining $\tau-\sigma=u\not\in-\N$. Thus~\eqref{i: sigma = theta ne tau} applies in this case. Moreover~\eqref{WW, sigma = theta ne tau} holds because $0<u<1$. So~\eqref{i: sigma = theta ne tau} is satisfied.
		
				\item The case $\sigma\ne\theta=\tau$ happens when $\kappa=\frac{1}{2}-u$. Then $\sigma=\frac{1}{2}+u$ and $\tau=\frac{1}{2}$, obtaining $\sigma-\tau=u\not\in-\N$. Thus~\eqref{i: sigma ne theta = tau} applies in this case. Moreover~\eqref{WW, sigma ne theta = tau} holds because $0<u<1$. Hence~\eqref{i: sigma ne theta = tau} is satisfied.
		
				\item The case $\sigma\ne\theta=\tau+1$ happens when $\kappa=-\frac{1}{2}-u$. Then $\sigma=\frac{3}{2}+u$ and $\tau=-\frac{1}{2}$, obtaining $\sigma-\tau-1=1+u\not\in-\N$. Thus~\eqref{i: sigma ne theta = tau+1} applies in this case. Moreover~\eqref{WW, sigma ne theta = tau+1} holds because $0<u<1$. Hence~\eqref{i: sigma ne theta = tau+1} is satisfied.
		
				\item Finally, assume that $\sigma\ne\theta\ne\tau$.  The condition $\sigma-\theta,\tau-\theta\not\in-\N$ means that $\kappa\not\in(\frac{1}{2}+\N)\cup(\frac{1}{2}-u-\N)$, which in turn means that $\kappa\ne\frac{1}{2},\frac{1}{2}-u,-\frac{1}{2}-u$ because $-\frac{3}{2}<\kappa<\frac{3}{2}-u$. But $\sigma=\theta$ if $\kappa=\frac{1}{2}$, $\tau=\theta$ if $\kappa=\frac{1}{2}-u$, and $\theta=\tau+1$ if $\kappa=-\frac{1}{2}-u$, as we have seen in the previous cases. So $\sigma-\theta,\tau-\theta\not\in-\N$, and~\eqref{i: sigma ne theta ne tau} applies in this case. Moreover, since $0<u<1$,~\eqref{WW, sigma ne theta ne tau} holds just when $-1-\frac{u}{2}<\kappa<1-\frac{u}{2}$. Thus~\eqref{i: sigma ne theta ne tau} is satisfied assuming the stated conditions on $\kappa$.	
			
			\end{itemize}
		Therefore $\WW_{2,1}$ is defined in one of the above ways if $-1-\frac{u}{2}<\kappa<1-\frac{u}{2}$.
	\end{itemize}
There are the following overlaps of the conditions in Table~\ref{table: 2, PP_i, QQ_j, WW_i,j}:
	\begin{itemize}
	
		\item Both $\PP_1$ and $\PP_2$ are defined for $-\frac{1}{2}<\kappa<\frac{3}{2}-2u$, and $\PP_1=\PP_2$ just when $\kappa=\frac{1}{2}-u$.
		
		\item Both $\QQ_1$ and $\QQ_2$ are defined for $u-\frac{3}{2}<\kappa<\frac{1}{2}-u$, and $\QQ_1=\QQ_2$ just when $\kappa=-\frac{1}{2}$.
		
		\item Both $\WW_{1,1}$ and $\WW_{2,2}$ are defined for $u-\frac{1}{2}<\kappa<\frac{1}{2}-2u$ (if $u<\frac{1}{3}$), but $\WW_{1,1}\ne\WW_{2,2}$ for all such $\kappa$.
		
		\item Both $\WW_{1,1}$ and $\WW_{2,1}$ are defined for $u-\frac{1}{2}<\kappa<1-\frac{u}{2}$, and $\WW_{1,1}=\WW_{2,1}$ just when $\kappa=\frac{1}{2}$.
		
		\item Both $\WW_{2,2}$ and $\WW_{2,1}$ are defined for $-1-\frac{u}{2}<\kappa<\frac{1}{2}-2u$, and $\WW_{2,2}=\WW_{2,1}$ just when $\kappa=-\frac{1}{2}-u$.
	
	\end{itemize}
Propositions~\ref{p: PP},~\ref{p: QQ} and~\ref{p: WW} also give the following spectral estimates, for all $\epsilon>0$: 
  	\begin{itemize}
    
    		\item The spectrum of $\PP_1$ consists of eigenvalues $\lambda_0\le\lambda_2\le\cdots$, taking multiplicity into account, such that there are some $D=D(\kappa,u)>0$ and $C=C(\epsilon,\kappa,u)>0$ so that, for all $k\in2\N$,
  			\begin{align}
    				\lambda_k&\ge(2k+(1\mp1)(1+2(\kappa+u)))s+\mu^2Ds^u(k+1)^{-u}\;,
    				\label{eigenvalues ge ..., PP_1}\\
    				\lambda_k&\le(2k+(1\mp1)(1+2(\kappa+u)))s\notag\\
    				&\phantom{=\text{}}\text{}+(2k+1+2(\kappa+u))\mu^2\epsilon s^u+\mu^2Cs^u\;.
    				\label{eigenvalues le ..., PP_1}
  			\end{align}
		The first term of the right-hand side of~\eqref{eigenvalues ge ..., PP_1} and~\eqref{eigenvalues le ..., PP_1} for $\PP^+_1$ and $\PP^-_1$ is $2ks$ and $2(k+1+2(\kappa+u))s$, respectively.
    
    		\item The spectrum of $\PP_2$ consists of eigenvalues $\lambda_0\le\lambda_2\le\cdots$, taking multiplicity into account, such that there are some $D=D(\kappa,u)>0$ and $C=C(\epsilon,\kappa,u)>0$ so that, for all $k\in2\N$,
  			\begin{align}
    				\lambda_k&\ge(2k+4-(1\pm1)(1+2(\kappa+u)))s+\mu^2Ds^u(k+1)^{-u}\;,\label{eigenvalues ge ..., PP_2}\\
    				\lambda_k&\le(2k+4-(1\pm1)(1+2(\kappa+u)))s\notag\\
    				&\phantom{=\text{}}\text{}+(2k+3-2(\kappa+u))\mu^2\epsilon s^u+\mu^2Cs^u\;.\label{eigenvalues le ..., PP_2}
  			\end{align}
		The first term of the right-hand side of~\eqref{eigenvalues ge ..., PP_2} and~\eqref{eigenvalues le ..., PP_2} for $\PP^+_2$ and $\PP^-_2$ becomes $2(k+1-2(\kappa+u))s$ and $2(k+2)s$, respectively.
    
    		\item The spectrum of $\QQ_1$ consists of eigenvalues $\lambda_1\le\lambda_3\le\cdots$, taking multiplicity into account, such that there are some $D=D(\kappa,u)>0$ and $C=C(\epsilon,\kappa,u)>0$ so that, for all $k\in2\N+1$,
  			\begin{align}
    				\lambda_k&\ge(2k+2-(1\mp1)(1-2\kappa))s+\mu^2Ds^u(k+1)^{-u}\;,
    				\label{eigenvalues ge ..., QQ_1}\\
    				\lambda_k&\le(2k+2-(1\mp1)(1-2\kappa))s+(2k+1+2\kappa)\mu^2\epsilon s^u+\mu^2Cs^u\;.
    				\label{eigenvalues le ..., QQ_1}
  			\end{align}
		The first term of the right-hand side of~\eqref{eigenvalues ge ..., QQ_1} and~\eqref{eigenvalues le ..., QQ_1} for $\QQ_1^+$ and $\QQ_1^-$ is $2(k+1)s$ and $2(k+2\kappa)s$, respectively.
    
    		\item The spectrum of $\QQ_2$ consists of eigenvalues $\lambda_1\le\lambda_3\le\cdots$, taking multiplicity into account, such that there are some $D=D(\kappa,u)>0$ and $C=C(\epsilon,\kappa,u)>0$ so that, for all $k\in2\N+1$,
  			\begin{align}
    				\lambda_k&\ge(2k-2+(1\pm1)(1-2\kappa))s+\mu^2Ds^u(k+1)^{-u}\;,
    				\label{eigenvalues ge ..., QQ_2}\\
    				\lambda_k&\le(2k-2+(1\pm1)(1-2\kappa))s+(2k-1-2\kappa)\mu^2\epsilon s^u+\mu^2Cs^u\;.
    				\label{eigenvalues le ..., QQ_2}
  			\end{align}
		The first term of the right-hand side of~\eqref{eigenvalues ge ..., QQ_2} and~\eqref{eigenvalues le ..., QQ_2} for $\QQ_2^+$ and $\QQ_2^-$ is $2(k-2\kappa)s$ and $2(k-1)s$, respectively.
		
		\item For $\WW_{2,1}$, we can take $\tilde u=\frac{u+1}{2}$ satisfying~\eqref{tilde u}. Moreover the maximum eigenvalue of $\mp sV$ is $s(1\mp(2\kappa+u)-u)$. Thus the spectrum of $\WW_{2,1}$ consists of two groups of eigenvalues, $\lambda_0\le\lambda_2\le\cdots$ and $\lambda_1\le\lambda_3\le\cdots$, repeated according to multiplicity, such that there are some $D=D(\kappa,u)>0$, $C=C(\epsilon,\kappa,u)>0$, $\widetilde C=\widetilde C(\epsilon,\kappa,u)>0$ and $E=E(\epsilon,\kappa)>0$ so that, for all $k\in2\N$,
			\begin{align}
    				\lambda_k&\ge\big(1-2\mu u\epsilon s^{\frac{u-1}{2}}\big)(2k+3-2\kappa)s\notag\\
				&\phantom{\le\text{}}\text{}+\mu^2Ds^u(k+1)^{-u}-2\mu u\widetilde Cs^{\frac{u+1}{2}}\mp(1+2\kappa)s\;,
				\label{eigenvalues ge ..., WW_2,1, even case}\\
				\lambda_k&\le(2k+4-(1\pm1)(2\kappa+u))s\notag\\
    				&\phantom{\le\text{}}\text{}+(2k+3-2\kappa)\epsilon(\mu^2s^u+4\mu us^{\frac{u+1}{2}})
				+\mu^2Cs^u+4\mu uEs^{\frac{u+1}{2}}\;,
				\label{eigenvalues le ..., WW_2,1, even case}
			\end{align}
		and, for all $k\in2\N+1$,
  			\begin{align}
				\lambda_k&\ge\big(1-2\mu u\epsilon s^{\frac{u-1}{2}}\big)(2k+1+2(\kappa+u))s\notag\\
				&\phantom{\le\text{}}\text{}+\mu^2Ds^u(k+1)^{-u}-2\mu u\widetilde Cs^{\frac{u+1}{2}}\pm(1+2(\kappa+u))s\;,
				\label{eigenvalues ge ..., WW_2,1, odd case}\\
    				\lambda_k&\le(2k+2+(1\mp1)(2\kappa+u))s\notag\\
    				&\phantom{\le\text{}}\text{}+(2k+1+2(\kappa+u))\epsilon(\mu^2s^u+4\mu us^{\frac{u+1}{2}})
				+\mu^2Cs^u+4\mu uEs^{\frac{u+1}{2}}\;.
				\label{eigenvalues le ..., WW_2,1, odd case}
  			\end{align}
		
		\item $\WW_{1,1}$ and $\WW_{2,2}$ also have a discrete spectrum, which has the lower bound given by~\eqref{lambda_k ge ..., case of WW} and Proposition~\ref{p: WW}~\eqref{i: WW, xi', xi''}. We omit its explicit expression because it will not be used. The lower estimate of Proposition~\ref{p: WW}~\eqref{i: tilde u} may not be possible for $\WW_{1,1}$ and $\WW_{2,2}$ in general. In fact, according to Remark~\ref{r: P_0, Q_0, P, Q, F, G}~\eqref{i: remark about tilde u}, the existence of $\tilde u$ for $\WW_{1,1}$ (respectively, $\WW_{2,2}$) is characterized by the additional condition $2\kappa>u$ (respectively, $2\kappa<-3u$), which is an additional restriction.

	\end{itemize}
Table~\ref{table: 2, sign of the eigenvalues of PP_i, QQ_j and WW_i,j} contains the information about the sign of the eigenvalues of $\PP_i$, $\QQ_j$ and $\WW_{i,j}$ given by the above spectral estimates.

\begin{table}[h]
\renewcommand{\arraystretch}{1.3}
\begin{tabular}{cl l | l |}
\cline{3-3}
&& \multicolumn{1}{|c|}{Sign of eigenvalues} \\
\hline
\multicolumn{2}{|c|}{$\PP_1$} & $+$\quad $\forall k\in2\N$ \\
\hline
\multicolumn{1}{|c|}{\multirow{3}{*}{$\PP_2^+$}} & \multicolumn{1}{l|}{\multirow{2}{*}{\color{lightgray}$\kappa>\frac{1}{2}-u$ }} & \color{lightgray} ?\quad if $k<2(\kappa+u)-1$ even \\
\multicolumn{1}{|c|}{} & \multicolumn{1}{l|}{} & $+$\quad if $k\ge2(\kappa+u)-1$ even \\
\cline{2-2}\cline{3-3}
\multicolumn{1}{|c|}{} & \multicolumn{1}{l|}{$\kappa\le\frac{1}{2}-u$} & $+$\quad $\forall k\in2\N$ \\
\cline{1-2}\cline{3-3}
\multicolumn{2}{|c|}{$\PP_2^-$} & $+$\quad $\forall k\in2\N$ \\
\hline
\hline
\multicolumn{2}{|c|}{$\QQ_1^+$} & $+$\quad $\forall k\in2\N+1$ \\
\cline{1-2}\cline{3-3}
\multicolumn{1}{|c|}{\multirow{3}{*}{$\QQ_1^-$}} & \multicolumn{1}{l|}{$\kappa\ge-\frac{1}{2}$} & $+$\quad $\forall k\in2\N+1$\\
\cline{2-2}\cline{3-3}
\multicolumn{1}{|c|}{} & \multicolumn{1}{l|}{\multirow{2}{*}{\color{lightgray}$\kappa<-\frac{1}{2}$}} & \color{lightgray} ?\quad if $k<-2\kappa$ odd\\
\multicolumn{1}{|c|}{} & \multicolumn{1}{l|}{} & $+$\quad if $k\ge-2\kappa$ odd \\
\hline
\multicolumn{2}{|c|}{$\QQ_2$} & $+$\quad $\forall k\in2\N+1$ \\
\hline
\hline
\multicolumn{2}{|c|}{$\WW_{i,j}$} & $+$\quad if $k\gg0$ \\
\hline
\end{tabular}
\vspace{1mm}
\caption{Sign of the eigenvalues of $\PP_i$, $\QQ_i$ and $\WW_{i,j}$}
\label{table: 2, sign of the eigenvalues of PP_i, QQ_j and WW_i,j}
\end{table}

\subsubsection{Laplacians of the maximum/minimum i.b.c.}\label{sss: complex 2, max/min i.b.c.}

\begin{prop}\label{p: 2}
  	Tables~\ref{table: Delta_max/min,0},~\ref{table: Delta_max/min,2} and~\ref{table: Delta_max/min,1} describe $\D_{\text{\rm max/min}}$ for the stated values of $\kappa$.
\end{prop}

\begin{table}[h]
\renewcommand{\arraystretch}{1.3}
\begin{tabular}{c|c|cc|c|}
\cline{2-2}\cline{5-5}
& \multicolumn{1}{c|}{$\D_{\text{\rm max},0}$} & \quad && \multicolumn{1}{c|}{$\D_{\text{\rm min},0}$} \\
\cline{1-2}\cline{4-5}
\multicolumn{1}{|c|}{$\kappa>-\frac{1}{2}$} & $\PP_1$ & \quad & \multicolumn{1}{|c|}{$\kappa\ge\frac{1}{2}-u$} & $\PP_1$ \\
\cline{1-2}\cline{4-5}
\multicolumn{1}{|c|}{\color{lightgray}$-\frac{1}{2}-u<\kappa\le-\frac{1}{2}$} & \color{lightgray}? & \quad & \multicolumn{1}{|c|}{$\kappa<\frac{1}{2}-u$} & $\PP_2$ \\
\cline{1-2}\cline{4-5}
\multicolumn{1}{|c|}{$\kappa\le-\frac{1}{2}-u$} & $\PP_2$ & \quad & \multicolumn{1}{c}{} & \multicolumn{1}{c}{} \\
\cline{1-2}
\end{tabular}
\vspace{1mm}
\caption{Description of $\D_{\text{\rm max/min},0}$}
\label{table: Delta_max/min,0}
\end{table}

\begin{table}[h]
\renewcommand{\arraystretch}{1.3}
\begin{tabular}{c|c|cc|c|}
\cline{2-2}\cline{5-5}
& \multicolumn{1}{c|}{$\D_{\text{\rm max},2}$} & \quad && \multicolumn{1}{c|}{$\D_{\text{\rm min},2}$} \\
\cline{1-2}\cline{4-5}
\multicolumn{1}{|c|}{$\kappa>-\frac{1}{2}$}  & $\QQ_1$ & \quad & \multicolumn{1}{|c|}{$\kappa\ge\frac{1}{2}$} & $\QQ_1$ \\
\cline{1-2}\cline{4-5}
\multicolumn{1}{|c|}{$\kappa\le-\frac{1}{2}$} & $\QQ_2$ & \quad & \multicolumn{1}{|c|}{\color{lightgray}$\frac{1}{2}-u\le\kappa<\frac{1}{2}$} & \color{lightgray}? \\
\cline{1-2}\cline{4-5}
\multicolumn{1}{c}{} & \multicolumn{1}{c}{} & \quad & \multicolumn{1}{|c|}{$\kappa<\frac{1}{2}-u$} & $\QQ_2$ \\
\cline{4-5}
\end{tabular}
\vspace{1mm}
\caption{Description of $\D_{\text{\rm max/min},2}$}
\label{table: Delta_max/min,2}
\end{table}

\begin{table}[h]
\renewcommand{\arraystretch}{1.3}
\begin{tabular}{c|c|cc|c|}
\cline{2-2}\cline{5-5}
& \multicolumn{1}{c|}{$\D_{\text{\rm max},1}$} & \quad & & \multicolumn{1}{c|}{$\D_{\text{\rm min},1}$} \\
\cline{1-2}\cline{4-5}
\multicolumn{1}{|c|}{$\kappa>u-\frac{1}{2}$} & $\WW_{1,1}$ & \quad & \multicolumn{1}{|c|}{$\kappa\ge\frac{1}{2}$} & $\WW_{1,1}$ \\
\cline{1-2}\cline{4-5}
\multicolumn{1}{|c|}{\color{lightgray}$-\frac{1}{2}<\kappa\le u-\frac{1}{2}$} & \color{lightgray}? & \quad & \multicolumn{1}{|c|}{$\frac{1}{2}-u\le\kappa<\frac{1}{2}$} & $\WW_{2,1}$ \\
\cline{1-2}\cline{4-5}
\multicolumn{1}{|c|}{$-\frac{1}{2}-u<\kappa\le-\frac{1}{2}$} & $\WW_{2,1}$ & \quad & \multicolumn{1}{|c|}{\color{lightgray}$\frac{1}{2}-2u\le\kappa<\frac{1}{2}-u$} & \color{lightgray}? \\
\cline{1-2}\cline{4-5}
\multicolumn{1}{|c|}{$\kappa\le-\frac{1}{2}-u$} & $\WW_{2,2}$ &  \quad & \multicolumn{1}{|c|}{$\kappa<\frac{1}{2}-2u$} & $\WW_{2,2}$ \\
\cline{1-2}\cline{4-5}
\end{tabular}
\vspace{1mm}
\caption{Description of $\D_{\text{\rm max/min},1}$}
\label{table: Delta_max/min,1}
\end{table}

\begin{proof}
	The operators $d_{0,2}$, $\delta_{0,2}$, $d_{1,1}$ and $\delta_{1,1}$ are like $d$ and $\delta$ in Section~\ref{ss: complex 1}. So Proposition~\ref{p: 1} and Remark~\ref{r: 1}~\eqref{i: EE_i} give the following:
	\begin{align}
		\sD(d_{0,2,\text{\rm max}})&\supset
			\begin{cases}
				\FF_1^0 & \text{if $\kappa>-\frac{1}{2}-u$} \\
				\FF_2^0 & \text{if $\kappa\le-\frac{1}{2}-u$}\;,
			\end{cases}\label{sD(d_0,2,max) supset ...}\\
		\sD(d_{0,2,\text{\rm min}})&\supset
			\begin{cases}
				\FF_1^0 & \text{if $\kappa\ge\frac{1}{2}-u$} \\
				\FF_2^0 & \text{if $\kappa<\frac{1}{2}-u$}\;,
			\end{cases}\label{sD(d_0,2,min) supset ...}\\
		\sD(\delta_{0,2,\text{\rm max}})&\supset
			\begin{cases}
				\FF_1^{1,2} & \text{if $\kappa\ge\frac{1}{2}-u$} \\
				\FF_2^{1,2} & \text{if $\kappa<\frac{1}{2}-u$}\;,
			\end{cases}\label{sD(delta_0,2,max) supset ...}\\
		\sD(\delta_{0,2,\text{\rm min}})&\supset
			\begin{cases}
				\FF_1^{1,2} & \text{if $\kappa>-\frac{1}{2}-u$} \\
				\FF_2^{1,2} & \text{if $\kappa\le-\frac{1}{2}-u$}\;,
			\end{cases}\label{sD(delta_0,2,min) supset ...}\\
		\sD(d_{1,1,\text{\rm max}})&\supset
			\begin{cases}
				\FF_1^{1,1} & \text{if $\kappa>-\frac{1}{2}$} \\
				\FF_2^{1,1} & \text{if $\kappa\le-\frac{1}{2}$}\;,
			\end{cases}\label{sD(d_1,1,max) supset ...}\\
		\sD(d_{1,1,\text{\rm min}})&\supset
			\begin{cases}
				\FF_1^{1,1} & \text{if $\kappa\ge\frac{1}{2}$} \\
				\FF_2^{1,1} & \text{if $\kappa<\frac{1}{2}$}\;,\\
			\end{cases}\label{sD(d_1,1,min) supset ...}\\
		\sD(\delta_{1,1,\text{\rm max}})&\supset
			\begin{cases}
				\FF_1^2 & \text{if $\kappa\ge\frac{1}{2}$} \\
				\FF_2^2 & \text{if $\kappa<\frac{1}{2}$}\;,
			\end{cases}\label{sD(delta_1,1,max) supset ...}\\
		\sD(\delta_{1,1,\text{\rm min}})&\supset
			\begin{cases}
				\FF_1^2 & \text{if $\kappa>-\frac{1}{2}$} \\
				\FF_2^2 & \text{if $\kappa\le-\frac{1}{2}$}\;,\\
			\end{cases}\label{sD(delta_1,1,min) supset ...}
	\end{align}
	\begin{alignat*}{4}
		d_{0,2,\text{\rm max}}&=d_{0,2,\text{\rm min}}\;,&\quad
		\delta_{0,2,\text{\rm max}}&=\delta_{0,2,\text{\rm min}}&\quad
		&\text{if}&\quad|\kappa+u|&\ge\textstyle{\frac{1}{2}}\;,\\
		d_{1,1,\text{\rm max}}&=d_{1,1,\text{\rm min}}\;,&\quad
		\delta_{1,1,\text{\rm max}}&=\delta_{1,1,\text{\rm min}}&\quad
		&\text{if}&\quad|\kappa|&\ge\textstyle{\frac{1}{2}}\;.
	\end{alignat*}
On the other hand, since $d_{0,1}$, $\delta_{0,1}$, $d_{1,2}$ and $\delta_{1,2}$  are multiplication operators, we have
	\begin{alignat*}{2}
		d_{0,1,\text{\rm max}}&=d_{0,1,\text{\rm min}}\;,&\quad
		\delta_{0,1,\text{\rm max}}&=\delta_{0,1,\text{\rm min}}\;,\\
		d_{1,2,\text{\rm max}}&=d_{1,2,\text{\rm min}}\;,&\quad
		\delta_{1,2,\text{\rm max}}&=\delta_{1,2,\text{\rm min}}\;.
	\end{alignat*}
These are maximal multiplication operators \cite[Examples~III-2.2 and~V-3.22]{Kato1995}. They satisfy the following:
	\begin{align}
		\sD(d_{0,1,\text{\rm max/min}})&\supset
			\begin{cases}
				\FF_1^0 & \text{if $\kappa>-\frac{1}{2}$}\\
				\FF_2^0 & \text{if $\kappa<\frac{3}{2}-2u$}\;,
			\end{cases}
		\label{sD(d_0,1,max/min supset ...}\\
		\sD(\delta_{0,1,\text{\rm max/min}})&\supset
			\begin{cases}
				\FF_1^{1,1} & \text{if $\kappa>u-\frac{1}{2}$}\\
				\FF_2^{1,1} & \text{if $\kappa<\frac{3}{2}-u$}\;,
			\end{cases}
		\label{sD(delta_0,1,max/min supset ...}\\
		\sD(d_{1,2,\text{\rm max/min}})&\supset
			\begin{cases}
				\FF_1^{1,2} & \text{if $\kappa>-\frac{3}{2}$}\\
				\FF_2^{1,2} & \text{if $\kappa<\frac{1}{2}-2u$}\;,
			\end{cases}
		\label{sD(d_1,2,max/min supset ...}\\
		\sD(\delta_{1,2,\text{\rm max/min}})&\supset
			\begin{cases}
				\FF_1^2 & \text{if $\kappa>u-\frac{3}{2}$}\\
				\FF_2^2 & \text{if $\kappa<\frac{1}{2}-u$}\;.
			\end{cases}
		\label{sD(delta_1,2,max/min supset ...}
	\end{align}  
By Remark~\ref{r: 1}~\eqref{i: xi_n}, we also get
	\begin{alignat}{2}
      		\sD(d_{\text{\rm min},0})&=\sD(d_{0,1,\text{\rm min}})\cap\sD(d_{0,2,\text{\rm min}})\;,&\quad
      		d_{\text{\rm min},0}&=
        			\begin{pmatrix}
          			d_{0,1,\text{\rm min}}|_{\sD(d_{\text{\rm min},0})}\\
          			d_{0,2,\text{\rm min}}|_{\sD(d_{\text{\rm min},0})}
        			\end{pmatrix}\;,\label{d_min,0}\\
		\sD(\delta_{\text{\rm min},1})&=\sD(\delta_{1,1,\text{\rm min}})\cap\sD(\delta_{1,2,\text{\rm min}})\;,
		&\quad
      		\delta_{\text{\rm min},1}&=
        			\begin{pmatrix}
          			\delta_{1,1,\text{\rm min}}|_{\sD(\delta_{\text{\rm min},1})}\\
          			\delta_{1,2,\text{\rm min}}|_{\sD(\delta_{\text{\rm min},1})}
        			\end{pmatrix}\;,\label{delta_min,1}
    	\end{alignat}
complementing Lemma~\ref{l: oplus 1} in this case.

	From~\eqref{sD(d_0,2,max) supset ...}--\eqref{delta_min,1}, Lemmas~\ref{l: oplus 1} and~\ref{l: oplus 2}, and \cite[Chapter~XI-12, p.~338, Eq.~(1)]{Yosida1980}, it follows that
		\begin{align*}
			\sD(\D_{\text{\rm max},0}^{1/2})=\sD(d_{\text{\rm max},0})
			=\sD(d_{0,1,\text{\rm max}})\cap\sD(d_{0,2,\text{\rm max}})&\supset
				\begin{cases}
					\FF_1^0 & \text{if $\kappa>-\frac{1}{2}$}\\
					\FF_2^0 & \text{if $\kappa\le-\frac{1}{2}-u$}\;,
				\end{cases}\\
			\sD(\D_{\text{\rm min},0}^{1/2})=\sD(d_{\text{\rm min},0})
			=\sD(d_{0,1,\text{\rm min}})\cap\sD(d_{0,2,\text{\rm min}})&\supset
				\begin{cases}
					\FF_1^0 & \text{if $\kappa\ge\frac{1}{2}-u$}\\
					\FF_2^0 & \text{if $\kappa<\frac{1}{2}-u$}\;,
				\end{cases}\\
			\sD(\D_{\text{\rm max},2}^{1/2})=\sD(\delta_{\text{\rm min},1})
			=\sD(\delta_{1,1,\text{\rm min}})\cap\sD(\delta_{1,2,\text{\rm min}})&\supset
				\begin{cases}
					\FF_1^2 & \text{if $\kappa>-\frac{1}{2}$}\\
					\FF_2^2 & \text{if $\kappa\le-\frac{1}{2}$}\;,
				\end{cases}\\
			\sD(\D_{\text{\rm min},2}^{1/2})=\sD(\delta_{\text{\rm max},1})
			=\sD(\delta_{1,1,\text{\rm max}})\cap\sD(\delta_{1,2,\text{\rm max}})&\supset
				\begin{cases}
					\FF_1^2 & \text{if $\kappa\ge\frac{1}{2}$}\\
					\FF_2^2 & \text{if $\kappa<\frac{1}{2}-u$}\;,
				\end{cases}
		\end{align*}
		\begin{align*}
			\sD(\D_{\text{\rm max},1}^{1/2})&=\sD(\delta_{\text{\rm min},0}+d_{\text{\rm max},1})
			=\sD(\delta_{\text{\rm min},0})\cap\sD(d_{\text{\rm max},1})\\
			&\supset(\sD(\delta_{0,1,\text{\rm min}})\oplus\sD(\delta_{0,2,\text{\rm min}}))
			\cap(\sD(d_{1,1,\text{\rm max}}\oplus\sD(d_{1,2,\text{\rm max}}))\\
			&\supset
				\begin{cases}
					\FF_{1,1}^1 & \text{if $\kappa>u-\frac{1}{2}$}\\
					\FF_{2,1}^1 & \text{if $-\frac{1}{2}-u<\kappa\le-\frac{1}{2}$}\\
					\FF_{2,2}^1 & \text{if $\kappa\le-\frac{1}{2}-u$}\;,
				\end{cases}\\
			\sD(\D_{\text{\rm min},1}^{1/2})&=\sD(\delta_{\text{\rm max},0}+d_{\text{\rm min},1})
			=\sD(\delta_{\text{\rm max},0})\cap\sD(d_{\text{\rm min},1})\\
			&\supset(\sD(\delta_{0,1,\text{\rm max}})\oplus\sD(\delta_{0,2,\text{\rm max}}))
			\cap(\sD(d_{1,1,\text{\rm min}}\oplus\sD(d_{1,2,\text{\rm min}}))\\
			&\supset
				\begin{cases}
					\FF_{1,1}^1 & \text{if $\kappa\ge\frac{1}{2}$}\\
					\FF_{2,1}^1 & \text{if $\frac{1}{2}-u\le\kappa<\frac{1}{2}$}\\
					\FF_{2,2}^1 & \text{if $\kappa<\frac{1}{2}-2u$}\;.
				\end{cases}
		\end{align*}
	Since $\FF_i^0$, $\FF_j^2$ and $\FF_{i,j}^1$ are cores of $\PP_i^{1/2}$, $\QQ_j^{1/2}$ and $\WW_{i,j}^{1/2}$, respectively, and taking into account Table~\ref{table: 2, PP_i, QQ_j, WW_i,j}, it follows that
		\begin{alignat*}{2}
				\D_{\text{\rm max},0}^{1/2}&\supset
					\begin{cases}
						\PP_1^{1/2} & \text{if $\kappa>-\frac{1}{2}$}\\
						\PP_2^{1/2} & \text{if $\kappa\le-\frac{1}{2}-u$}\;,
					\end{cases}&\quad
				\D_{\text{\rm min},0}^{1/2}&\supset
					\begin{cases}
						\PP_1^{1/2} & \text{if $\kappa\ge\frac{1}{2}-u$}\\
						\PP_2^{1/2} & \text{if $\kappa<\frac{1}{2}-u$}\;,
					\end{cases}\\
				\D_{\text{\rm max},2}^{1/2}&\supset
					\begin{cases}
						\QQ_1^{1/2} & \text{if $\kappa>-\frac{1}{2}$}\\
						\QQ_2^{1/2} & \text{if $\kappa\le-\frac{1}{2}$}\;,
					\end{cases}&\quad
				\D_{\text{\rm min},2}^{1/2}&\supset
					\begin{cases}
						\QQ_1^{1/2} & \text{if $\kappa\ge\frac{1}{2}$}\\
						\QQ_2^{1/2} & \text{if $\kappa<\frac{1}{2}-u$}\;,
					\end{cases}
		\end{alignat*}
		\begin{align*}
				\D_{\text{\rm max},1}^{1/2}&\supset
					\begin{cases}
						\WW_{1,1}^{1/2} & \text{if $\kappa>u-\frac{1}{2}$}\\
						\WW_{2,1}^{1/2} & \text{if $-\frac{1}{2}-u<\kappa\le-\frac{1}{2}$}\\
						\WW_{2,2}^{1/2} & \text{if $\kappa\le-\frac{1}{2}-u$}\;,
					\end{cases}\\
				\D_{\text{\rm min},1}^{1/2}&\supset
					\begin{cases}
						\WW_{1,1}^{1/2} & \text{if $\kappa\ge\frac{1}{2}$}\\
						\WW_{2,1}^{1/2} & \text{if $\frac{1}{2}-u\le\kappa<\frac{1}{2}$}\\
						\WW_{2,2}^{1/2} & \text{if $\kappa<\frac{1}{2}-2u$}\;.
					\end{cases}
		\end{align*}
	But these inclusions are equalities because they involve self-adjoint operators.
\end{proof}

\begin{prop}\label{p: 2, ker Delta_max/min,r = 0}
		We have $\ker\D_{\text{\rm max/min}}=0$.
\end{prop}

\begin{proof}
	We have $\ker\D_{\text{\rm max/min,ev}}=0$ because $\ker d_{\text{\rm max/min},0}=0$ and $\ker\delta_{\text{\rm max/min},1}=0$ by Lemma~\ref{l: oplus 1},~\eqref{d_min,0} and~\eqref{delta_min,1}, since $d_{0,1,\text{\rm max/min}}$ and $\delta_{1,2,\text{\rm max/min}}$ are maximal multiplication operators in $L_+^2$ by continuous non-vanishing functions.\footnote{We may also use Table~\ref{table: 2, sign of the eigenvalues of PP_i, QQ_j and WW_i,j} and Proposition~\ref{p: 2} for some values of $\kappa$ (Tables~\ref{table: Delta_max/min,0} and~\ref{table: Delta_max/min,2}).}
	
	Since $\sigma(\D_{\text{\rm max/min,ev}})$ is bounded away from $0$, we get $\sR(\D_{\text{\rm max/min},0})=L_+^2=\sR(\D_{\text{\rm max/min},2})$ by the spectral theorem. The maximal multiplication operator by $\rho^{\pm u}$ in $L_+^2$ will be also denoted by $\rho^{\pm u}$. Let $\phi\in\sD(\D_{\text{\rm max/min},0})$ such that $\D_{\text{\rm max/min},0}\phi\in\sD(\rho^u)$. By~\eqref{d_1,1 = rho^-u d_0,2 rho^u},
		\begin{align*}
			\psi:=\frac{1}{\mu}\rho^ud_{0,2,\text{\rm max/min}}\phi
			&\in\sD(\delta_{0,2,\text{\rm max/min}}\,\rho^{-u})
			\cap\sD(\rho^u\,\delta_{0,2,\text{\rm max/min}}\,\rho^{-u})\\
			&=\sD(\rho^{-u}\,\delta_{1,1,\text{\rm max/min}})
			\cap\sD(\delta_{1,1,\text{\rm max/min}})\;.
		\end{align*}
	Then $\psi\in\sD(\delta_{\text{\rm max/min},1})$ by~\eqref{delta_min,1} since $\rho^{-u}\psi\in L_+^2$ and $\delta_{1,2,\text{\rm max/min}}$ is the maximal multiplication operator by $-\mu\rho^{-u}$. In the following, for the sake of simplicity, the notation $d_{0,2}$, $\delta_{1,1}$, $\delta_{0,2}$ and $\D_0$ is used for $d_{0,2,\text{\rm max/min}}$, $\delta_{1,1,\text{\rm max/min}}$, $\delta_{0,2,\text{\rm max/min}}$ and $\D_{\text{\rm max/min},0}$, respectively. It also follows from~\eqref{d_1,1 = rho^-u d_0,2 rho^u} that
		\begin{multline*}
			d_{\text{\rm max/min},0}(\phi)+\delta_{\text{\rm max/min},1}(\psi)=
				\begin{pmatrix}
					\mu\rho^{-u}\phi+\delta_{1,1}\psi \\
					d_{0,2}\phi-\mu\rho^{-u}\psi
				\end{pmatrix} \\
			=
				\begin{pmatrix}
					\mu\rho^{-u}\phi+\frac{1}{\mu}\delta_{1,1}\rho^ud_{0,2}\phi \\
					0
				\end{pmatrix}
			=
				\begin{pmatrix}
					\mu\rho^{-u}\phi+\frac{1}{\mu}\rho^u\delta_{0,2}d_{0,2}\phi \\
					0
				\end{pmatrix}
			=
				\begin{pmatrix}
					\frac{1}{\mu}\rho^u\D_0\phi \\
					0
				\end{pmatrix}\;.
		\end{multline*}
	Since $\sR(\D_{\text{\rm max/min},0})=L_+^2$, we get
		\[
			\sR(\rho^u)\oplus0\subset\sR(d_{\text{\rm max/min},0})+\sR(\delta_{\text{\rm max/min},1})\;.
		\]
	
	With an analogous argument, using Lemma~\ref{l: oplus 1} instead of~\eqref{delta_min,1}, we get
		\[
			0\oplus\sR(\rho^u)\subset\sR(d_{\text{\rm max/min},0})+\sR(\delta_{\text{\rm max/min},1})\;.
		\]
	Therefore
		\[
			\sR(\rho^u)\oplus\sR(\rho^u)
			\subset\sR(d_{\text{\rm max/min},0})+\sR(\delta_{\text{\rm max/min},1})\;,
		\]
	obtaining that $\sR(d_{\text{\rm max/min},0})+\sR(\delta_{\text{\rm max/min},1})$ is dense in $L^2_+\oplus L^2_+$ because $\sR(\rho^u)$ is dense in $L_+^2$. Thus $\ker\D_{\text{\rm max/min},1}=0$ \cite[Lemma~2.1]{BruningLesch1992}.
\end{proof}

\begin{cor}\label{c: 2, the same eigenvalues}
	$\D_{\text{\rm max/min,ev}}$ and $\D_{\text{\rm max/min},1}$ have the same eigenvalues, with the same multiplicity.
\end{cor}

\begin{proof}
	This is a direct consequence of Proposition~\ref{p: 2, ker Delta_max/min,r = 0} and Lemma~\ref{l: discrete spectrum}. 
\end{proof}

\begin{rem}\label{r: u > 0}
Some generalities about this complex of length two hold for all $u>0$, like~\eqref{sD(d_0,2,max) supset ...}--\eqref{delta_min,1}, Proposition~\ref{p: 2, ker Delta_max/min,r = 0} and Corollary~\ref{c: 2, the same eigenvalues}. But the main results require $0<u<1$.
\end{rem}

Concerning the spectrum, the following corollary fills the gaps in Tables~\ref{table: Delta_max/min,0}--\ref{table: Delta_max/min,1}.

\begin{cor}\label{c: 2, spectrum}
	Tables~\ref{table: sigma(Delta_max/min,ev)} and~\ref{table: sigma(Delta_max/min,1)} describe the spectra of $\D_{\text{\rm max/min,ev}}$ and $\D_{\text{\rm max/min},1}$ in terms of the spectra of $\PP_i$, $\QQ_j$ and $\WW_{i,j}$ for the stated values of $\kappa$.
\end{cor}

\begin{table}[h]
\renewcommand{\arraystretch}{1.3}
\begin{tabular}{c|c|cc|c|}
\cline{2-2}\cline{5-5}
& \multicolumn{1}{c|}{$\sigma(\D_{\text{\rm max,ev}})$} & \quad && \multicolumn{1}{c|}{$\sigma(\D_{\text{\rm min,ev}})$} \\
\cline{1-2}\cline{4-5}
\multicolumn{1}{|c|}{$\kappa>-\frac{1}{2}$} & $\sigma(\PP_1\oplus\QQ_1)$ & \quad & \multicolumn{1}{|c|}{$\kappa\ge\frac{1}{2}$} & $\sigma(\PP_1\oplus\QQ_1)$ \\
\cline{1-2}\cline{4-5}
\multicolumn{1}{|c|}{$-\frac{1}{2}-u<\kappa\le-\frac{1}{2}$} & $\sigma(\WW_{2,1})$ & \quad & \multicolumn{1}{|c|}{$\frac{1}{2}-u\le\kappa<\frac{1}{2}$} & $\sigma(\WW_{2,1})$ \\
\cline{1-2}\cline{4-5}
\multicolumn{1}{|c|}{$\kappa\le-\frac{1}{2}-u$} & $\sigma(\PP_2\oplus\QQ_2)$ & \quad & \multicolumn{1}{|c|}{$\kappa<\frac{1}{2}-u$} & $\sigma(\PP_2\oplus\QQ_2)$ \\
\cline{1-2}\cline{4-5}
\end{tabular}
\vspace{1mm}
\caption{Spectrum of $\D_{\text{\rm max/min,ev}}$}
\label{table: sigma(Delta_max/min,ev)}
\end{table}

\begin{table}[h]
\renewcommand{\arraystretch}{1.3}
\begin{tabular}{c|c|cc|c|}
\cline{2-2}\cline{5-5}
& \multicolumn{1}{c|}{$\sigma(\D_{\text{\rm max},1})$} & \quad && \multicolumn{1}{c|}{$\sigma(\D_{\text{\rm min},1})$} \\
\cline{1-2}\cline{4-5}
\multicolumn{1}{|c|}{$\kappa>u-\frac{1}{2}$} & $\sigma(\WW_{1,1})$ & \quad & \multicolumn{1}{|c|}{$\kappa\ge\frac{1}{2}$} & $\sigma(\WW_{1,1})$ \\
\cline{1-2}\cline{4-5}
\multicolumn{1}{|c|}{$-\frac{1}{2}<\kappa\le u-\frac{1}{2}$} & $\sigma(\PP_1\oplus\QQ_1)$ & \quad 
& \multicolumn{1}{|c|}{$\frac{1}{2}-u\le\kappa<\frac{1}{2}$} & $\sigma(\WW_{2,1})$ \\
\cline{1-2}\cline{4-5}
\multicolumn{1}{|c|}{$-\frac{1}{2}-u<\kappa\le-\frac{1}{2}$} & $\sigma(\WW_{2,1})$ & \quad 
& \multicolumn{1}{|c|}{$\frac{1}{2}-2u\le\kappa<\frac{1}{2}-u$} & $\sigma(\PP_2\oplus\QQ_2)$ \\
\cline{1-2}\cline{4-5}
\multicolumn{1}{|c|}{$\kappa\le-\frac{1}{2}-u$} & $\sigma(\WW_{2,2})$ & \quad & \multicolumn{1}{|c|}{$\kappa<\frac{1}{2}-2u$} & $\sigma(\WW_{2,2})$ \\
\cline{1-2}\cline{4-5}
\end{tabular}
\vspace{1mm}
\caption{Spectrum of $\D_{\text{\rm max/min},1}$}
\label{table: sigma(Delta_max/min,1)}
\end{table}

\begin{proof}
	This is a direct consequence of Proposition~\ref{p: 2} and Corollary~\ref{c: 2, the same eigenvalues}. 
\end{proof}

\subsection{The wave operator}
\label{s: wave, simple}

For the Hermitian bundle versions of $E$ and $F$, consider the wave operator $\exp(itD_{\text{\rm max/min}})$ ($i=\sqrt{-1}$) on $L^2(E)$ or $L^2(F)$, which is bounded.

\begin{prop}\label{p: wave, simple}
  	For $\phi$ in $L^2(E)$ or $L^2(F)$, let $\phi_t=\exp(itD_{\text{\rm max/min}})\phi$. If $\supp \phi\subset(0,a]$ for some $a>0$, then $\supp\phi_t\subset(0,a+|t|]$ for all $t\in\R$.
\end{prop}

\begin{proof}
  	The case of $E$ is given by \cite[Proposition~8.7~(ii)]{AlvCalaza2017}. Then consider the case of $F$, where the proof needs a slight change because the needed description of $\sD^\infty(\D_{\text{\rm max/min}})$ is not available. Since $\exp(itD_{\text{\rm max/min}})$ is bounded, we can assume that $\phi\in\sD^\infty(\D_{\text{\rm max/min}})$. Write $\phi_t=\phi_{t,0}+\phi_{t,1}+\phi_{t,2}$ with $\phi_{t,r}\in C^\infty(F_r)\equiv C^\infty_+$ ($r=0,2$), and $\phi_{t,1}\equiv\left(\begin{smallmatrix}\phi_{t,1,1}\\\phi_{t,1,2}\end{smallmatrix}\right)\in C^\infty(F_1)\equiv C^\infty_+\oplus C^\infty_+$. Suppose that $t\ge0$, the other case being analogous. For any $c>a$,
    		\begin{align*}
      			\frac{d}{dt}\int_{a+t}^c|\phi_t(\rho)|^2\,d\rho
			&=\int_{a+t}^c((iD\phi_t,\phi_t)+(\phi_t,iD\phi_t))(\rho)\,d\rho-|\phi_t(a+t)|^2\\
      			&=i\int_{a+t}^c((D\phi_t,\phi_t)-(\phi_t,D\phi_t))(\rho)\,d\rho-|\phi_t(a+t)|^2\;.
    		\end{align*}
  	Now, $d_{0,1}\equiv\delta_{0,1}$ and $d_{1,2}\equiv\delta_{1,2}$ are multiplication operators by real valued functions. Moreover $d_{0,2}$ and $\delta_{0,2}$ are equal to $\frac{d}{d\rho}$ and $-\frac{d}{d\rho}$, respectively, up to the sum of multiplication operators by the same real valued functions, and the same is true for $d_{1,1}$ and $\delta_{1,1}$. Thus
		\begin{multline*}
      			(D\phi_t,\phi_t)-(\phi_t,D\phi_t)\\
				\begin{aligned}
					&=\left(\delta_{0,1}\phi_{t,1,1}+\delta_{0,2}\phi_{t,1,2},\phi_{t,0}\right)
      					+\left(d_{1,1}\phi_{t,1,1}+d_{1,2}\phi_{t,1,2},\phi_{t,2}\right)\\
					&\phantom{=\text{}}\text{}+\left(d_{0,1}\phi_{t,0}+\delta_{1,1}\phi_{t,2},\phi_{t,1,1}\right)
      					+\left(d_{0,2}\phi_{t,0}+\delta_{1,2}\phi_{t,2},\phi_{t,1,2}\right)\\
					&\phantom{=\text{}}\text{}-\left(\phi_{t,0},\delta_{0,1}\phi_{t,1,1}+\delta_{0,2}\phi_{t,1,2}\right)
      					-\left(\phi_{t,2},d_{1,1}\phi_{t,1,1}+d_{1,2}\phi_{t,1,2}\right)\\
					&\phantom{=\text{}}\text{}-\left(\phi_{t,1,1},d_{0,1}\phi_{t,0}+\delta_{1,1}\phi_{t,2}\right)
      					-\left(\phi_{t,1,2},d_{0,2}\phi_{t,0}+\delta_{1,2}\phi_{t,2}\right)\\
					&=-\phi_{t,1,2}'\ol{\phi_{t,0}}+\phi_{t,1,1}'\ol{\phi_{t,2}}-\phi_{t,2}'\ol{\phi_{t,1,1}}+\phi_{t,0}'\ol{\phi_{t,1,2}}\\
					&\phantom{=\text{}}\text{}+\phi_{t,0}\ol{\phi_{t,1,2}'}-\phi_{t,2}\ol{\phi_{t,1,1}'}
					+\phi_{t,1,1}\ol{\phi_{t,2}'}-\phi_{t,1,2}\ol{\phi_{t,0}'}\\
					&=2i\,\Im(\phi_{t,0}\ol{\phi_{t,1,2}'}+\phi_{t,1,1}'\ol{\phi_{t,2}}
					+\phi_{t,1,1}\ol{\phi_{t,2}'}+\phi_{t,0}'\ol{\phi_{t,1,2}})\\
					&=2i\,\Im(\phi_{t,1,1}\ol{\phi_{t,2}}+\phi_{t,0}\ol{\phi_{t,1,2}})'\;.
			\end{aligned}
		\end{multline*}
	Therefore 
	\[
		i\int_{a+t}^c((D\phi_t,\phi_t)-(\phi_t,D\phi_t))(\rho)\,d\rho\in\R\;,
	\]
	and
    		\begin{multline*}
      			\left|\int_{a+t}^c((D\phi_t,\phi_t)-(\phi_t,D\phi_t))(\rho)\,d\rho\right|\\
				\begin{aligned}
					&\le2|(\phi_{t,1,1}\ol{\phi_{t,2}}+\phi_{t,0}\ol{\phi_{t,1,2}})(c)
					-(\phi_{t,1,1}\ol{\phi_{t,2}}+\phi_{t,0}\ol{\phi_{t,1,2}})(a+t)|\\
      					&\le|\phi_{t,1,1}(c)|^2+|\phi_{t,2}(c)|^2+|\phi_{t,1,2}(c)|^2+|\phi_{t,0}(c)|^2\\
					&\phantom{\le\text{}}\text{}+|\phi_{t,1,1}(a+t)|^2+|\phi_{t,2}(a+t)|^2
					+|\phi_{t,1,2}(a+t)|^2+|\phi_{t,0}(a+t)|^2\\
					&=|\phi_t(c)|^2+|\phi_t(a+t)|^2\;.
				\end{aligned}
    		\end{multline*}
  	Since $t\mapsto \phi_t$ defines a differentiable map with values in $L^2(F)$, it follows that there is a sequence $a<c_i\uparrow\infty$ such that $\phi_t(c_i)\to0$, and
		\[
      			\frac{d}{dt}\int_{a+t}^\infty|\phi_t(\rho)|^2\,d\rho
			=\lim_i\frac{d}{dt}\int_{a+t}^{c_i}|\phi_t(\rho)|^2\,d\rho
			\le\lim_i|\phi_t(c_i)|^2=0\;.
    		\]
	So
    		\[
      			\int_{a+t}^{\infty}|\phi_t(\rho)|^2\,d\rho\le\int_a^{\infty}|\phi_0(\rho)|^2\,d\rho=\int_a^\infty|\phi(\rho)|^2\,d\rho=0\;.\qed
    		\]
\renewcommand{\qed}{}
\end{proof}

\section{Witten's perturbation on a cone}\label{s: Witten cone}

For rel-Morse functions, the rel-local analysis of the Witten's perturbed Laplacian will be reduced to the case of the functions $\pm\frac{1}{2}\rho^2$ on a stratum of a cone with a model adapted metric, where $\rho$ denotes the radial function. This kind of rel-local analysis begins in this section.

\subsection{Witten's perturbation}\label{ss: Witten}

To begin with, recall the following generalities about the Witten's perturbation. Let $M\equiv(M,g)$ be a Riemannian $n$-manifold. For all $x\in M$ and $\alpha\in T_xM^*$, let
  \[
    \alpha\lrcorner=(-1)^{nr+n+1}\,\star\,\alpha\!\wedge\,\star=-\iota_{\alpha^\sharp}\quad\text{on}\quad\bigwedge^rT_xM^*\;,
  \]
involving the Hodge star operator $\star$ on $\bigwedge T_xM^*$ defined by any choice of orientation of $T_xM$. For any $f\in\Cinf(M)$, E.~Witten \cite{Witten1982} has introduced the following perturbations of $d$, $\delta$, $D$ and $\D$, depending on $s\ge0$:
  \begin{gather}
    d_s=e^{-sf}\,d\,e^{sf}=d+s\,df\!\wedge\;,\label{d_s}\\
    \delta_s=e^{sf}\,\delta\,e^{-sf}=\delta-s\,df\lrcorner\;,\label{delta_s}\\
    D_s=d_s+\delta_s=D+sR\;,\notag\\
    \D_s=D_s^2=d_s\delta_s+\delta_sd_s=\D+s(RD+DR)+s^2R^2\;,\label{Delta_s}
  \end{gather}
where $R={df\!\wedge}-{df\lrcorner}$. Notice that $\delta_s=d_s^\dag$; thus $D_s$ and $\D_s$ are formally self-adjoint. By analyzing the terms $RD+DR$ and $R^2$, the expression~\eqref{Delta_s} becomes
  \begin{equation}\label{Delta_s with Hess f and df}
    \D_s=\D+s\,\boldsymbol{\Hess}f+s^2\,|df|^2\;,
  \end{equation}
where $\boldsymbol{\Hess}f$ is an endomorphism defined by $\Hess f$ \cite[Lemma~9.17]{Roe1998}, satisfying $|\boldsymbol{\Hess}f|=|\Hess f|$ \cite[Section~9]{AlvCalaza2017}.

\subsection{De~Rham operators on a cone}\label{ss: de Rham operators cone}

Let $L$ be a non-empty compact stratification. Consider a stratum $N$ of $L$, and the corresponding stratum $M=N\times\R_+$ of $c(L)$. We use the notation $\tilde n=\dim N$ and $n=\dim M=\tilde n+1$. Let $\pi:M\to N$ be the first factor projection, and $\rho$ the radial function on $c(L)$. From $\bigwedge TM^*=\bigwedge TN^*\boxtimes\bigwedge T\R_+^*$, we get a canonical identity
  \begin{equation}\label{bigwedge^rTM^*}
    \bigwedge^rTM^*\equiv\pi^*\bigwedge^rTN^*\oplus d\rho\wedge\pi^*\bigwedge^{r-1}TN^*
    \equiv\pi^*\bigwedge^rTN^*\oplus\pi^*\bigwedge^{r-1}TN^*
  \end{equation}
for every degree $r$. So
  \begin{align}
    \Omega^r(M)&\equiv\Cinf(\R_+,\Omega^r(N))\oplus d\rho\wedge\Cinf(\R_+,\Omega^{r-1}(N))
    \label{Omega^r(M) with d rho}\\
    &\equiv\Cinf(\R_+,\Omega^r(N))\oplus\Cinf(\R_+,\Omega^{r-1}(N))\;.
    \label{Omega^r(M)}
  \end{align}
Here, smooth functions $\R_+\to\Omega(N)$ are defined by considering $\Omega(N)$ as Fr\'echet space with the weak $C^\infty$ topology. In this section, all matrix expressions of vector bundle homomorphisms on $\bigwedge^r TM^*$ or differential operators on $\Omega^r(M)$ will be considered with respect to the decompositions~\eqref{bigwedge^rTM^*} and~\eqref{Omega^r(M)}.

Let $d$ and $\tilde d$ denote the exterior derivatives on $\Omega(M)$ and $\Omega(N)$, respectively. We have \cite[Lemma~10.1]{AlvCalaza2017}
    \begin{equation}\label{d}
      d\equiv
        \begin{pmatrix}
          \tilde d & 0 \\
          \frac{d}{d\rho} & -\tilde d
        \end{pmatrix}\;.
    \end{equation}

Fix a general adapted metric $\tilde g$ on $N$. For $u>0$, the metric $g=\rho^{2u}\tilde g+d\rho^2$ is a general adapted metric on $M$. The induced metrics on $\bigwedge TM^*$ and $\bigwedge TN^*$ are also denoted by $g$ and $\tilde g$, respectively. Fix some degree $r\in\{0,1,\dots,n\}$, and, to simplify the expressions, let
	\begin{equation}\label{kappa=(n-2r-1)u/2}
    		\kappa=(n-2r-1)\textstyle{\frac{u}{2}}\;.
  	\end{equation}  
According to~\eqref{bigwedge^rTM^*},
  \begin{equation}\label{g}
    g\equiv\rho^{-2ru}\,\tilde g\oplus\rho^{-2(r-1)u}\,\tilde g
  \end{equation}
on $\bigwedge^rTM^*$. Choose an orientation on an open subset $W\subset N$, and let $\tilde\omega$ denote the corresponding $\tilde g$-volume form on $W$. Consider the orientation on $W\times\R_+\subset M$ so that the corresponding $g$-volume form is
  \begin{equation}\label{omega}
    \omega=\rho^{(n-1)u}\,d\rho\wedge\tilde\omega\;.
  \end{equation}
The corresponding Hodge star operators on $\bigwedge T(W\times\R_+)^*$ and $\bigwedge TW^*$ will be denoted by $\star$ and $\tilde\star$, respectively. Like in \cite[Lemma~10.2]{AlvCalaza2017}, from~\eqref{g} and~\eqref{omega}, it follows that
    \begin{equation}\label{star}
      \star\equiv
        \begin{pmatrix}
          0 & \rho^{2(\kappa+u)}\,\tilde\star \\
          (-1)^r\rho^{2\kappa}\,\tilde\star & 0
        \end{pmatrix}
    \end{equation}
on $\bigwedge^rT(W\times\R_+)^*$. Let $L^2\Omega^r(M)=L^2\Omega^r(M,g)$ and $L^2\Omega^r(N)=L^2\Omega^r(N,\tilde g)$. From~\eqref{g} and~\eqref{omega}, we also get that~\eqref{Omega^r(M)} induces the identity of Hilbert spaces\footnote{Recall that, for Hilbert spaces $\fH'$ and $\fH''$, with scalar products $\langle\ ,\ \rangle'$ and $\langle\ ,\ \rangle''$, the notation $\fH'\,\widehat{\otimes}\,\fH''$ is used for the Hilbert space tensor product. This is the Hilbert space completion of the algebraic tensor product $\fH'\otimes \fH''$ with respect to the scalar product defined by $\langle u'\otimes u'',v'\otimes v''\rangle=\langle u',v'\rangle'\,\langle u'',v''\rangle''$.}
  \begin{equation}\label{L^2Omega^r(M)}
    L^2\Omega^r(M)
    \equiv\left(L^2_{\kappa,+}\,\widehat{\otimes}\,L^2\Omega^r(N)\right)
    \oplus\left(L^2_{\kappa+u,+}\,\widehat{\otimes}\,L^2\Omega^{r-1}(N)\right)\;.
  \end{equation}
Let $\delta$ and $\tilde\delta$ denote the exterior coderivatives on $\Omega(M)$ and $\Omega(N)$, respectively. Like in \cite[Lemma~10.3]{AlvCalaza2017}, using~\eqref{d},~\eqref{star} and~\eqref{[d/d rho,rho^a right]}, we get
    \begin{equation}\label{delta}
      \delta\equiv
        \begin{pmatrix}
          \rho^{-2u}\,\tilde\delta & -\frac{d}{d\rho}-2(\kappa+u)\rho^{-1} \\
          0 & -\rho^{-2u}\,\tilde\delta
        \end{pmatrix}
    \end{equation}
on $\Omega^r(M)$. Let $\D$ and $\widetilde{\D}$ denote the Laplacians on $\Omega(M)$ and $\Omega(N)$, respectively. Like in \cite[Corollary~10.4]{AlvCalaza2017}, from~\eqref{d},~\eqref{delta} and~\eqref{[d/d rho,rho^a right]}, it follows that
    \begin{equation}\label{D}
      \D\equiv
        \begin{pmatrix}
          P & -2u\rho^{-1}\,\tilde d \\
          -2u\rho^{-2u-1}\,\tilde\delta & Q
        \end{pmatrix}
      \end{equation}
  on $\Omega^r(M)$, where
  \begin{align}
    P&=\rho^{-2u}\,\widetilde{\D}-\textstyle{\frac{d^2}{d\rho^2}-2\kappa\rho^{-1}\,\frac{d}{d\rho}}\;,\label{P in D}\\
    Q&=\rho^{-2u}\,\widetilde{\D}-\textstyle{\frac{d^2}{d\rho^2}-2(\kappa+u)\frac{d}{d\rho}\,\rho^{-1}}\;.\label{Q in D}
  \end{align}

\subsection{Witten's perturbation on a cone}\label{ss: Witten cone}

Let $d_s$, $\delta_s$, $D_s$ and $\D_s$ ($s\ge0$) denote the Witten's perturbations of $d$, $\delta$, $D$ and $\D$ induced by the function $f=\pm\frac{1}{2}\rho^2$ on $M$. The more explicit notation $d^\pm_s$, $\delta^\pm_s$, $D^\pm_s$ and $\D^\pm_s$ may be used if needed. In this case, $df=\pm\rho\,d\rho$. According to~\eqref{Omega^r(M)},
  	\[
    		\rho\,d\rho\wedge\equiv
      			\begin{pmatrix}
        				0 & 0\\
        				\rho & 0
      			\end{pmatrix}\;,\quad
    		-\rho\,d\rho\lrcorner\equiv
      			\begin{pmatrix}
        				0 & \rho\\
        				0 & 0
      			\end{pmatrix}\;.
    	\]
So, by~\eqref{d},~\eqref{delta},~\eqref{d_s} and~\eqref{delta_s},
	\begin{align}
      		d_s&\equiv
        			\begin{pmatrix}
          			\tilde d & 0 \\
          			\frac{d}{d\rho}\pm s\rho & -\tilde d
        			\end{pmatrix}\;,
		\label{d_s^pm}
		\\
      		\delta_s&\equiv
        			\begin{pmatrix}
          			\rho^{-2u}\,\tilde\delta & -\frac{d}{d\rho}-2(\kappa+u)\rho^{-1}\pm s\rho\\
          			0 & -\rho^{-2u}\,\tilde\delta
        			\end{pmatrix}\;,
		\label{delta_s^pm}
    	\end{align}
on $\Omega^r(M)$. Now,
  \[
    R=\pm\rho(d\rho\!\wedge-\,d\rho\lrcorner)\equiv\pm
       \begin{pmatrix}
          0 & \rho\\
          \rho & 0
        \end{pmatrix}\;,
  \]
and therefore
  \begin{equation}\label{R^2=rho^2}
    R^2\equiv
       \begin{pmatrix}
          \rho^2 & 0\\
          0 & \rho^2
        \end{pmatrix}
      \equiv\rho^2\;.
  \end{equation}
Like in \cite[Lemma~10.6]{AlvCalaza2017}, we get
	\begin{equation}\label{RD+DR=V}
		RD+DR=\mp V
	\end{equation}	
on $\Omega^r(M)$, where $V$ is given by~\eqref{V}. As a consequence of~\eqref{Delta_s},~\eqref{D} and~\eqref{RD+DR=V}, we obtain
    	\begin{align}
      		\D_s&\equiv
        			\begin{pmatrix}
          			P_s & -2u\rho^{-1}\tilde d \\
          			-2u\rho^{-2u-1}\tilde\delta & Q_s
        			\end{pmatrix}\label{D_s}\\
		\intertext{on $\Omega^r(M)$, where}
		P_s&=\rho^{-2u}\widetilde{\D}+H-2\kappa\rho^{-1}\,\textstyle{\frac{d}{d\rho}}\mp s(1+2\kappa)\;,\label{P_s}\\
      		Q_s&=\rho^{-2u}\widetilde{\D}+H-2(\kappa+u)\textstyle{\frac{d}{d\rho}}\,\rho^{-1}\mp s(-1+2(\kappa+u))\;.\label{Q_s}
    \end{align}

\section{Splitting of the Witten's complex on a cone}\label{s: splitting}

\subsection{Spectral decomposition on the link of the cone}\label{ss: spectral decomposition on the link}

Theorem~\ref{t:  spectrum of Delta_max/min} is proved by induction on the depth. Thus, with the notation of Section~\ref{s: Witten cone}, suppose that $\tilde g$ is good, and $\widetilde\D_{\text{\rm max/min}}$ satisfies the statement of Theorem~\ref{t:  spectrum of Delta_max/min}. Moreover suppose that $g$ is also good; that is, $u\le1$.

Let $\widetilde{\HH}_{\text{\rm max/min}}=\ker\widetilde{D}_{\text{\rm max/min}}=\ker\widetilde{\D}_{\text{\rm max/min}}$, which is a graded subspace of $\Omega(N)\cap L^2\Omega(N)$. For every degree $r$, let $\widetilde{\sR}_{\text{\rm max/min},r-1},\widetilde{\sR}_{\text{\rm max/min},r}^*\subset L^2\Omega^r(N)$ be the images of $\tilde d_{\text{\rm max/min},r-1}$ and $\tilde\delta_{\text{\rm max/min},r}$, respectively, which are closed subspaces. By restriction, $\widetilde{\D}_{\text{\rm max/min}}$ defines self-adjoint operators in $\widetilde{\sR}_{\text{\rm max/min},r-1}$ and $\widetilde{\sR}_{\text{\rm max/min},r-1}^*$, with the same eigenvalues \cite[Section~5.1]{AlvCalaza2017}. For any eigenvalue $\tilde\lambda$ of the restriction of $\widetilde{\D}_{\text{\rm max/min}}$ to $\widetilde{\sR}_{\text{\rm max/min},r-1}$, let $\widetilde{\sR}_{\text{\rm max/min},r-1,\tilde\lambda}$ and $\widetilde{\sR}_{\text{\rm max/min},r-1,\tilde\lambda}^*$ denote the corresponding $\tilde\lambda$-eigenspaces. We have\footnote{Consider a family of Hilbert spaces, $\fH_a$ with scalar product $\langle\ ,\ \rangle_a$. Recall that the Hilbert space direct sum, $\widehat{\bigoplus}_a\fH^a$, is the Hilbert space completion of the algebraic direct sum, $\bigoplus_a\fH^a$, with respect to the scalar product $\langle(u^a),(v^a)\rangle=\sum_a\langle u^a,v^a\rangle_a$. Thus $\widehat{\bigoplus}_a\fH^a=\bigoplus_a\fH^a$ if and only if the family is finite.}
  \begin{equation}\label{L^2 Omega^r(N)}
    L^2\Omega^r(N)=\widetilde{\HH}_{\text{\rm max/min}}^r\oplus
    \widehat{\bigoplus_{\tilde\lambda,\tilde\lambda'}}\left(\widetilde{\sR}_{\text{\rm max/min},r-1,\tilde\lambda}
    \oplus\widetilde{\sR}^*_{\text{\rm max/min},r,\tilde\lambda'}\right)\;,
  \end{equation}
where $\tilde\lambda$ and $\tilde\lambda'$ run in the spectrum of the restrictions of $\widetilde{\D}_{\text{\rm max/min}}$ to $\widetilde{\sR}_{\text{\rm max/min},r-1}$ and $\widetilde{\sR}_{\text{\rm max/min},r}^*$, respectively.

\subsection{Subcomplexes of length one}
\label{ss: subcomplexes 1}

Given $0\neq\gamma\in \widetilde{\HH}_{\text{\rm max/min}}^r$, consider the canonical identities 
\begin{equation}\label{C^infty_+ equiv C^infty_+ gamma}
	C^\infty_+\equiv C^\infty_+\,\gamma\subset\Omega^r(M)\;,\quad
	C^\infty_+\equiv C^\infty_+\,d\rho\wedge\gamma\subset\Omega^{r+1}(M)\;.
\end{equation}
The following result follows from~\eqref{d_s^pm} and~\eqref{delta_s^pm}.

\begin{lem}\label{l: subcomplexes 1}
  For $s\ge0$, $d_s$ and $\delta_s$ define maps
    \begin{center}
      \begin{picture}(260,32) 
        \put(0,13){$0$}
        \put(64,13){$C^\infty_+\,\gamma$}
        \put(147,13){$C^\infty_+\,d\rho\wedge\gamma$}
        \put(252,13){$0\;.$}
        \put(260,13){$.$}
        \put(21,22){\Small$d_{s,r-1}$}
        \put(21,3){\Small$\delta_{s,r-1}$}
        \put(109,22){\Small$d_{s,r}$}
        \put(109,3){\Small$\delta_{s,r}$}
        \put(209,22){\Small$d_{s,r+1}$}
        \put(209,3){\Small$\delta_{s,r+1}$}
        \put(10,17){\vector(1,0){47}}
        \put(57,14){\vector(-1,0){47}}
        \put(93,17){\vector(1,0){47}}
        \put(140,14){\vector(-1,0){47}}
        \put(198,17){\vector(1,0){47}}
        \put(245,14){\vector(-1,0){47}}
      \end{picture}
    \end{center}
   Moreover, using~\eqref{C^infty_+ equiv C^infty_+ gamma},
     \[
       d_{s,r}=\textstyle{\frac{d}{d\rho}}\pm s\rho\;,\quad
       \delta_{s,r}=-\textstyle{\frac{d}{d\rho}}-2\kappa\rho^{-1}\pm s\rho\;.
     \]
\end{lem}

Let $\EE_{\gamma,0}$ denote the subcomplex of length one of $(\Omega(M),d_s)$ defined by
  \[
    \EE_{\gamma,0}^r=C^\infty_{+,0}\,\gamma\equiv C^\infty_{+,0}\;,\quad
    \EE_{\gamma,0}^{r+1}=C^\infty_{+,0}\,d\rho\wedge\gamma\equiv C^\infty_{+,0}\;.
  \]
The closure of $\EE_{\gamma,0}$ in $L^2\Omega(M)$ is denoted by $L^2\EE_\gamma$. By~\eqref{L^2Omega^r(M)},
  \begin{gather*}
    L^2\EE_\gamma^r=L^2_{\kappa,+}\,\gamma\equiv L^2_{\kappa,+}\;,\quad
    L^2\EE_\gamma^{r+1}=L^2_{\kappa,+}\,d\rho\wedge\gamma\equiv L^2_{\kappa,+}\;.
  \end{gather*}
  
Assume now that $s>0$. With the notation of Section~\ref{ss: complex 1}, consider the real version of the elliptic complex $(E,d)$ determined by $s$ and $\kappa$ (given by~\eqref{kappa=(n-2r-1)u/2}). Using Lemma~\ref{l: subcomplexes 1} and~\eqref{[d/d rho,rho^a right]}, like in \cite[Proposition~12.3]{AlvCalaza2017}, we get the following.
  
\begin{prop}\label{p: EE_gamma,i}
  The operator $\rho^\kappa:L^2_{\kappa,+}\to L^2_+$ defines a unitary isomorphism $L^2\EE_\gamma\to L^2(E)$, which restricts to an isomorphism of complexes, $(\EE_{\gamma,0},d_s)\to(C^\infty_0(E),d)$, up to a shift of degree.
\end{prop}

By Proposition~\ref{p: EE_gamma,i}, $(\EE_{\gamma,0},d_s)$ has a maximum/minimum Hilbert complex extension in $L^2\EE_\gamma$. Let $(\sD_\gamma,\mathbf{d}_{s,\gamma})$ be the maximum/minimum Hilbert complex extension of $(\EE_{\gamma,0},d_s)$ if $\gamma\in\widetilde{\HH}_{\text{\rm max/min}}^r$, and $\mathbf{\D}_{s,\gamma}$ the corresponding Laplacian. Let $\HH_{s,\gamma}=\HH_{s,\gamma}^r\oplus\HH_{s,\gamma}^{r+1}=\ker\mathbf{\D}_{s,\gamma}$, with the induced grading. The more explicit notation $\mathbf{d}^\pm_{s,\gamma}$, $\mathbf{\D}^\pm_{s,\gamma}$ and $\HH_{s,\gamma}^\pm=\HH_{s,\gamma}^{\pm,r}\oplus\HH_{s,\gamma}^{\pm,r+1}$ may be also used.

\begin{cor}\label{c: mathbf D^pm_s,gamma}
    		\begin{enumerate}[{\rm(}i\/{\rm)}]
    
      			\item\label{i: mathbf d_s,gamma is discrete} $\bDelta_{s,\gamma}$ has a discrete spectrum.
			
			\item\label{i: dim HH^pm,r_s,gamma} The dimensions of $\HH^{\pm,r}_{s,\gamma}$ and $\HH^{\pm,r+1}_{s,\gamma}$ are given in Table~\ref{table: dim HH^pm,r_s,gamma}.
      
      			\item\label{i: langle he_s,e_s rangle to 1} If $e_s\in\HH_{s,\gamma}$ with norm one for every $s$, and $h$ is a bounded measurable function on $\R_+$ with $h(\rho)\to1$ as $\rho\to0$, then $\langle he_s,e_s\rangle\to1$ as $s\to\infty$.
      
      			\item\label{i: O(s)} All nonzero eigenvalues of $\mathbf{\D}_{s,\gamma}$ are positive and in $O(s)$ as $s\to\infty$.
      
    		\end{enumerate}
\end{cor}

\begin{table}[h]
\renewcommand{\arraystretch}{1.3}
\begin{tabular}{c|c|c|c|c|c|c|c|c|}
\cline{2-9}
& \multicolumn{4}{c|}{$\gamma\in\widetilde{\HH}_{\text{\rm max}}^r$} & \multicolumn{4}{c|}{$\gamma\in\widetilde{\HH}_{\text{\rm min}}^r$} \\
\cline{2-9}
& $\HH_{s,\gamma}^{+,r}$ & $\HH_{s,\gamma}^{+,r+1}$ & $\HH_{s,\gamma}^{-,r}$ & $\HH_{s,\gamma}^{-,r+1}$ & $\HH_{s,\gamma}^{+,r}$ & $\HH_{s,\gamma}^{+,r+1}$ & $\HH_{s,\gamma}^{-,r}$ & $\HH_{s,\gamma}^{-,r+1}$ \\
\hline
\multicolumn{1}{|c|}{$\kappa\ge\frac{1}{2}$} & \multirow{2}{*}{$1$} & \multicolumn{2}{c|}{\multirow{3}{*}{$0$}} & \multirow{2}{*}{$0$} & $1$ & \multicolumn{2}{c|}{\multirow{3}{*}{$0$}} & $0$ \\
\cline{1-1}\cline{6-6}\cline{9-9}
\multicolumn{1}{|c|}{$|\kappa|<\frac{1}{2}$} && \multicolumn{2}{c|}{} && \multirow{2}{*}{$0$} & \multicolumn{2}{c|}{} & \multirow{2}{*}{$1$} \\
\cline{1-2}\cline{5-5}
\multicolumn{1}{|c|}{$\kappa\le-\frac{1}{2}$} & $0$ & \multicolumn{2}{c|}{} & $1$ && \multicolumn{2}{c|}{} & \\
\hline
\end{tabular}
\vspace{1mm}
\caption{Dimensions of $\HH^{\pm,r}_{s,\gamma}$ and $\HH^{\pm,r+1}_{s,\gamma}$}
\label{table: dim HH^pm,r_s,gamma}
\end{table}

\begin{proof}
	This follows from  Propositions~\ref{p: EE_gamma,i} and~\ref{p: 1}, Corollary~\ref{c: h hat chi_PP,0}, Section~\ref{sss: complex 1, Laplacian}, and the choice made to define $\bd_{s,\gamma}$.
\end{proof}

\subsection{Subomplexes of length two}
\label{ss: subcomplexes 2}

Let $\mu=\sqrt{\tilde\lambda}$ for an eigenvalue $\tilde\lambda$ of the restriction of $\widetilde{\D}_{\text{\rm max/min}}$ to $\widetilde{\sR}_{\text{\rm max/min},r-1}$. According to \cite[Section~5.1]{AlvCalaza2017}, there are nonzero differential forms,
  	\[
    		\alpha\in\widetilde{\sR}_{\text{\rm max/min},r-1,\tilde\lambda}\subset\Omega^r(N)\;,\quad
    		\beta\in\widetilde{\sR}_{\text{\rm max/min},r-1,\tilde\lambda}^*\subset\Omega^{r-1}(N)\;,
  	\]
such that $\tilde d\beta=\mu\alpha$ and $\tilde\delta\alpha=\mu\beta$. Consider the canonical identities
    	\begin{gather}
      		C^\infty_+\equiv C^\infty_+\,\beta\subset\Omega^{r-1}(M)\;,\quad
		C^\infty_+\equiv C^\infty_+\,d\rho\wedge\alpha\subset\Omega^{r+1}(M)\;,
		\label{C^infty_+ equiv C^infty_+ beta}\\
      		C^\infty_+\oplus C^\infty_+\equiv C^\infty_+\,\alpha+C^\infty_+\,d\rho\wedge\beta
		\subset\Omega^r(M)
		\label{C^infty_+ oplus C^infty_+ equiv C^infty_+ alpha + C^infty_+ d rho wedge beta}\;.
    	\end{gather}
The following result follows from~\eqref{d_s^pm} and~\eqref{delta_s^pm}.

\begin{lem}\label{l: subcomplexes 2}
  For $s\ge0$, $d_s$ and $\delta_s$ define maps
    \begin{center}
      \begin{picture}(264,67) 
        \put(0,48){$0$}
        \put(64,48){$C^\infty_+\,\beta$}        
        \put(148,48){$C^\infty_+\,\alpha+C^\infty_+\,d\rho\wedge\beta$}
        \put(148,13){$C^\infty_+\,d\rho\wedge\alpha$}
        \put(256,13){$0\;.$}
        \put(21,57){\Small$d_{s,r-2}$}
        \put(21,38){\Small$\delta_{s,r-2}$}
        \put(105,57){\Small$d_{s,r-1}$}
        \put(105,38){\Small$\delta_{s,r-1}$}
        \put(110,22){\Small$d_{s,r}$}
        \put(110,3){\Small$\delta_{s,r}$}
        \put(211,22){\Small$d_{s,r+1}$}
        \put(211,3){\Small$\delta_{s,r+1}$}
        \put(10,52){\vector(1,0){47}}
        \put(57,49){\vector(-1,0){47}}
        \put(94,52){\vector(1,0){47}}
        \put(141,49){\vector(-1,0){47}}
        \put(94,17){\vector(1,0){47}}
        \put(141,14){\vector(-1,0){47}}
        \put(200,17){\vector(1,0){47}}
        \put(247,14){\vector(-1,0){47}}
      \end{picture}
    \end{center}
  Moreover, according to~\eqref{C^infty_+ equiv C^infty_+ beta} and~\eqref{C^infty_+ oplus C^infty_+ equiv C^infty_+ alpha + C^infty_+ d rho wedge beta},
    \begin{align*}
      d_{s,r-1}&=
        \begin{pmatrix}
          \mu\\
          \frac{d}{d\rho}\pm s\rho
        \end{pmatrix}\;,\\
      \delta_{s,r-1}&=
        \begin{pmatrix}
          \mu\rho^{-2u} & -\frac{d}{d\rho}-2(\kappa+u)\rho^{-1}\pm s\rho
        \end{pmatrix}\;,\\
      d_{s,r}&=
        \begin{pmatrix}
          \frac{d}{d\rho}\pm s\rho & -\mu
        \end{pmatrix}\;,\\
      \delta_{s,r}&=
        \begin{pmatrix}
          -\frac{d}{d\rho}-2\kappa\rho^{-1}\pm s\rho\\
          -\mu\rho^{-2u}
        \end{pmatrix}\;.
    \end{align*}
\end{lem}

Let $\FF_{\alpha,\beta,0}=\FF_{\alpha,\beta,0}^{r-1}\oplus\FF_{\alpha,\beta,0}^r\oplus\FF_{\alpha,\beta,0}^{r+1}$ denote the subcomplex of length two of $(\Omega(M),d_s)$ defined by
  \begin{gather*}
    \FF_{\alpha,\beta,0}^{r-1}=C^\infty_{+,0}\,\beta\equiv C^\infty_{+,0}\;,\quad
    \FF_{\alpha,\beta,0}^{r+1}=C^\infty_{+,0}\,d\rho\wedge\alpha\equiv C^\infty_{+,0}\;,\\
    \FF_{\alpha,\beta,0}^r=C^\infty_{+,0}\,\alpha+C^\infty_{+,0}\,d\rho\wedge\beta
    \equiv C^\infty_{+,0}\oplus C^\infty_{+,0}\;.
  \end{gather*}
The closure of $\FF_{\alpha,\beta,0}$ in $L^2\Omega(M)$ is denoted by $L^2\FF_{\alpha,\beta}$. By~\eqref{L^2Omega^r(M)},
  \begin{gather*}
    L^2\FF_{\alpha,\beta}^{r-1}=L^2_{\kappa+u,+}\,\beta\equiv L^2_{\kappa+u,+}\;,\quad
    L^2\FF_{\alpha,\beta}^{r+1}=L^2_{\kappa,+}\,d\rho\wedge\alpha\equiv L^2_{\kappa,+}\;,\\
    L^2\FF_{\alpha,\beta}^r=L^2_{\kappa,+}\,\alpha+L^2_{\kappa+u,+}\,d\rho\wedge\beta
    \equiv L^2_{\kappa,+}\oplus L^2_{\kappa+u,+}\;.
  \end{gather*}

Assume now that $s>0$. With the notation of Section~\ref{ss: complex 2}, consider the real version of the elliptic complex $(F,d)$ determined by $s$ and $\kappa$ (given by~\eqref{kappa=(n-2r-1)u/2}). Using Lemma~\ref{l: subcomplexes 2} and~\eqref{[d/d rho,rho^a right]}, we get the following (cf.\ \cite[Proposition~12.9]{AlvCalaza2017}).
  
\begin{prop}\label{p: FF_alpha,beta}
  	If $u<1$, then $\rho^\kappa:L^2_{\kappa,+}\to L^2_+$ and $\rho^{\kappa+u}:L^2_{\kappa+u,+}\to L^2_+$ define a unitary isomorphism $L^2\FF_{\alpha,\beta}\to L^2(F)$, which restricts to an isomorphism of complexes, $(\FF_{\alpha,\beta,0},d_s)\to(C^\infty_0(F),d)$, up to a shift of degree. 
\end{prop}

By Proposition~\ref{p: FF_alpha,beta}, $(\FF_{\alpha,\beta,0},d_s)$ has a maximum/minimum Hilbert complex extension in $L^2\FF_{\alpha,\beta}$. Let $(\sD_{\alpha,\beta},\mathbf{d}_{s,\alpha,\beta})$ be the maximum/minimum Hilbert complex extension of $(\FF_{\alpha,\beta,0},d_s)$ if $\alpha\in\widetilde{\sR}_{\text{\rm max/min},r-1,\tilde\lambda}$ and $\beta\in\widetilde{\sR}^*_{\text{\rm max/min},r-1,\tilde\lambda}$. Let $\mathbf{\D}_{s,\alpha,\beta}$ denote the corresponding Laplacian. The more explicit notation $\mathbf{d}^\pm_{s,\alpha,\beta}$ and $\mathbf{\D}^\pm_{s,\alpha,\beta}$ may be used.

\begin{cor}\label{c: mathbf D^pm_s,alpha,beta}
		\begin{enumerate}[{\rm(}i\/{\rm)}]
  
    			\item\label{i: mathbf d_s,alpha,beta is discrete} $\bDelta_{s,\alpha,\beta}$ has a discrete spectrum.
    
    			\item\label{i: eigenvalues of mathbf Delta_s,alpha,beta in O(s)} The eigenvalues of $\mathbf{\D}_{s,\alpha,\beta}$ are positive and in $O(s)$ as $s\to\infty$.
  
  		\end{enumerate}
\end{cor}

\begin{proof}
  	In the case $u<1$, this follows from Proposition~\ref{p: FF_alpha,beta} and Corollary~\ref{c: 2, spectrum}. In the case $u=1$, this is the content of \cite[Proposition~12.11]{AlvCalaza2017}.
\end{proof}

\begin{rem}
	According to~\eqref{RD+DR=V}--\eqref{Q_s}, we have
  		\begin{alignat*}{2}
    			\D_s&\equiv H-2\kappa\rho^{-1}\,\textstyle{\frac{d}{d\rho}}\mp s(1+2\kappa)
			&\quad&\text{on $C^\infty_+\equiv C^\infty_+\,\gamma$}\;,\\
			\D_s&\equiv H-2\kappa\textstyle{\frac{d}{d\rho}}\,\rho^{-1}\mp s(-1+2\kappa)
			&\quad&\text{on $C^\infty_+\equiv C^\infty_+\,d\rho\wedge\gamma$}\;,\\
			\D_s&\equiv H-2(\kappa+u)\rho^{-1}\,\textstyle{\frac{d}{d\rho}}
			+\mu^2\rho^{-2u}\mp s(1+2(\kappa+u))
			&\quad&\text{on $C^\infty_+\equiv C^\infty_+\,\beta$}\;,\\
			\D_s&\equiv H-2\kappa\textstyle{\frac{d}{d\rho}}\,\rho^{-1}+\mu^2\rho^{-2u}\mp s(-1+2\kappa)
			&\quad&\text{on $C^\infty_+\equiv C^\infty_+\,d\rho\wedge\alpha$}\;,
		\end{alignat*}
	and
		\[
			\D_s\equiv
      				\begin{pmatrix}
        					P_{\mu,s} & -2\mu u\rho^{-1} \\
        					-2\mu u\rho^{-2u-1} & Q_{\mu,s}
      				\end{pmatrix}
		\]
	on $C^\infty_+\oplus C^\infty_+\equiv C^\infty_+\,\alpha+C^\infty_+\,d\rho\wedge\beta$, where
		\begin{align*}
			P_{\mu,s}&=H-2\kappa\rho^{-1}\,\textstyle{\frac{d}{d\rho}}
			+\mu^2\rho^{-2u}\mp s(1+2\kappa)\;,\\
    			Q_{\mu,s}&=H-2(\kappa+u)\textstyle{\frac{d}{d\rho}}\,\rho^{-1}
			+\mu^2\rho^{-2u}\mp s(-1+2(\kappa+u))\;.
  		\end{align*}
	So the results of Section~\ref{s: P, Q, W} could be applied to these expressions. We opted for analyzing first the complexes of Section~\ref{s: 2 simple types of elliptic complexes} for the sake of simplicity because we have $a=b=0$, $L^2_+$ is used instead of $L^2_{\kappa,+}$ or $L^2_{\kappa+u,+}$, and Remark~\ref{r: 1} is directly applied.
\end{rem}

\subsection{Splitting into subcomplexes}\label{ss: splitting}

Let $\CC_{\text{\rm max/min},0}$ denote an orthonormal frame of $\widetilde{\HH}_{\text{\rm max/min}}$ consisting of homogeneous differential forms. For every positive eigenvalue $\mu$ of $\widetilde{D}_{\text{\rm max/min}}$, let $\CC_{\text{\rm max/min},\mu}$ be an orthonormal frame of the $\mu$-eigenspace of $\widetilde{D}_{\text{\rm max/min}}$ consisting of differential forms $\alpha+\beta$ like in Section~\ref{ss: subcomplexes 2}. Then let
  	\[
    		\mathbf{d}_{s,\text{\rm max/min}}=\bigoplus_\gamma\mathbf{d}_{s,\gamma}
		\oplus\widehat{\bigoplus_\mu}\bigoplus_{\alpha+\beta}\mathbf{d}_{s,\alpha,\beta}\;,
  	\]
where $\gamma$ runs in $\CC_{\text{\rm max/min},0}$, $\mu$ runs in the positive spectrum of $\widetilde{D}_{\text{\rm max/min}}$, and $\alpha+\beta$ runs in $\CC_{\text{\rm max/min},\mu}$. The notation $\mathbf{d}^\pm_{s,\text{\rm max/min}}$ may be also used when $\mathbf{d}^\pm_{s,\gamma}$ and $\mathbf{d}^\pm_{s,\alpha,\beta}$ are considered.

\begin{prop}\label{p: splitting}
  	We have $d_{s,\text{\rm max/min}}=\mathbf{d}_{s,\text{\rm max/min}}$.
\end{prop}

\begin{proof}
  This follows like \cite[Proposition~12.12]{AlvCalaza2017}, using \cite[Lemma~5.2]{AlvCalaza2017}, \cite[Lemma~3.6 and~(2.38b)]{BruningLesch1992},~\eqref{Omega^r(M) with d rho} and~\eqref{L^2 Omega^r(N)}.
\end{proof}


Let $\HH_{s,\text{\rm max/min}}=\bigoplus_r\HH_{s,\text{\rm max/min}}^r=\ker\D_{s,\text{\rm max/min}}$, with the induced grading. The superindex ``$\pm$'' may be added to this notation to indicate that we are referring to $\D_{s,\text{\rm max/min}}^\pm$.

\begin{cor}\label{c:d^pm_s,max/min are discrete}
  		\begin{enumerate}[{\rm(}i\/{\rm)}]
  
    			\item\label{i: d^pm_s,max/min is discrete} $\D_{s,\text{\rm max/min}}$ has a discrete spectrum.
    
    			\item\label{i: HH_max/min^+,r} Table~\ref{table: dim HH^pm,r_s,max/min} describes the isomorphism class of $\HH_{s,\text{\rm max/min}}^{\pm,*}$.
    
    			\item\label{i: langle he_s^pm, e_s^pm rangle to 1} If $e_s\in\HH_{s,\text{\rm max/min}}$ has norm one for every $s$, and $h$ is a bounded measurable function on $\R_+$ with $h(\rho)\to1$ as $\rho\to0$, then $\langle he_s,e_s\rangle\to1$ as $s\to\infty$.
    
    			\item\label{i: lambda^pm_s,max/min,k in O(s)} Let $0\le\lambda_{s,\text{\rm max/min},0}\le\lambda_{s,\text{\rm max/min},1}\le\cdots$ be the eigenvalues of $\D_{s,\text{\rm max/min}}$, repeated according to their multiplicities. Given $k\in\N$, if $\lambda_{s,\text{\rm max/min},k}>0$ for some $s$, then $\lambda_{s,\text{\rm max/min},k}>0$ for all $s$, and $\lambda_{s,\text{\rm max/min},k}\in O(s)$ as $s\to\infty$.
  
  			\item\label{i: liminf_k lambda^pm_s,max/min,k k^-theta > 0} There is some $\theta>0$ such that $\liminf_k\lambda_{s,\text{\rm max/min},k}k^{-\theta}>0$.
  
  		\end{enumerate}
\end{cor}

\begin{table}[h]
\renewcommand{\arraystretch}{1.3}
\begin{tabular}{c|c|c|c|c|}
\cline{2-5}
& $\HH_{s,\text{\rm max}}^{+,r}$ & $\HH_{s,\text{\rm max}}^{-,r+1}$ & $\HH_{s,\text{\rm min}}^{+,r}$ & $\HH_{s,\text{\rm min}}^{-,r+1}$ \\
\hline
\multicolumn{1}{|c|}{$\kappa\ge\frac{1}{2}$} & \multirow{2}{*}{$H_{\text{\rm max}}^r(N)$} & \multirow{2}{*}{$0$} & $H_{\text{\rm min}}^r(N)$ & $0$ \\
\cline{1-1}\cline{4-5}
\multicolumn{1}{|c|}{$|\kappa|<\frac{1}{2}$} &&& \multirow{2}{*}{$0$} & \multirow{2}{*}{$H_{\text{\rm min}}^r(N)$} \\
\cline{1-3}
\multicolumn{1}{|c|}{$\kappa\le-\frac{1}{2}$} & $0$ & $H_{\text{\rm max}}^r(N)$ && \\
\hline
\end{tabular}
\vspace{1mm}
\caption{Spaces isomorphic to $\HH_{s,\text{\rm max/min}}^{\pm,*}$}
\label{table: dim HH^pm,r_s,max/min}
\end{table}

\begin{proof}
  	In the case $u=1$, this result was already shown in \cite[Corollary~12.13]{AlvCalaza2017}. So we consider only the case $0<u<1$. For all $\gamma$, $\mu$ and $\alpha+\beta$ as above, $\bDelta_{s,\gamma}$ and $\bDelta_{s,\alpha,\beta}$ have a discrete spectrum by Corollaries~\ref{c: mathbf D^pm_s,gamma}~\eqref{i: mathbf d_s,gamma is discrete} and~\ref{c: mathbf D^pm_s,alpha,beta}~\eqref{i: mathbf d_s,alpha,beta is discrete}. Moreover the union of their spectra has no accumulation points according to Section~\ref{s: 2 simple types of elliptic complexes} and since $\widetilde\D_{\text{\rm max/min}}$ is discrete. Then~\eqref{i: d^pm_s,max/min is discrete} follows by Proposition~\ref{p: splitting}. 
	
	Now, properties~\eqref{i: HH_max/min^+,r}--\eqref{i: lambda^pm_s,max/min,k in O(s)} follow directly from Corollaries~\ref{c: mathbf D^pm_s,gamma} and~\ref{c: mathbf D^pm_s,alpha,beta}, and Proposition~\ref{p: splitting}.
  
  To prove~\eqref{i: liminf_k lambda^pm_s,max/min,k k^-theta > 0}, let $0\le\tilde\lambda_{\text{\rm max/min},0}\le\tilde\lambda_{\text{\rm max/min},1}\le\cdots$ denote the eigenvalues of $\widetilde{\D}_{\text{\rm max/min}}$, repeated according to their multiplicities. Since $N$ satisfies Theorem~\ref{t:  spectrum of Delta_max/min}~\eqref{i: liminf_k lambda_max/min,k k^-theta > 0 for some theta>0} with $\tilde g$, there is some $C_0,\theta_0>0$ such that
    \begin{equation}\label{tilde lambda_max/min,ell}
      \tilde\lambda_{\text{\rm max/min},\ell}\ge C_0\ell^{\theta_0}
    \end{equation}
  for all $\ell$ large enough. Consider the counting function
    \[
      \fN^\pm_{s,\text{\rm max/min}}(\lambda)=\#\left\{\,k\in\N\mid\lambda^\pm_{s,\text{\rm max/min},k}<\lambda\,\right\}\quad(\lambda>0)\;.
    \]
  From Proposition~\ref{p: 2, ker Delta_max/min,r = 0}, Corollary~\ref{c: 2, spectrum},~\eqref{eigenvalues, 1, AA_1}--\eqref{eigenvalues, 1, BB_2},~\eqref{eigenvalues ge ..., PP_1},~\eqref{eigenvalues ge ..., PP_2},~\eqref{eigenvalues ge ..., QQ_1},~\eqref{eigenvalues ge ..., QQ_2},~\eqref{eigenvalues ge ..., WW_2,1, even case},~\eqref{eigenvalues ge ..., WW_2,1, odd case} and~\eqref{tilde lambda_max/min,ell}, and the choices made to define $\bd_\gamma$ and $\bd_{\alpha,\beta}$ (Sections~\ref{ss: subcomplexes 1} and~\ref{ss: subcomplexes 2}), it follows that there are some $C_1,C_2>0$ and $C_3,C'_3\in\R$ such that
    \begin{multline*}
    	\fN^\pm_{s,\text{\rm max/min}}(\lambda)\\
		\begin{aligned}
      			&\le\#\left\{\,(k,\ell)\in\N^2\mid C_1k+C_2\,\tilde\lambda_{\text{\rm max/min},\ell}(k+1)^{-u}+C'_3\le\lambda\,\right\}\\
      			&\le\#\{\,(k,\ell)\in\N^2\mid C_1k+C_2C_0\ell^{\theta_0}(k+1)^{-u}+C_3\le\lambda\,\}\\
      			&\le\#\left\{\,(k,\ell)\in\N^2\;\Bigg|\;0\le\frac{\lambda-C_3}{C_1},\ 
			\ell\le\left(\frac{\lambda-C_3-C_1k}{C_2C_0}\right)^{\frac{1}{\theta_0}}(k+1)^{\frac{u}{\theta_0}}\,\right\}\;.
		\end{aligned}
    \end{multline*}
Consider the function
	\[
		f:\left[-1,a:=\frac{\lambda-C_3}{C_1}\right]\to[0,\infty)\;,\quad f(x)=\left(\frac{\lambda-C_3-C_1x}{C_2C_0}\right)^{\frac{1}{\theta_0}}(x+1)^{\frac{u}{\theta_0}}\;.
	\]
Elementary calculus shows that $f$ vanishes at $x=-1,a$, it reaches its maximum at
	\[
		x=b:=\frac{\lambda u-C_3u-C_1}{C_1(1+u)}\;,
	\]
and it is strictly increasing (respectively, decreasing) on $[-1,b]$ (respectively, $[b,a]$). It follows that\footnote{A similar argument is made in the proof of \cite[Corollary~12.13-(viii)]{AlvCalaza2017}. In that case, the authors use a strictly decreasing function $f:(-\infty,a]\to[0,\infty)$. The resulting estimate should be
	\[
		\fN^\pm_{s,\text{\rm max/min}}(\lambda)\le\int_0^af(x)\,dx+f(0)+a+1\;,
	\]
but the terms $f(0)+a+1$ were missing in that publication. This correction does not affect the final estimate of $\fN^\pm_{s,\text{\rm max/min}}(\lambda)$ obtained there.}
	\[
		\fN^\pm_{s,\text{\rm max/min}}(\lambda)\le\int_0^af(x)\,dx+2f(b)+a+1\;.
	\]
But
	\[
		f(b)=\left(\frac{\lambda-C_3+C_1}{(1+u)C_2C_0}\right)^{\frac{1}{\theta_0}}\left(\frac{u(\lambda-C_3+C_1)}{(1+u)C_1}\right)^{\frac{u}{\theta_0}}\;,
	\]
and
 \begin{align*}
      \int_0^af(x)\,dx&\le\left(\int_0^{\frac{\lambda-C_3}{C_1}}\left(\frac{\lambda-C_3-C_1x}{C_2C_0}\right)^{\frac{2}{\theta_0}}\,dx\right)^{\frac{1}{2}}\left(\int_0^{\frac{\lambda-C_3}{C_1}}(x+1)^{\frac{2u}{\theta_0}}\,dx\right)^{\frac{1}{2}}\\
      &\le\left(\frac{\theta_0(\lambda-C_3)^{\frac{2}{\theta_0}+1}}{(2+\theta_0)(C_2C_0)^{\frac{2}{\theta_0}}C_1}\right)^{\frac{1}{2}}\left(\frac{\theta_0(\lambda-C_3+C_1)^{\frac{2u}{\theta_0}+1}}{(2u+\theta_0)C_1^{\frac{2u}{\theta_0}+1}}\right)^{\frac{1}{2}}\\
      &=\frac{\theta_0(\lambda-C_3)^{\frac{1}{\theta_0}+\frac{1}{2}}(\lambda-C_3+C_1)^{\frac{u}{\theta_0}+\frac{1}{2}}}{(2+\theta_0)^{\frac{1}{2}}(2u+\theta_0)^{\frac{1}{2}}(C_2C_0)^{\frac{1}{\theta_0}}C_1^{1+\frac{u}{\theta_0}}}\;.
    \end{align*}
  So $\fN^\pm_{s,\text{\rm max/min}}(\lambda)\le C\lambda^{\frac{1+u}{\theta_0}+1}$ for some $C>0$ and all large enough $\lambda$, giving~\eqref{i: liminf_k lambda^pm_s,max/min,k k^-theta > 0} with $\theta=\frac{1+u}{\theta_0}+1$.
\end{proof}

Table~\ref{table: kappa : r, 1} describes the above conditions on $\kappa$ in terms of $r$.

\begin{table}[h]
\renewcommand{\arraystretch}{1.3}
\begin{tabular}{|c|c|}
\hline
$\kappa\ge\frac{1}{2}$ & $r\le\frac{n-1}{2}-\frac{1}{2u}$  \\
\hline
$|\kappa|<\frac{1}{2}$ & $|r-\frac{n-1}{2}|<\frac{1}{2u}$ \\
\hline
$\kappa\le-\frac{1}{2}$ & $r\ge\frac{n-1}{2}+\frac{1}{2u}$ \\
\hline
\end{tabular}
\vspace{1mm}
\caption{Correspondence between conditions on $\kappa$ and $r$}
\label{table: kappa : r, 1}
\end{table}

\section{Relatively local model of the Witten's perturbation}\label{s: rel-local model}

Let $m\in\N$, and let $L_1,\dots,L_a$ be compact stratifications. For each $i=1,\dots,a$, let $N_i$ be a dense stratum of $L_i$, let $k_i=\dim N_i+1$, and let $*_i$ and $\rho_i$ be the vertex and radial function of $c(L_i)$. Then $M:=\R^m\times\prod_{i=1}^a(N_i\times\R_+)$ is a dense stratum of $A:=\R^m\times\prod_{i=1}^ac(L_i)$. For any relatively compact open neighborhood $O$ of $x:=(0,*_1,\dots,*_a)$, all general adapted metrics on $M$ are quasi-isometric on $M\cap O$ to a metric of the form $g=g_0+\sum_{i=1}^a\rho_i^{2u_i}\tilde g_i+(d\rho_i)^2$, where $g_0$ is the Euclidean metric on $\R^m$, every $\tilde g_i$ is a general adapted metric on $N_i$, and $u_i>0$. Suppose that $g$ is good; i.e., the metrics $\tilde g_i$ are good, and $u_i\le1$. We can assume that every $N_i$ is connected, which means that the fiber of $\lim:\widehat M\to\overline M$ over $x$ consists of a unique point, which can be identified to $x$ (see \cite[Proof of Proposition~3.20]{AlvCalaza2017}). According to Section~\ref{ss: rel-Morse}, the rel-local model of a rel-Morse function around a rel-critical point is of the form $f=\frac{1}{2}(\rho_+^2-\rho_-^2)$, where $\rho_\pm$ is the radial function of $\R^{m_\pm}\times\prod_{i\in I_\pm}c(L_i)$, for some decomposition $m=m_++m_-$ ($m_\pm\in\N$), and some partition of $\{1,\dots,a\}$ into sets $I_\pm$. The rel-critical set of $f$ consists only of $x$. Let $d_s$, $\delta_s$, $D_s$ and $\D_s$ be the Witten's perturbations of $d$, $\delta$, $D$ and $\D$ on $\Omega(M)$ induced by $f$. Let $\HH_{s,\text{\rm max/min}}=\bigoplus_r\HH_{s,\text{\rm max/min}}^r=\ker\D_{s,\text{\rm max/min}}$, with the induced grading. The following result is a direct consequence of Corollary~\ref{c:d^pm_s,max/min are discrete} and \cite[Example~9.1 and Lemma 5.1]{AlvCalaza2017}, taking also into account Table~\ref{table: kappa : r, 1}.

\begin{cor}\label{c: rel-local model}
		\begin{enumerate}[{\rm(}i\/{\rm)}]
  
    			\item\label{i: d_s,max/min is discrete} $\D_{s,\text{\rm max/min}}$ has a discrete spectrum. 
    
    			\item\label{i: M_+ = N_+ times R_+ and M_- = N_- times R_+}  We have
      				\[
        					\HH_{s,\text{\rm max/min}}^r\cong\bigoplus_{(r_1,\dots,r_a)}
					\bigotimes_{i=1}^aH_{\text{\rm max/min}}^{r_i}(N_i)\;,
      				\]
	where $(r_1,\dots,r_a)$ runs in the subset of $\N^a$ defined by the conditions
				\begin{gather*}
					r=m_-+\sum_{i=1}^ar_i+|I_-|\;,\\
						\begin{alignedat}{2}
							&\left.
								\begin{array}{ll}
									r_i<\frac{k_i-1}{2}+\frac{1}{2u_i} & \text{if $i\in I_+$}\\[4pt]
									r_i\ge\frac{k_i-1}{2}+\frac{1}{2u_i} & \text{if $i\in I_-$}
								\end{array}
							\right\}&\quad&\text{for $\HH_{s,\text{\rm max}}^r$}\;,\\
							&\left.
								\begin{array}{ll}
									r_i\le\frac{k_i-1}{2}-\frac{1}{2u_i} & \text{if $i\in I_+$}\\[4pt]
									r_i>\frac{k_i-1}{2}-\frac{1}{2u_i} & \text{if $i\in I_-$}
								\end{array}
							\right\}&\quad&\text{for $\HH_{s,\text{\rm min}}^r$}\;.
						\end{alignedat}
				\end{gather*}
				
			\item\label{i: langle he_s, e_s rangle to 1} If $e_s\in\HH_{s,\text{\rm max/min}}$ with norm one for every $s$, and $h$ is a bounded measurable function on $\R_+$ with $h(\rho)\to1$ as $\rho\to0$, then $\langle he_s,e_s\rangle\to1$ as $s\to\infty$.
    
    \item\label{i: lambda_s,max/min,k in O(s)} Let $0\le\lambda_{s,\text{\rm max/min},0}\le\lambda_{s,\text{\rm max/min},1}\le\cdots$ be the eigenvalues of $\D_{s,\text{\rm max/min}}$, repeated according to their multiplicities. Given $k\in\N$, if $\lambda_{s,\text{\rm max/min},k}>0$ for some $s$, then $\lambda_{s,\text{\rm max/min},k}>0$ for all $s$ and $\lambda_{s,\text{\rm max/min},k}\in O(s)$ as $s\to\infty$.
  
    \item\label{i: liminf_k lambda_s,max/min,k k^-theta > 0} There is some $\theta>0$ such that $\liminf_k\lambda_{s,\text{\rm max/min},k}\,k^{-\theta}>0$.
  
  \end{enumerate}
\end{cor}

For every $\rho>0$, let $B_\rho$ be the open ball of center $0$ and radius $\rho$ in $\R^m$, and let
  	\[
    		U_{x,\rho}=B_\rho\times\prod_{i=1}^a(N_i\times(0,\rho))\subset M\;.
  	\]

Taking complex coefficients, by Propositions~\ref{p: EE_gamma,i},~\ref{p: FF_alpha,beta} and~\ref{p: splitting}, the following result clearly boils down to the case of Proposition~\ref{p: wave, simple}.

\begin{prop}\label{p: wave, model}
  For $\alpha\in L^2\Omega(M)$, let $\alpha_t=\exp(itD_{s,\text{\rm max/min}})\alpha$. If $\supp\alpha\subset\ol{U_{x,a}}$ for some $a>0$, then $\supp\alpha_t\subset\ol{U_{x,a+|t|}}$ for all $t\in\R$.
\end{prop}

\section{Proof of Theorem~\ref{t:  spectrum of Delta_max/min}}\label{s: proof of thm 1.1}

This theorem follows from Corollary~\ref{c: rel-local model}~\eqref{i: d_s,max/min is discrete},\eqref{i: liminf_k lambda_s,max/min,k k^-theta > 0} with the same arguments as \cite[Theorem~1.1]{AlvCalaza2017}. More precisely, \cite[Propositions~14.2 and~14.3]{AlvCalaza2017} are used to globalize the properties of the rel-local model, the min-max principle (see e.g.\ \cite[Theorem~XIII.1]{ReedSimon1978}) is used to show that the properties of the statement are invariant by taking Witten's perturbation defined by rel-admissible functions, and Remark~\ref{r: rel-admissible, rel-critical, rel-non-degenerate}~\eqref{i: ex of rel-admissible function},\eqref{i: lambda_a} is used to produce rel-admissible cutoff functions and partitions of unity with bounded differential. These functions are needed for the Witten's perturbation and to apply \cite[Propositions~14.2 and~14.3]{AlvCalaza2017}.

\section{Functional calculus}\label{s: functional calculus}

Let $M$ be a stratum of a compact stratification, equipped with a good general adapted metric $g$. Let $f$ be any rel-admissible function on $M$, and let $d_s$, $\delta_s$, $D_s$ and $\D_s$ be the corresponding Witten's perturbations of $d$, $\delta$, $D$ and $\D$. Since $f$ is rel-admissible, for every $s$, $\D_s-\D$ is a homomorphism with uniformly bounded norm by~\eqref{Delta_s with Hess f and df}. From~\eqref{Delta_s with Hess f and df} and the min-max principle (see e.g.\ \cite[Theorem~XIII.1]{ReedSimon1978}), it also follows that $\sD(\D_{s,\text{\rm max/min}})=\sD(\D_{\text{\rm max/min}})$, $\sD^\infty(\D_{s,\text{\rm max/min}})=\sD^\infty(\D_{\text{\rm max/min}})$, and that the properties stated in Theorem~\ref{t:  spectrum of Delta_max/min} can be extended to the perturbation $\D_{s,\text{\rm max/min}}$. 

For any rapidly decaying function $\phi$ on $\R$, $\phi(\D_{s,\text{\rm max/min}})$ is a Hilbert-Schmidt operator on $L^2\Omega(M)$ by the version of Theorem~\ref{t:  spectrum of Delta_max/min}~\eqref{i: liminf_k lambda_max/min,k k^-theta > 0 for some theta>0} for $\D_{s,\text{\rm max/min}}$. In fact, $\phi(\D_{s,\text{\rm max/min}})$ is a trace class operator because $\phi$ can be given as the product of two rapidly decaying functions, $|\phi|^{1/2}$ and $\sign(\phi)\,|\phi|^{1/2}$, where $\sign(\phi)(x)=\sign(\phi(x))\in\{\pm1\}$ if $\phi(x)\ne0$.

Like in the case of closed manifolds (see e.g.\ \cite[Chapters~5 and~8]{Roe1998}), $\phi(\D_{s,\text{\rm max/min}})$ is given by a Schwartz kernel $K_s$, and $\Tr\phi(\D_{s,\text{\rm max/min}})$ equals the integral of the pointwise trace of $K_s$ on the diagonal. But we do not know whether $K_s$ is uniformly bounded because a ``rel-Sobolev embedding theorem'' is missing  \cite[Section~19]{AlvCalaza2017}. Theorem~\ref{t:  spectrum of Delta_max/min}~\eqref{i: liminf_k lambda_max/min,k k^-theta > 0 for some theta>0} becomes important in our arguments to make up for this lack.

\section{The wave operator}\label{s: wave}

With the notation of Section~\ref{s: functional calculus}, suppose that $f$ is a rel-Morse function. Take a general chart $O\equiv O'$ around every $x\in\Crit_{\text{\rm rel}}(f)$, like in Section~\ref{ss: rel-Morse}. Let us add the subindex ``$x$''  to the notation of $M'$, $N_i$, $m_\pm$ and $I_\pm$ in this case. Take a good adapted metric $g'_x$ on $M'_x$ of the form used in Section~\ref{s: rel-local model}. Consider the Witten's perturbed operators $d'_{x,s}$, $\delta'_{x,s}$, $D'_{x,s}$ and $\D'_{x,s}$ on $\Omega(M'_x)$ defined by the function $f':=\frac{1}{2}(\rho_+^2-\rho_-^2)$ (a prime and the subindex $x$ is added to their notation). Add also a prime to the notation of the sets $U_{x,\rho}$ of Section~\ref{s: rel-local model}, considered in $M'_x$. Let $\rho_0>0$ such that $\ol{U'_{x,\rho_0}}\subset O'$.  Then, for $0<\rho\le\rho_0$, there is some open $U_{x,\rho}\subset M$ so that $U_{x,\rho}\equiv U'_{x,\rho}$. Moreover, according to Remark~\ref{r: sum_a lambda_a g_a}, we can assume $g|_{U_{x,\rho_0}}\equiv g'_x|_{U'_{x,\rho_0}}$. 

Consider the wave equation
  	\begin{equation}\label{wave}
    		\frac{d\alpha_t}{dt}-iD_s\alpha_t=0\;,
  		\end{equation}
where $\alpha_t\in\Omega(M)$ depends smoothly on $t$. Given any $\alpha\in\sD^\infty(\D_{s,\text{\rm max/min}})$, its solution with the initial condition $\alpha_0=\alpha$ is given by $\alpha_t=\exp(itD_{s,\text{\rm max/min}})\alpha$. Moreover a usual energy estimate shows that such a solution is unique (see e.g.\ \cite[Proposition~7.4]{Roe1998}); in fact, given any $c>0$, it is also unique for $|t|\le c$.  

\begin{prop}\label{p: finite propagation speed towards/from the critical points}
  Let $0<a<b<\rho_0$ and $\alpha\in L^2\Omega(M)$. The following properties hold for $\alpha_t=\exp(itD_{s,\text{\rm max/min}})\alpha$:
    \begin{enumerate}[{\rm(}i\/{\rm)}]
    
      \item\label{i: supp alpha_t subset M sm U_x,a-|t|} If $\supp\alpha\subset M\sm U_{x,a}$, then $\supp\alpha_t\subset M\sm U_{x,a-|t|}$ for $0<|t|\le a$.
      
      \item\label{i: supp alpha_t subset ol U_x,a+|t|} If $\supp\alpha\subset\ol{U_{x,a}}$, then $\supp\alpha_t\subset\ol{U_{x,a+|t|}}$ for $0<|t|\le b-a$.
    
    \end{enumerate}
\end{prop}

\begin{proof}
  	First, let us prove~\eqref{i: supp alpha_t subset ol U_x,a+|t|}. We can assume that $\alpha\in\sD^\infty(\D_{s,\text{\rm max/min}})$ because $\exp(itD_{s,\text{\rm max/min}})$ is bounded. Since $\supp\alpha\subset\ol{U_{x,a}}$, we have $\alpha|_{U_{x,\rho_0}}\equiv\alpha'|_{U'_{x,\rho_0}}$ for a unique $\alpha'\in\Omega(M'_x)$ supported in $\ol{U'_{x,a}}$. We get $\alpha'\in\sD^\infty(\D'_{x,s,\text{\rm max/min}})$ because $\alpha\in\sD^\infty(\D_{s,\text{\rm max/min}})$. Let $\alpha'_t=\exp(itD'_{x,s,\text{\rm max/min}})\alpha'$. By Proposition~\ref{p: wave, model}, we have $\supp\alpha'_t\subset\ol{U'_{x,a+|t|}}$ for $0<|t|\le b-a$. Then $\alpha'_t|_{U'_{x,\rho_0}}\equiv\beta_t|_{U_{x,\rho_0}}$ for a unique $\beta_t\in\Omega(M)$ supported in $\ol{U_{x,a+|t|}}$. Now, $\beta_t\in\sD^\infty(\D_{s,\text{\rm max/min}})$ because $\alpha'_t\in\sD^\infty(\D'_{x,s,\text{\rm max/min}})$. Moreover $\beta_t$ satisfies~\eqref{wave} for $|t|\le b-a$ with initial condition $\beta_0=\alpha$. So $\beta_t=\alpha_t$ by the uniqueness of the solution of~\eqref{wave}, obtaining $\supp\alpha_t\subset\ol{U_{x,a+|t|}}$.
	
	Finally,~\eqref{i: supp alpha_t subset M sm U_x,a-|t|} follows from~\eqref{i: supp alpha_t subset ol U_x,a+|t|} in the following way. For any $\beta\in\Omega_0(M)$ with $\supp\beta\subset\ol{U_{x,a-|t|}}$, let $\beta_\tau=\exp(i\tau D_{s,\text{\rm max/min}})\beta$ for $\tau\in\R$. By~\eqref{i: supp alpha_t subset ol U_x,a+|t|}, we get $\supp\beta_{-t}\subset\ol{U_{x,a}}$, and therefore $\langle \alpha_t,\beta\rangle=\langle\alpha,\beta_{-t}\rangle=0$. This shows that $\supp\alpha_t\subset M\sm U_{x,a-|t|}$.
\end{proof}

\begin{rem}
	The steps given to achieve Proposition~\ref{p: finite propagation speed towards/from the critical points} are simpler here than in \cite{AlvCalaza2017}. In fact, it would be difficult to adapt the arguments of \cite{AlvCalaza2017} since an expression of $\sD^\infty(\D_{\text{\rm max/min}})$ is missing in Section~\ref{sss: complex 2, max/min i.b.c.}.
\end{rem}

\section{Proof of Theorem~\ref{t:  Morse inequalities}}

This theorem now follows like \cite[Theorem 1.2]{AlvCalaza2017}. Thus the details are omitted. 

Consider the notation of Section~\ref{s: wave}. By~\eqref{d_s}, the numbers $\beta_{\text{\rm max/min}}^r$ are also given by the cohomology of $d_{s,\text{\rm max/min}}=d_{\text{\rm max/min}}+s\,df\wedge$ on $\sD(d_{s,\text{\rm max/min}})=e^{-sf}\,\sD(d_{\text{\rm max/min}})=\sD(d_{\text{\rm max/min}})$.

Let $\phi$ be a smooth rapidly decaying function on $\R$ with $\phi(0)=1$. Then $\phi(\D_{s,\text{\rm max/min}})$ is of trace class (Section~\ref{s: functional calculus}), and let $\mu_{s,\text{\rm max/min}}^r=\Tr(\phi(\D_{s,\text{\rm max/min},r}))$. Then the following result follows formally like \cite[Proposition~14.3]{Roe1998}.

\begin{prop}\label{p: analytic Morse inequalities}
  	We have
    		\begin{align*}
      			\sum_{r=0}^k(-1)^{k-r}\beta_{\text{\rm max/min}}^r
			&\le\sum_{r=0}^k(-1)^{k-r}\mu_{s,\text{\rm max/min}}^r\quad(0\le k<n)\;,\\
      			\chi_{\text{\rm max/min}}&=\sum_{r=0}^n(-1)^r\,\mu_{s,\text{\rm max/min}}^r\;.
    		\end{align*}
\end{prop}

For $\rho\le\rho_0$, let $U_\rho=\bigcup_xU_{x,\rho}$, with $x$ running in $\Crit_{\text{\rm rel}}(f)$. Fix some $\rho_1>0$ such that $4\rho_1<\rho_0$. Let $\fG$ and $\fH$ be the Hilbert subspaces of $L^2\Omega(M)$ consisting of forms essentially supported in $M\sm U_{\rho_1}$ and $M\sm U_{3\rho_1}$, respectively. Since
  	\[
		\D_{s,\text{\rm max/min}}=\D_{\text{\rm max/min}}+s\,\boldsymbol{\Hess}f+s^2\,|df|^2
  	\]
on $\sD(\D_{s,\text{\rm max/min}})=\sD(\D_{\text{\rm max/min}})$ for all $s\ge0$ by~\eqref{Delta_s with Hess f and df}, it follows that there is some $C>0$ so that, if $s$ is large enough,
  \begin{equation}\label{Delta_s,max/min ge Delta_max/min+Cs^2}
    \D_{s,\text{\rm max/min}}\ge\D_{\text{\rm max/min}}+Cs^2
    \quad\text{on}\quad\fG\cap\sD(\D_{\text{\rm max/min}})\;.
  \end{equation}
Let $h$ be a rel-admissible function on $M$ such that $h\ge0$, $h\equiv1$ on $U_{\rho_1}$ and $h\equiv0$ on $M\sm U_{2\rho_1}$ (Remark~\ref{r: rel-admissible, rel-critical, rel-non-degenerate}~\eqref{i: ex of rel-admissible function}). Then $T_{s,\text{\rm max/min}}=\D_{s,\text{\rm max/min}}+hCs^2$, with domain $\sD(\D_{\text{\rm max/min}})$, is self-adjoint in $L^2\Omega(M)$ with a discrete spectrum. Moreover
  \begin{equation}\label{T_s,max/min ge Delta_max/min + Cs^2}
    T_{s,\text{\rm max/min}}\ge\D_{\text{\rm max/min}}+Cs^2
  \end{equation}
for $s$ large enough by~\eqref{Delta_s,max/min ge Delta_max/min+Cs^2}.

Take some $\phi\in\SS_{\text{\rm ev}}$ such that $\phi\ge0$, $\phi(0)=1$, $\supp\hat\phi\subset[-\rho_1,\rho_1]$, and $\phi|_{[0,\infty)}$ is monotone \cite[Section~18.2]{AlvCalaza2017}, where $\hat\phi$ denotes its Fourier transform. Write $\phi(x)=\psi(x^2)$ for some $\psi\in\SS$. Using Proposition~\ref{p: finite propagation speed towards/from the critical points}~\eqref{i: supp alpha_t subset M sm U_x,a-|t|}, the argument of the first part of the proof of \cite[Lemma~14.6]{Roe1998} gives the following. 

\begin{lem}\label{l:psi(Delta_s,max/min)}
  $\psi(\D_{s,\text{\rm max/min}})=\psi(T_{s,\text{\rm max/min}})$ on $\fH$.
\end{lem}

Let $\Pi:L^2\Omega(M)\to\fH$ denote the orthogonal projection. According to Section~\ref{s: functional calculus}, $\psi(\D_{s,\text{\rm max/min}})$ is of trace class for all $s\ge0$. Then the self-adjoint operator $\Pi\,\psi(\D_{s,\text{\rm max/min}})\,\Pi$ is also of trace class (see e.g.\ \cite[Proposition~8.8]{Roe1998}).

\begin{lem}\label{l:Tr ... to0}
  $\Tr(\Pi\,\psi(\D_{s,\text{\rm max/min}})\,\Pi)\to0$ as $s\to\infty$.
\end{lem}

\begin{proof}
  	This follows like \cite[Lemma~18.3]{AlvCalaza2017}, using~\eqref{T_s,max/min ge Delta_max/min + Cs^2}, the min-max principle and Lemma~\ref{l:psi(Delta_s,max/min)}, and expressing the trace as sum of eigenvalues.
\end{proof}

The following is a direct consequence of Corollary~\ref{c: rel-local model}~\eqref{i: d^pm_s,max/min is discrete}--\eqref{i: lambda^pm_s,max/min,k in O(s)}.

\begin{cor}\label{c:lim_s to infty Tr(h(rho) phi(Delta_x,s,max/min,r))}
  If $h$ is a bounded measurable function on $\R_+$ such that $h(\rho)\to1$ as $\rho\to0$, then $\Tr(h(\rho)\,\phi(\D'_{x,s,\text{\rm max/min},r}))\to\nu_{x,\text{\rm max/min}}^r$ as $s\to\infty$.
\end{cor}

For every $x\in\Crit_{\text{\rm rel}}(f)$, let $\widetilde{\fH}_x\subset L^2\Omega(M)$ and $\widetilde{\fH}'_x\subset L^2\Omega(M'_x)$ be the Hilbert subspaces of differential forms supported in $\ol{U_{x,3\rho_1}}$ and $\ol{U'_{x,3\rho_1}}$, respectively. We have $\widetilde{\fH}_x\equiv\widetilde{\fH}'_x$ because $g\equiv g'_x$ on $U_{x,\rho_0}\equiv U'_{x,\rho_0}$. Moreover $\D_s\equiv\D'_{x,s}$ on differential forms supported in $U_{x,\rho_0}\equiv U'_{x,\rho_0}$. By using Proposition~\ref{p: finite propagation speed towards/from the critical points}~\eqref{i: supp alpha_t subset ol U_x,a+|t|}, the argument of the first part of the proof of \cite[Lemma~14.6]{Roe1998} can be adapted to show the following. 

\begin{lem}\label{l:phi(Delta_s,max/min)}
  $\phi(\D_{s,\text{\rm max/min}})\equiv\phi(\D'_{x,s,\text{\rm max/min}})$ on $\widetilde{\fH}_x\equiv\widetilde{\fH}'_x$ for all $x\in\Crit_{\text{\rm rel}}(f)$.
\end{lem}

For every $x\in\Crit_{\text{\rm rel}}(f)$, let $\widetilde{\Pi}_x:L^2\Omega(M)\to\widetilde{\fH}_x$ and $\widetilde{\Pi}'_x:L^2\Omega(M'_x)\to\widetilde{\fH}'_x$ denote the orthogonal projections. Since the subspaces $\widetilde{\fH}_x$ are orthogonal to each other, $\widetilde{\Pi}:=\sum_x\widetilde{\Pi}_x:L^2\Omega(M)\to\widetilde{\fH}:=\sum_x\widetilde{\fH}_x$ is the orthogonal projection.

\begin{lem}\label{l:Tr(Pi phi(Delta_s,max/min,r) Pi)=nu_max/min^r}
  $\Tr(\widetilde{\Pi}\,\phi(\D_{s,\text{\rm max/min},r})\,\widetilde{\Pi})\to\nu_{\text{\rm max/min}}^r$ as $s\to\infty$.
\end{lem}

\begin{proof}
  	This follows like \cite[Lemma~18.3]{AlvCalaza2017}, using Corollary~\ref{c:lim_s to infty Tr(h(rho) phi(Delta_x,s,max/min,r))} and Lemma~\ref{l:phi(Delta_s,max/min)}, and, for all $x\in\Crit_{\text{\rm rel}}(f)$, considering $\widetilde\Pi'_x$ as the multiplication operator by the characteristic function of $U'_{x,3\rho_1}$.
\end{proof}

Since $\Pi+\widetilde{\Pi}=1$, Theorem~\ref{t:  Morse inequalities} follows from Proposition~\ref{p: analytic Morse inequalities}, and Lemmas~\ref{l:Tr ... to0} and~\ref{l:Tr(Pi phi(Delta_s,max/min,r) Pi)=nu_max/min^r}.

\bibliographystyle{amsplain}


\providecommand{\bysame}{\leavevmode\hbox to3em{\hrulefill}\thinspace}
\providecommand{\MR}{\relax\ifhmode\unskip\space\fi MR }
\providecommand{\MRhref}[2]{%
  \href{http://www.ams.org/mathscinet-getitem?mr=#1}{#2}
}
\providecommand{\href}[2]{#2}

\end{document}